\documentclass[11pt,a4paper,leqno]{article}
\usepackage{a4wide}
\usepackage{color}

\usepackage{latexsym}
\usepackage{graphicx}
\usepackage{amsmath}
\usepackage{amssymb}
\newtheorem{defin}{Definition}
\newtheorem{lemma}{Lemma}
\newtheorem{prop}{Proposition}
\newtheorem{theo}{Theorem}

\newenvironment{proof}{\medskip\par\noindent{\bf Proof}}{\hfill $\Box$
\medskip\par}
\begin{document}
\title{On multiscale Gevrey and $q-$Gevrey asymptotics for some linear $q-$difference differential
initial value Cauchy problems}

\author{{\bf A. Lastra, S. Malek}\\
University of Alcal\'{a}, Departamento de F\'{i}sica y Matem\'{a}ticas,\\
Ap. de Correos 20, E-28871 Alcal\'{a} de Henares (Madrid), Spain,\\
University of Lille 1, Laboratoire Paul Painlev\'e,\\
59655 Villeneuve d'Ascq cedex, France,\\
{\tt alberto.lastra@uah.es}\\
{\tt Stephane.Malek@math.univ-lille1.fr }}
\date{\today}

\maketitle

\thispagestyle{empty}
{ \small \begin{center}
{\bf Abstract}
\end{center}

We study the asymptotic behavior of the solutions related to a singularly perturbed q-difference-differential problem in the complex domain. The analytic solution can be splitted according to the nature of the equation and its geometry so that both, Gevrey and $q-$Gevrey asymptotic phenomena are observed and can be distinguished, relating the analytic and the formal solution. 

The proof leans on a two level novel version of Ramis-Sibuya theorem under Gevrey and q-Gevrey orders.  

\noindent Key words: asymptotic expansion, Borel-Laplace transform, Fourier transform, Cauchy problem, formal power series,
nonlinear integro-differential equation, nonlinear partial differential equation, singular perturbation. 2010 MSC: 35C10, 35C20.}
\bigskip \bigskip

\section{Introduction}

The present work deals with a family of q-difference-differential equations of the form
\begin{equation}\label{e1}
P(t,z,\epsilon,\sigma_{q^{\delta}},\partial_t,\partial_z)y(t,z,\epsilon)=f(t,z,\epsilon),
\end{equation}
where $P(t,z,\epsilon,Z_1,Z_2,Z_3)$ is a polynomial in the first, and the last three variables, with coefficients depending holomorphically on $z$ and the perturbation parameter on some domain. More precisely, $P$ can be splitted in the following way:
$$P(t,z,\epsilon,\sigma_{q^{\delta}},\partial_t,\partial_z)=P_1(t,z,\epsilon,\sigma_{q^{\delta}},\partial_t,\partial_z)\mathbf{c}_{\mathbf{F}}(\epsilon)^{-1}P_2(t,z,\epsilon,\partial_t,\partial_z),$$
with
$$P_1(t,z,\epsilon,\sigma_{q^{\delta}},\partial_t,\partial_z)=Q(\partial_{z})-R_{D}(\partial_{z})\epsilon^{kd_{D}}(t^{k+1}\partial_{t})^{d_{D}}- \sum_{l=1}^{D-1} \epsilon^{\Delta_{l}}(t^{k+1}\partial_{t})^{d_{l}}c_{l}(z,\epsilon)R_{l}(\partial_{z})\sigma_{q^{\delta}},$$
and 
$$P_2\left(t,z,\epsilon,\partial_t,\partial_z\right)=\mathbf{Q}(\partial_z)\partial_{t}-\sum_{l=1}^{\mathbf{D}}\epsilon^{\boldsymbol{\Delta}_l}t^{\mathbf{d}_l}\partial_t^{\boldsymbol{\delta}_l}\mathbf{R}_l(\partial_z)-\mathbf{c}_{0,0}(\epsilon)\mathbf{c}_{0}(z,\epsilon)
\mathbf{R}_{0}(\partial_z).$$
Here, $\delta, q$ are positive real numbers with $q>1$, and $\sigma_{q^{\delta}}$ stands for the dilation operator on $t$ variable $\sigma_{q^{\delta}}(H)(t)=H(q^{\delta }t)$. $D,\mathbf{D}\ge2$ are fixed integers. The elements $Q,\mathbf{Q}.R_l,\mathbf{R}_h$  for every $1\le l\le D$ and $0\le h\le \mathbf{D}$ are polynomials with complex coefficients, and $k\ge 1$ is an integer number. We consider nonnegative integers $d_l$ for every $1\le l\le D$ and $\mathbf{d}_l,\boldsymbol{\delta}_l,\boldsymbol{\Delta}_l$ for every $1\le l\le D$, and $\Delta_l$ for $1\le l\le D-1$. 

The functions $\mathbf{c}_{\mathbf{F}}(\epsilon)$, $\mathbf{c}_{0,0}(\epsilon)$ turn out to be holomorphic functions in a neighborhood of the origin in the perturbation parameter, and $\mathbf{c}_0(z,\epsilon)$ and $c_l(z,\epsilon)$ for $1\le l\le D-1$ are holomorphic functions defined in a horizontal strip in $z$ variable, and a neighborhood of the origin with respect to the perturbation parameter, $\epsilon$. 

The forcing term $f(t,z,\epsilon)$ is chosen as a holomorphic function with respect to $(t,\epsilon)$ in a product of discs centered at the origin, and in a horizontal strip with respect to $z$. The precised requirements on the elements involved in the equation are accurately described in the following sections.

This problem embraces two different theories on the asymptotic behavior of the solutions of equations, recently studied by several authors: Gevrey and $q-$Gevrey expansions. More precisely, Gevrey asymptotics come out from the geometric disposal of the singularities coming from the operator $P_2$, whilst $q-$Gevrey asymptotics are due to the appearance of moving singularities usually related to dilation operators, such as the one involved in the operator $P_1$. In order to study equation (\ref{e1}) we regard two auxiliary problems. On the one hand, we consider the problem
\begin{equation}\label{e2}
P_1(t,z,\epsilon,\sigma_{q^{\delta}},\partial_t,\partial_z)u(t,z,\epsilon)=f(t,z,\epsilon),
\end{equation}
and then, the solution of the previous problem is set as the forcing term of a second auxiliary problem 
\begin{equation}\label{e3}
P_2\left(t,z,\epsilon,\partial_t,\partial_z\right)y(t,z,\epsilon)=\mathbf{c}_{\mathbf{F}}(\epsilon)u(t,z,\epsilon)
\end{equation}
to be studied. As a whole, the solution $y(t,z,\epsilon)$ would provide a solution of the main problem (\ref{e1}). As a matter of fact, two quite different pictures emerge. 

Equation (\ref{e3}) falls in the framework of the asymptotic study of singularly perturbed problems which have been studied in the case of a real perturbation parameter $\epsilon$ and $P_1$ acts on $\mathcal{C}^{\infty}(\mathbb{R}^d)$ functions or Sovolev spaces $H(\mathbb{R}^d)$. We refer to~\cite{kada} for the details. The case of a complex perturbation parameter $\epsilon$ has led to results in which the nature of the singularities arising from the problem describe different types of singularities. We refer to the work by M. Canalis-Durand, J. Mozo-Fern\'andez and R. Sch\"afke~\cite{camo}, the second author~\cite{ma12}, the authors~\cite{lamaade16} and the authors and J. Sanz~\cite{lamasa1}. In~\cite{lama0}, the authors study the parametric multi-level Gevrey solutions coming from a splitting of the equation which generates two Gevrey levels. In contrast to~\cite{lama0}, the splitting of the main equation in this work entails the necessity of a very accurate estimation of the situation of the moving singularities coming from the auxiliary equation (\ref{e2}).

The study of equation (\ref{e2}) is connected to the behavior of singularities appearing in the asymptotic study of advanced/delayed partial differential equations (see~\cite{lamasa2}, and the references therein). The study of this auxiliary equation continues a series of works dedicated to the asymptotic behavior of holomorphic solutions to different q-difference-differential problem involving irregular singularities such as~\cite{lamaq,lama2,lamaq2}. Another approach to handle this kind of equations is based on Newton polygon methods, see~\cite{taya2}, and also in the study of $q-$analogs of Briot-Bouquet type partial differential equations, see~\cite{ya}.

As a matter of fact, the appearance of different types of singularities, coming from small divisor problems in (\ref{e2}) and (\ref{e3}) makes the problem much more difficult to handle rather than solving both equations independently. The procedure followed rests on a meticulous tracing of the decreasing rate to zero of the singularities related to the small divison problem derived from (\ref{e2}). This would allow to compute the solution of the auxiliar problem (see Theorem~\ref{teo1}).

The plan of the work is as follows. 

In the first section, we state the definition on some Banach spaces of functions with some $q-$exponential growth with respect to some variable $\tau$, and exponential decay with respect to some other variable $m$ (see Definition~\ref{defi2}) and we provide some results to estimate the norm of certain continuous operators defined on such spaces. Afterwards, we introduce another auxiliary Banach spaces of functions with exponential growth of finite order, and decay of order 1 (see Definition~\ref{defi3}), and recall the action of certain operators on such spaces.

Departing from a good covering in $\mathbb{C}^{\star}$, $(\mathcal{E}_{p})_{0\le p\le \varsigma-1}$, (see Definition~\ref{goodcovering}), we construct the solution of the problem
\begin{equation}\label{e2p}
P_1(t,z,\epsilon,\sigma_{q^{\delta}},\partial_t,\partial_z)u_p(t,z,\epsilon)=f(t,z,\epsilon),
\end{equation}
for every $0\le p\le \varsigma-1$, in the form of the Fourier and $m_k$-Laplace transform (see Definition~\ref{defi4}) of a function obtained by a fixed point argument in the first Banach space of functions with $q-$exponential growth and exponential decay. Indeed, we have
$$u_{p}(t,z,\epsilon) = \frac{k}{(2\pi)^{1/2}} \int_{-\infty}^{+\infty} \int_{L_{\gamma_{p}}}
w_{k}^{p}(u,m,\epsilon) \exp( -(\frac{u}{\epsilon t})^{k} ) e^{izm} \frac{du}{u} dm,$$
where $L_{\gamma_p}$ stands for an appropriate direction, and there exists $C_1>0$ such that
$$|w_{k}^{p}(u,m,\epsilon)|\le C_1 \exp(-\beta|m|)\frac{1}{(1+|m|)^{\mu}}|\tau|\exp\left(\frac{k_1}{2}\frac{\log^2(|\tau|+\tau_0)}{\log(q)}+\alpha\log(|\tau|+\tau_0)\right),$$
for every $m\in\mathbb{R}$, $\epsilon\in\mathcal{E}_{p}$ and $u$ in an infinite sectorial triangle frame (see Figure~\ref{figure2}). The direction of integration $\gamma_p$ is chosen in such a way that $L_{\gamma_p}$ belongs to the frame. Consequently, $u_p(t,z,\epsilon)$ turns out to be a holomorphic function defined in $\mathcal{T}\times H_{\beta'}\times\mathcal{E}_p$, for some finite sector $\mathcal{T}$, and a horizontal strip $H_{\beta'}$. At this point, it is crucial to split the integration path in an infinite number of pieces in order to attain $q-$exponential flatness with respect to $\epsilon$, and uniformly in the other variables, of the difference of two consecutive solutions (in the sense that these solutions correspond to consecutive sectors in the perturbation parameter) (see (\ref{q_exp_small_difference_u_p})). The geometry of the problem is described in detail in Section~\ref{seccion42} to clarify this essential and delicate point. A q-Gevrey version of Ramis-Sibuya theorem, Theorem (q-RS), guarantees the existence of a formal solution $\hat{u}(t,z,\epsilon)$ of (\ref{e2p}) which asymptotically represents $u_p(t,z,\epsilon)$ for every $0\le p\le \varsigma-1$. 

For every $0\le p\le \varsigma-1$, we choose a family of sectors $(\mathcal{E}_{p',p})_{0\le p'\le \chi_p}$ with $\mathcal{E}_{p',p}\subseteq\mathcal{E}_{p}$ for every $0\le p'\le \chi_p-1$ and such that $(\mathcal{E}_{p',p})_{\substack{0\le p\le \varsigma-1\\0\le p'\le \chi_p-1}}$ is a good covering in $\mathbb{C}^{\star}$, and consider the problem
\begin{equation}\label{e3pp}
P_2\left(t,z,\epsilon,\partial_t,\partial_z\right)y_{p',p}(t,z,\epsilon)=\mathbf{c}_{\mathbf{F}}(\epsilon)u_p(t,z,\epsilon)
\end{equation}
for every $0\le p\le \varsigma-1$ and $0\le p'\le \chi_p$. The previous problem is solved by means of a fixed point result leading to the solution
$$y_{p',p}(t,z,\epsilon)=\frac{k}{(2\pi)^{1/2}}\int_{-\infty}^{\infty}\int_{L_{\gamma_{p',p}}}v_{k}^{\mathfrak{e}_{p',p},\mathfrak{d}_p}(u,m,\epsilon)e^{-\left(\frac{u}{t\epsilon}\right)^{k}}\frac{du}{u}\exp(izm)dm,$$
for some direction of integration $L_{\gamma_{p',p}}$, and where $v_{k}^{\mathfrak{e}_{p',p},\mathfrak{d}_p}$ satisfies the next estimates: there exists $C_2>0$ with
$$|v_{k}^{\mathfrak{e}_{p',p},\mathfrak{d}_p}(\tau,m,\epsilon)|\le C_2\frac{1}{(1+|m|)^{\mu}}\frac{|\tau|}{1+|\tau|^{2k}}\exp(-\beta|m|+\nu|\tau|^k),$$
for all $\tau\in S_{\mathfrak{e}_{p',p}}$, which is an infinite sector containing direction $L_{\gamma_{p',p}}$, all $m\in\mathbb{R}$ and all $\epsilon\in\mathcal{E}_{p',p}$. The funtion $y_{p',p}(t,z,\epsilon)$ is holomorphic in a finite sector with vertex at the origin with respect to $t$ variable, an horizontal strip with respect to $z$ variable, and in $\mathcal{E}_{p',p}$ with respect to $\epsilon$. In addition to that, the deformation of the integration paths derives to accurate estimations of the difference of two consecutive solutions (in the sense of above). Again, the geometry of the problem has to be inspected to attain accurate bounds on this difference along infinite segments in which the paths have to be splitted. It is worth mentioning that the difference of two consecutive solutions would become $q-$exponentially flat (resp. exponentially flat) whenever the consecutive sectors of the good covering come from different (resp. come from a common) $p\in\{0,1,\ldots,\varsigma-1\}$ (see Theorem~\ref{teo3}). This would suggest the application of a novel version of a multilevel Gevrey and $q-$Gevrey Ramis-Sibuya theorem (see Theorem~\ref{teo4}) in order to conclude the existence of a formal solution $\hat{y}(t,z,\epsilon)$ of the second auxiliary problem (\ref{e3pp}) which represents each of the analytic solutions $y_{p',p}(t,z,\epsilon)$, seen as a function of the perturbation parameter with coefficients in some Banach space. The formal and the analytic solution are decomposed accordingly to the multilevel asymptotic representation, in the shape of that of multisummability described in~\cite{ba2}. 

More precisely, the main result of the present work (Theorem~\ref{teopral}) states that equation
\begin{equation}\label{e0}P(t,z,\epsilon,\sigma_{q^{\delta}},\partial_t,\partial_z)y_{p',p}(t,z,\epsilon)=f(t,z,\epsilon),
\end{equation}
admits an analytic solution $y_{p',p}(t,z,\epsilon)$, holomorphic on the product of a finite sector with vertex at the origin, times a horizontal strip times $\mathcal{E}_{p',p}$, for every $0\le p\le \varsigma-1$ and all $0\le p'\le \chi_p-1$, with $y_{p',p}(0,z,\epsilon)\equiv 0$. In addition to that, (\ref{e0}) admits a unique formal solution $\hat{y}(t,z,\epsilon)$, written as a formal power series in the perturbation parameter, and with coefficients in an appropriate Banach space, $\mathbb{F}$, in such a way that $y_{p',p}(t,z,\epsilon)$ (resp. $\hat{y}(t,z,\epsilon)$) can be written in the form 
$$y_{p',p}(t,z,\epsilon)=a(t,z,\epsilon)+y^{1}_{p',p}(t,z,\epsilon)+y^{2}_{p',p}(t,z,\epsilon),$$
(resp. 
$$\hat{y}(t,z,\epsilon)=a(t,z,\epsilon)+\hat{y}_{1}(t,z,\epsilon)+\hat{y}_{2}(t,z,\epsilon),)$$
where $a(t,z,\epsilon)\in\mathbb{F}\{\epsilon\}$, $\epsilon\mapsto y^{1}_{p',p}(t,z,\epsilon)$ admits $\hat{y}_{1}(t,z,\epsilon)$ as its Gevrey asymptotic expansion of order $1/k$ in $\mathcal{E}_{p',p}$, and $\epsilon\mapsto y^{2}_{p',p}(t,z,\epsilon)$ admits $\hat{y}_{2}(t,z,\epsilon)$ as its $q-$Gevrey asymptotic expansion of order $1/\kappa$ for every $0<\kappa<\frac{k}{\delta}\min_{l=1}^{D-1}d_l$ in $\mathcal{E}_{p',p}$. These asymptotic representations are common for every $0\le p\le \varsigma-1$ and all $0\le p'\le \chi_p-1$, and depend on the geometric configuration of the scheme.

The work is structured in the following sections. In Section~\ref{seccion2} and Section~\ref{seccion3}, we introduce some Banach spaces of functions which take part in the subsequent sections, and state some continuity properties under the action of certain operators. In Section~\ref{seccion41} we describe some properties of the elements involved in an auxiary $q-$difference and convolution initial value problem, in which the geometry of the problem, described in Seccion~\ref{seccion42} in detail, plays a crucial role. In Section~\ref{seccion43} we give rise to the auxiliary problem within the geometry described in the previous subsection. In Section~\ref{seccion5}, we recall some properties of Laplace and Fourier transforms that would allow us to reduce the study of main problem to the study of auxiliary problems, easier to handle. In Section~\ref{seccion6}, we recall some facts about $q-$Gevrey asymptotics and a $q-$analog verson of Ramis-Sibuya theorem, used in the subsequent sections to provide a two-level version of Ramis-Sibuya theorem (see Theorem~\ref{teo4} in Section~\ref{seccion111}) which would provide the asymptotic behavior of the analytic solution of the main problem. In Section~\ref{seccion112}, we give a detailed summary of the steps to follow in the resolution of the main problem under study. More precisely, we split the problem into two auxiliary problems: a first equation (\ref{e1p}), studied in detail in Section~\ref{seccion7}, for which we get the existence of a unique formal solution (see Theorem~\ref{teo2}) asymptotically related to the analytic solutions, for which we also provide estimations for the difference of two of them (see Theorem~\ref{teo1}). A second equation (\ref{e2_1}) is studied in Section~\ref{seccion8} and Section~\ref{seccion9} in two different Borel planes. We also give upper bounds of such solutions which are used in Section~\ref{seccion10} in the study of the second term in the splitting. In this section, we obtain the existence of an analytic solution of such problem, and estimations of the difference of two consecutive solutions (see Theorem~\ref{teo3}). We conclude the work with the main result of the present work, Theorem~\ref{teopral}, in which we give detail on the appearance of two different asymptotic behavior observed in the solution that emerge from the splitting.

\section{Banach spaces of functions with q-exponential growth and exponential decay}\label{seccion2}

\begin{defin}\label{defi1} Let $\beta, \mu \in \mathbb{R}$. We denote by
$E_{(\beta,\mu)}$ the vector space of continuous functions $h : \mathbb{R} \rightarrow \mathbb{C}$ such that
$$ ||h(m)||_{(\beta,\mu)} = \sup_{m \in \mathbb{R}} (1+|m|)^{\mu} \exp( \beta |m|) |h(m)| $$
is finite. The space $E_{(\beta,\mu)}$ equipped with the norm $||.||_{(\beta,\mu)}$ is a Banach space.
\end{defin}

\begin{defin}\label{defi2} Let $k_{1},\beta,\mu, \tau_{0}>0$, $q>1$ be positive real number and let $\alpha \in \mathbb{R}$.
Let $\Omega$ be an open subset of $\mathbb{C}$. We denote $\bar{\Omega}$ its closure. We denote
$\mathrm{Exp}_{(k_{1},\beta,\mu,\alpha);\bar{\Omega}}^{q}$ the vector space of continuous functions
$(\tau,m) \mapsto h(\tau,m)$ on $\bar{\Omega} \times \mathbb{R}$, holomorphic w.r.t. $\tau$ on $\Omega$
and such that
\begin{multline*}
||h(\tau,m)||_{(k_{1},\beta,\mu,\alpha);\bar{\Omega}} \\
= \sup_{\tau \in \bar{\Omega},m \in \mathbb{R}}
(1+|m|)^{\mu} e^{\beta |m|} \frac{1}{|\tau|} \exp( -\frac{k_{1}}{2}
\frac{\log^{2}(|\tau|+\tau_{0})}{\log(q)} - \alpha\log(|\tau|+\tau_{0}) )|h(\tau,m)|
\end{multline*}
is finite. One can check that the space $\mathrm{Exp}_{(k_{1},\beta,\mu,\alpha);\bar{\Omega}}^{q}$ endowed with the
norm $||.||_{(k_{1},\beta,\mu,\alpha);\bar{\Omega}}$ is a Banach space.
\end{defin}

We first check some elementary properties of these latter Banach spaces.

\begin{prop}\label{prop1} 1) Let $\Omega'$ be an open subset of $\Omega$. If
$f(\tau,m) \in \mathrm{Exp}_{(k_{1},\beta,\mu,\alpha);\bar{\Omega}}^{q}$, then
$f(\tau,m) \in \mathrm{Exp}_{(k_{1},\beta,\mu,\alpha);\bar{\Omega'}}^{q}$ and
\begin{equation}
||f(\tau,m)||_{(k_{1},\beta,\mu,\alpha);\bar{\Omega'}} \leq ||f(\tau,m)||_{(k_{1},\beta,\mu,\alpha);\bar{\Omega}} 
\end{equation}
2) Let $a(\tau,m)$ be a bounded continuous function on $\bar{\Omega} \times \mathbb{R}$, holomorphic w.r.t. $\tau$
on $\Omega$. Then,
\begin{equation}
||a(\tau,m)f(\tau,m)||_{(k_{1},\beta,\mu,\alpha);\bar{\Omega}} \leq
\left(\sup_{\tau \in \bar{\Omega},m \in \mathbb{R}}|a(\tau,m)|\right) ||f(\tau,m)||_{(k_{1},\beta,\mu,\alpha);\bar{\Omega}}
\end{equation}
for all $f(\tau,m) \in \mathrm{Exp}_{(k_{1},\beta,\mu,\alpha);\bar{\Omega}}^{q}$
\end{prop}

Let $\Omega$ be an open set in $\mathbb{C}$. Let $q > 1$, $\delta>0$ be real numbers. We denote
$q^{\delta}\Omega = \{ q^{\delta}\tau \in \mathbb{C} / \tau \in \Omega \}$.

\begin{prop}\label{prop2} Let $R_{1}(X),R_{2}(X) \in \mathbb{C}[X]$ such that
\begin{equation}
\mathrm{deg}(R_1) \leq \mathrm{deg}(R_2) \ \ , \ \ R_{2}(im) \neq 0
\end{equation}
for all $m \in \mathbb{R}$. Let $m \mapsto b(m)$ be a continuous function on $\mathbb{R}$ such that
$$ |b(m)| \leq \frac{1}{|R_{2}(im)|} $$
for all $m \in \mathbb{R}$. We also consider a function $c(m)$ that belongs to $E_{(\beta,\mu)}$.
Assume that $\mu \ge \mathrm{deg}(R_1) + 1$. Let $\gamma_{1},\gamma_{2} > 0$ be real numbers such that
\begin{equation}
 \gamma_{1} \geq k_{1}\delta + \gamma_{2}. \label{constraint_gamma_1_2}
\end{equation}
Let $a_{\gamma_{1}}(\tau)$ be a holomorphic function on $\Omega$, continuous on $\bar{\Omega}$, with
$$ |a_{\gamma_1}(\tau)| \leq \frac{1}{(1 + |\tau|)^{\gamma_1}} $$
for all $\tau \in \bar{\Omega}$. Then, there exists a constant $C_{1}>0$ (depending on
$q,\alpha,\delta,R_{1},R_{2},\mu$) such that
\begin{multline}
|| \tau^{\gamma_{2}} a_{\gamma_1}(\tau) b(m) \int_{-\infty}^{+\infty}
c(m-m_{1}) R_{1}(im_{1}) f(q^{\delta}\tau,m_{1}) dm_{1} ||_{(k_{1},\beta,\mu,\alpha);\bar{\Omega}}\\
\leq C_{1} \left( \sup_{\tau \in \bar{\Omega}} \frac{|\tau|^{\gamma_{2}}}{(1+|\tau|)^{\gamma_{1}}}
((q^{\delta}|\tau| + \tau_{0})(|\tau| + \tau_{0}))^{k_{1}\delta/2} \right) ||c(m)||_{(\beta,\mu)}
||f(\tau,m)||_{(k_{1},\beta,\mu,\alpha);\overline{q^{\delta}\Omega}}
\end{multline}
for all $f(\tau,m) \in \mathrm{Exp}_{(k_{1},\beta,\mu,\alpha);\overline{q^{\delta}\Omega}}^{q}$.
\end{prop}
\begin{proof}
Let $f(\tau,m) \in \mathrm{Exp}_{(k_{1},\beta,\mu,\alpha);\overline{q^{\delta}\Omega}}^{q}$. By the definition of the norm and the assumptions on the functions $a_{\gamma_1}(\tau)$ and $b(m)$, we can write
\begin{multline}
A := || \tau^{\gamma_{2}} a_{\gamma_1}(\tau) b(m) \int_{-\infty}^{+\infty}
c(m-m_{1}) R_{1}(im_{1}) f(q^{\delta}\tau,m_{1}) dm_{1} ||_{(k_{1},\beta,\mu,\alpha);\Omega}\\
\leq \sup_{\tau \in \bar{\Omega}, m \in \mathbb{R}} (1+|m|)^{\mu} e^{\beta|m|} \frac{1}{|\tau|}
\exp( -\frac{k_{1}}{2}
\frac{\log^{2}(|\tau|+\tau_{0})}{\log(q)} - \alpha\log(|\tau|+\tau_{0}) )
\frac{|\tau|^{\gamma_{2}}}{(1+|\tau|)^{\gamma_1}} \frac{1}{|R_{2}(im)|}\\
\times \int_{-\infty}^{+\infty} \{ (1+|m-m_{1}|)^{\mu} e^{\beta|m-m_{1}|} |c(m-m_{1})| \}
\\
\times \{ (1+|m_{1}|)^{\mu} e^{\beta |m_{1}|}
\frac{1}{|q^{\delta}\tau|} \exp( -\frac{k_{1}}{2}
\frac{\log^{2}(|q^{\delta}\tau|+\tau_{0})}{\log(q)} - \alpha\log(|q^{\delta}\tau|+\tau_{0}) )
|f(q^{\delta}\tau,m_{1})| \}\\
\times \mathcal{A}(\tau,m,m_{1}) dm_{1}
\end{multline}
where
\begin{multline*}
\mathcal{A}(\tau,m,m_{1}) = |R_{1}(im_{1})|
\frac{e^{-\beta|m-m_{1}|}e^{-\beta|m_{1}|}}{(1+|m-m_{1}|)^{\mu}(1+|m_{1}|)^{\mu}} |q^{\delta}\tau|\\
\times \exp( \frac{k_{1}}{2}
\frac{\log^{2}(|q^{\delta}\tau|+\tau_{0})}{\log(q)} + \alpha\log(|q^{\delta}\tau|+\tau_{0}) )
\end{multline*}
We deduce that
\begin{equation}
A \leq C_{1.1}C_{1.2}||c(m)||_{(\beta,\mu)}||f(\tau,m)||_{(k_{1},\beta,\mu,\alpha);\overline{q^{\delta}\Omega}} 
\end{equation}
with
\begin{multline*}
C_{1.1} = \sup_{\tau \in \bar{\Omega}} \exp( -\frac{k_{1}}{2}
\frac{\log^{2}(|\tau|+\tau_{0})}{\log(q)} - \alpha\log(|\tau|+\tau_{0}) )
\exp( \frac{k_{1}}{2}
\frac{\log^{2}(|q^{\delta}\tau|+\tau_{0})}{\log(q)} + \alpha\log(|q^{\delta}\tau|+\tau_{0}) )\\
\times \frac{|\tau|^{\gamma_{2}}}{(1+|\tau|)^{\gamma_{1}}}
\end{multline*}
and
$$
C_{1.2} = \sup_{m \in \mathbb{R}} q^{\delta}(1+|m|)^{\mu}e^{\beta |m|} \frac{1}{|R_{2}(im)|}
\int_{-\infty}^{+\infty}
\frac{|R_{1}(im_{1})|e^{-\beta|m-m_{1}|}e^{-\beta|m_{1}|}}{(1+|m-m_{1}|)^{\mu} (1+|m_{1}|)^{\mu}} dm_{1}.
$$
We first provide estimates for $C_{1.1}$. Indeed, using the fact that the function
$\psi(x)= (q^{\delta}x+\tau_{0})/(x+\tau_{0})$ is increasing on $\mathbb{R}_{+}$, we get that
\begin{equation}
\alpha\log(|q^{\delta}\tau|+\tau_{0}) - \alpha\log(|\tau|+\tau_{0}) =
\alpha \log( \frac{q^{\delta}|\tau| + \tau_{0}}{|\tau|+\tau_{0}} ) \leq \alpha \delta \log(q) \label{maj_diff_log}
\end{equation}
for all $\tau \in \bar{\Omega}$. Moreover, using (\ref{maj_diff_log}), we get that
\begin{multline}
\frac{k_1}{2\log(q)} \left( \log^{2}(|q^{\delta}\tau| + \tau_{0}) - \log^{2}(|\tau|+\tau_{0}) \right)\\
=
\frac{k_1}{2\log(q)} \log( \frac{|q^{\delta}\tau| + \tau_{0}}{|\tau| + \tau_{0}} )
\log( (|q^{\delta}\tau| + \tau_{0})(|\tau|+\tau_{0}) ) \leq
\frac{k_{1}\delta}{2} \log( (|q^{\delta}\tau| + \tau_{0})(|\tau|+\tau_{0}) ) \label{maj_diff_log_square}
\end{multline}
for all $\tau \in \bar{\Omega}$. From (\ref{maj_diff_log}) and (\ref{maj_diff_log_square}), we deduce that
\begin{equation}
C_{1.1} \leq q^{\alpha \delta} \sup_{\tau \in \bar{\Omega}} \frac{|\tau|^{\gamma_{2}}}{(1+|\tau|)^{\gamma_{1}}}
((q^{\delta}|\tau| + \tau_{0})(|\tau| + \tau_{0}))^{k_{1}\delta/2}
\end{equation}
which is finite since the condition (\ref{constraint_gamma_1_2}) holds. In the last step, we give upper bounds
for $C_{2.2}$. There exist two constants $\mathfrak{R}_{1},\mathfrak{R}_{2}>0$ such that
$$ |R_{1}(im_{1})| \leq \mathfrak{R}_{1}(1+|m_1|)^{\mathrm{deg}(R_1)} \ \ , \ \
|R_{2}(im)| \geq \mathfrak{R}_{2}(1+|m|)^{\mathrm{deg}(R_1)}
$$
for all $m,m_{1} \in \mathbb{R}$. Using these latter inequalities and the triangular inequality 
$|m| \leq |m-m_{1}| + |m_{1}|$, we get that
\begin{equation}
C_{1.2} \leq q^{\delta}\frac{\mathfrak{R}_{1}}{\mathfrak{R}_{2}}
\sup_{m \in \mathbb{R}} (1+|m|)^{\mu - \mathrm{deg}(R_{2})} \int_{-\infty}^{+\infty}
\frac{dm_{1}}{(1 + |m-m_{1}|)^{\mu}(1+|m_{1}|)^{\mu - \mathrm{deg}(R_1)}}
\end{equation}
which is finite provided that $\mu > \mathrm{deg}(R_1)$ and $\mathrm{deg}(R_2) \geq \mathrm{deg}(R_1)$
according to Lemma 4 in \cite{cota2} (see also Lemma 2.2 from \cite{ma2}).
\end{proof}

\section{Banach spaces of functions with exponential growth of finite order and decay of order 1}\label{seccion3}

In this section, we mainly recall some functional properties of the Banach spaces already introduced in the work
\cite{lama0}. Throughout the whole section we fix some positive real numbers $\nu,\beta,\mu>0$ and an integer $k \geq 1$.
Let $\Omega$ be a nonempty open subset of $\mathbb{C}$. As in the previous section, $\bar{\Omega}$ denotes
the closure of $\Omega$. We assume that $\bar{\Omega}$ is a \emph{starshaped} domain, meaning that
for all $z \in \bar{\Omega}$, the segment $[0,z]$ is contained in $\bar{\Omega}$.

\begin{defin}\label{defi3}
Let $F_{(\nu,\beta,\mu,k);\bar{\Omega}}$ be the vector space of continuous functions $(\tau,m) \mapsto h(\tau,m)$ on
$\bar{\Omega} \times \mathbb{R}$, which are holomorphic with respect to $\tau$ on $\Omega$ and
such that
$$ ||h(\tau,m)||_{(\nu,\beta,\mu,k);\bar{\Omega}} =
\sup_{\tau \in \bar{\Omega},m \in \mathbb{R}} (1+|m|)^{\mu}
\frac{1 + |\tau|^{2k}}{|\tau|}\exp( \beta|m| - \nu|\tau|^{k} ) |h(\tau,m)|$$
is finite. One can check that the normed space
$(F_{(\nu,\beta,\mu,k);\bar{\Omega}},||.||_{(\nu,\beta,\mu,k);\bar{\Omega}})$ is a Banach space.
\end{defin}

In the next lemma, we check
the continuity property related to the previous space by product with bounded functions.

\begin{lemma}\label{lema1} Let $(\tau,m) \mapsto a(\tau,m)$ be a continuous bounded function on
$\bar{\Omega} \times \mathbb{R}$, which is holomorphic with respect to $\tau$ on $\Omega$. Then, we have
\begin{equation}
|| a(\tau,m) h(\tau,m) ||_{(\nu,\beta,\mu,k);\bar{\Omega}} \leq
\left( \sup_{\tau \in \bar{\Omega},m \in \mathbb{R}} |a(\tau,m)| \right)
||h(\tau,m)||_{(\nu,\beta,\mu,k);\bar{\Omega}}
\end{equation}
for all $h(\tau,m) \in F_{(\nu,\beta,\mu,k);\bar{\Omega}}$.
\end{lemma}

In the next propositions, we study the continuity property of some convolution operators acting on the latter Banach spaces. Proposition~\ref{prop3} slightly differs from Proposition~\ref{prop1} in~\cite{lama0}, but its proof can be easily adapted in our framework. More precisely, the term $a_{\gamma_1,k}(\tau)$ and condition $\gamma_1\ge\nu_2$ join in the proof when estimating $B_{2.1}$ in the proof of Proposition 1 in~\cite{lama0}. Proposition~\ref{prop241} coincides with Proposition 2 in~\cite{lama0}. For these reasons, we omit the proofs of the following two results.

\begin{prop}\label{prop3} Let $\gamma_{1} \geq 0$ and $\chi>-1$ be real numbers. Let $\nu_{2} \geq -1$ be an integer.
We consider a function $a_{\gamma_{1},k}(\tau)$, continuous on
$\bar{\Omega}$ and holomorphic in $\Omega$, such that
$$ |a_{\gamma_{1},k}(\tau)| \leq \frac{1}{(1+|\tau|^{k})^{\gamma_1}} $$
for all $\tau \in \bar{\Omega}$.\medskip

\noindent If $1 + \chi + \nu_{2} \geq 0$ and $\gamma_{1} \geq \nu_{2}$, then there exists a constant $E_{2}>0$
(depending on $\nu,\nu_{2},\chi,\gamma_{1}$) such that
\begin{equation}
|| a_{\gamma_{1},k}(\tau) \int_{0}^{\tau^k} (\tau^{k}-s)^{\chi}s^{\nu_{2}}f(s^{1/k},m)
ds ||_{(\nu,\beta,\mu,k);\bar{\Omega}} \\
\leq E_{2}||f(\tau,m)||_{(\nu,\beta,\mu,k);\bar{\Omega}}
\label{conv_op_prod_s_continuity_1_a_gamma_1}
\end{equation}
for all $f(\tau,m) \in F_{(\nu,\beta,\mu,k);\bar{\Omega}}$.
\end{prop}

\begin{prop}\label{prop241} Let $Q(X),R(X) \in \mathbb{C}[X]$ be polynomials such that
\begin{equation}
\mathrm{deg}(R) \geq \mathrm{deg}(Q) \ \ , \ \ R(im) \neq 0 \label{cond_R_Q}
\end{equation}
for all $m \in \mathbb{R}$. Assume that $\mu > \mathrm{deg}(Q) + 1$. Let $m \mapsto b(m)$ be a continuous
function such that
$$
|b(m)| \leq \frac{1}{|R(im)|}
$$
for all $m \in \mathbb{R}$. Then, there exists a constant $E_{3}>0$ (depending on $Q,R,\mu,k,\nu$) such that
\begin{multline}
|| b(m) \int_{0}^{\tau^{k}} (\tau^{k}-s)^{\frac{1}{k}}\int_{-\infty}^{+\infty} f(m-m_{1})Q(im_{1})g(s^{1/k},m_{1})dm_{1}
\frac{ds}{s}||_{(\nu,\beta,\mu,k);\bar{\Omega}}\\
\leq E_{3} ||f(m)||_{(\beta,\mu)} ||g(\tau,m)||_{(\nu,\beta,\mu,k);\bar{\Omega}}
\label{norm_k_conv_f_g<norm_f_beta_mu_times_norm_k_g}
\end{multline}
for all $f(m) \in E_{(\beta,\mu)}$, all $g(\tau,m) \in F_{(\nu,\beta,\mu,k);\bar{\Omega}}$.
\end{prop}

\section{Analytic solutions of an auxiliary $q-$difference and convolution initial value problem}

\subsection{Layout of the initial value problem}\label{seccion41}
Let $k \geq 1$ and $D \geq 2$ be integers. We also fix $\beta,\mu>0$. Let $\delta,k_{1}>0$, $q>1$ and $\alpha$ be real numbers and for $1 \leq l \leq D$, let
$d_{l},\Delta_{l} \geq 0$ be nonnegative integers. We make the assumption that
\begin{equation}
kd_{D} - 1 \geq k_{1}\delta + kd_{l} \ \ , \ \ \Delta_{l} \geq kd_{l} \label{assum_dD_dl_k1_Delta_l}
\end{equation}
for all $1 \leq l \leq D-1$. Let $Q(X),R_{l}(X) \in \mathbb{C}[X]$, $1 \leq l \leq D$, be polynomials such that
\begin{equation}
\mathrm{deg}(Q) \geq \mathrm{deg}(R_D) \geq \mathrm{deg}(R_{l}) \ \ , \ \ Q(im) \neq 0 \ \ , \ \
R_{D}(im) \neq 0, \label{assum_deg_Q_RD_Rl}
\end{equation}
for all $m \in \mathbb{R}$, all $1 \leq l \leq D-1$. We make the additional assumption that there exists a bounded sectorial
annulus
$$ A_{Q,R_{D}} = \{ z \in \mathbb{C} / r_{Q,R_{D}} \leq |z| \leq r_{Q,R_{D}}^{1} \ \ , \ \
|\mathrm{arg}(z) - d_{Q,R_{D}}| \leq \eta_{Q,R_{D}} \} $$
with direction $d_{Q,R_{D}} \in \mathbb{R}$, aperture $2\eta_{Q,R_{D}}>0$ for some radius
$r_{Q,R_{D}},r_{Q,R_{D}}^{1}>0$ such that
\begin{equation}
\frac{Q(im)}{R_{D}(im)} \in A_{Q,R_{D}} \label{quotient_Q_RD_in_Ann}
\end{equation} 
for all $m \in \mathbb{R}$. We define the polynomial $P_{m}(\tau) = Q(im) - R_{D}(im)(k\tau^{k})^{d_{D}}$ and we
factorize it in the form
\begin{equation}
P_{m}(\tau) = -R_{D}(im)k^{d_{D}}\Pi_{l=0}^{kd_{D}-1} (\tau - q_{l}(m)) \label{factor_P_m_q_lm}
\end{equation}
where
\begin{equation}
q_{l}(m) = \left( \frac{|Q(im)|}{|R_{D}(im)|k^{d_D}} \right)^{\frac{1}{kd_{D}}}
\exp \left( \sqrt{-1}( \mathrm{arg}(\frac{Q(im)}{R_{D}(im)k^{d_{D}}})\frac{1}{kd_{D}} + \frac{2\pi l}{kd_{D}} ) \right)
\label{defin_roots_q_lm}
\end{equation}
for all $0 \leq l \leq kd_{D}-1$.
Let $\varsigma \geq 2$ be an integer and let $\mathfrak{d}_{p}$, $0 \leq p \leq \varsigma-1$, be real numbers.
We choose a set of unbounded sectors $U_{\mathfrak{d}_{p}}$, $0 \leq p \leq \varsigma-1$, with bisecting direction $\mathfrak{d}_{p}$, a small closed disc
$\bar{D}(0,\mu_{0})$, an annulus $\bar{A}_{\mu_{1}} = \{ \tau \in \mathbb{C}/ |\tau| \geq \mu_{1} \}$,
where $\mu_{1} > \mu_{0}$ are positive real numbers, and we prescribe the sectorial annulus $A_{Q,R_{D}}$ in such a way that
the next conditions hold.\medskip

\noindent 1) There exists a constant $M_{1}>0$ such that
\begin{equation}
|\tau - q_{l}(m)| \geq M_{1}(1 + |\tau|) \label{root_cond_1_q_lm}
\end{equation}
for all $0 \leq l \leq kd_{D}-1$, all $m \in \mathbb{R}$, all $\tau \in \bar{U}_{\mathfrak{d}_{p}} \cup
\bar{D}(0,\mu_{0}) \cup \bar{A}_{\mu_1}$. Indeed,
from (\ref{quotient_Q_RD_in_Ann}) and the explicit expression (\ref{defin_roots_q_lm}) of $q_{l}(m)$, we first observe
that one can select a small real number $\rho>0$ such that
$$ \mu_{0} + \rho < |q_{l}(m)| < \mu_{1} - \rho $$
for every $m \in \mathbb{R}$, all $0 \leq l \leq kd_{D}-1$ for an appropriate choice of $r_{Q,R_{D}},r_{Q,R_{D}}^{1}>0$
and of $\mu_{0},\mu_{1}>0$. We also see that for all $m \in \mathbb{R}$, all $0 \leq l \leq kd_{D}-1$, the roots
$q_{l}(m)$ remain in a union $\mathcal{U}$ of unbounded sectors with vertex at 0 that do not cover a full neighborhood of
the origin in $\mathbb{C}^{\ast}$ provided that $\eta_{Q,R_{D}}$ is small enough. Therefore, one can choose an adequate
set of unbounded sectors $\bar{U}_{\mathfrak{d}_{p}}$ with bisecting direction $\mathfrak{d}_{p}$, such that
$\bar{U}_{\mathfrak{d}_{p}} \cap \mathcal{U} = \emptyset$ for all $0\le p\le \varsigma-1$, with the property that for all
$0 \leq l \leq kd_{D}-1$ the quotients $q_{l}(m)/\tau$ lay outside
some small disc centered at 1 in $\mathbb{C}$ for all $\tau \in \bar{U}_{\mathfrak{d}_{p}}$, all $m \in \mathbb{R}$,
all $0 \leq p \leq \varsigma-1$. This yields (\ref{root_cond_1_q_lm}) for some small constant $M_{1}>0$.\medskip

\noindent 2) There exists a constant $M_{2}>0$ such that
\begin{equation}
|\tau - q_{l_0}(m)| \geq M_{2}|q_{l_0}(m)| \label{root_cond_2_q_lm}
\end{equation}
for some $l_{0} \in \{0,\ldots,kd_{D}-1 \}$, all $m \in \mathbb{R}$, all
$\tau \in \bar{U}_{\mathfrak{d}_p} \cup \bar{D}(0,\mu_{0}) \cup \bar{A}_{\mu_1}$. Indeed, for the
sectors $\bar{U}_{\mathfrak{d}_{p}}$, the disc $\bar{D}(0,\mu_{0})$ and the annulus $\bar{A}_{\mu_1}$
chosen as above in 1), we notice that for any fixed
$0 \leq l_{0} \leq kd_{D}-1$, the quotient $\tau/q_{l_0}(m)$ stays outside a small disc centered at 1 in
$\mathbb{C}$ for all $\tau \in \bar{U}_{\mathfrak{d}_{p}} \cup \bar{D}(0,\mu_{0}) \cup \bar{A}_{\mu_1}$,
all $m \in \mathbb{R}$, all $0 \leq p \leq \varsigma-1$. Hence (\ref{root_cond_2_q_lm}) must hold
for some small constant $M_{2}>0$.\medskip

\noindent 3) The set of directions $\mathfrak{d}_{p}$, $0 \leq p \leq \varsigma-1$ are chosen in such a way that
each sector
$$\bar{U}_{\mathfrak{d}_{p},\mathfrak{d}_{p+1}} =
\{ \tau \in \mathbb{C}^{\ast} / \mathrm{arg}(\tau) \in [\mathfrak{d}_{p},\mathfrak{d}_{p+1}] \}$$
contains at least
one root $q_{l}(m)$ of $P_{m}(\tau)$, for some $0 \leq l \leq kd_{D}-1$, all $m \in \mathbb{R}$. In particular, this entails $kd_D\ge \varsigma$. \medskip

By construction
of the roots (\ref{defin_roots_q_lm}) in the factorization (\ref{factor_P_m_q_lm}) and using the lower bound estimates
(\ref{root_cond_1_q_lm}), (\ref{root_cond_2_q_lm}), we get a constant $C_{P}>0$ such that
\begin{multline}
|P_{m}(\tau)| \geq M_{1}^{kd_{D}-1}M_{2}|R_{D}(im)|k^{d_{D}}( \frac{|Q(im)|}{|R_{D}(im)|k^{d_{D}}})^{\frac{1}{kd_{D}}}
(1+|\tau|)^{kd_{D}-1}\\
\geq C_{P}(r_{Q,R_{D}})^{\frac{1}{kd_{D}}}|R_{D}(im)|(1+|\tau|)^{kd_{D}-1} \label{lower_bounds_Pm}
\end{multline}
for all $\tau \in \bar{U}_{\mathfrak{d}_{p}} \cup \bar{D}(0,\mu_{0}) \cup \bar{A}_{\mu_1}$,
all $m \in \mathbb{R}$, all $0 \leq p \leq \varsigma-1$.\medskip

Let $m \mapsto C_{l}(m,\epsilon)$, $1 \leq l \leq D-1$, be functions that belong to the space
$E_{(\beta,\mu)}$ for some $\beta>0$ and $\mu > \mathrm{deg}(R_{D}) + 1$ and depend holomorphically on
$\epsilon$ on $D(0,\epsilon_{0})$ for some $\epsilon_{0}>0$. We also define a function 
$\psi_{k}(\tau,m,\epsilon)$ that belongs to the Banach space
$\mathrm{Exp}_{(k_{1},\beta,\mu,\alpha);\mathbb{C}}^{q}$ (see Definition~\ref{defi2}) for all $\epsilon \in D(0,\epsilon_{0})$.
We consider now the $q-$difference and convolution initial value problem
\begin{multline}
w_{k}(\tau,m,\epsilon) = \sum_{l=1}^{D-1} \epsilon^{\Delta_{l} - kd_{l}}
\frac{(k \tau^{k})^{d_{l}}}{P_{m}(\tau)} \frac{1}{(2\pi)^{1/2}}
\int_{-\infty}^{+\infty} C_{l}(m-m_{1},\epsilon)R_{l}(im_{1}) w_{k}(q^{\delta}\tau,m_{1},\epsilon) dm_{1}\\
+ \frac{\psi_{k}(\tau,m,\epsilon)}{P_{m}(\tau)} \label{q_diff_conv_init_v_prob}
\end{multline}
for given initial data $w_{k}(0,m,\epsilon) \equiv 0$. The main purpose of this section is the construction of
actual holomorphic solutions $w_{k,p}(\tau,m,\epsilon)$ of this problem which belong to some Banach space
$\mathrm{Exp}_{(k_{1},\beta,\mu,\alpha);\bar{\Omega}}^{q}$ on some suitable domains $\bar{\Omega} \subset \mathbb{C}$,
$0 \leq p \leq \varsigma-1$.

\subsection{Construction of the domains $R\Omega_{j}^{p}$, $\Theta_{\hat{Q}\mu_{1}}^p$ and norm estimates}\label{seccion42}

We preserve the value of the constants and the geometric constructions of the previous subsection. We assume there exist $\hat{Q}>1$ and $0 < \check{Q} < 1$ such that
\begin{equation}
\hat{Q}\mu_{1} = q^{\delta}\check{Q}\mu_{0}. \label{constraint_mu1_mu0}
\end{equation}
We consider two sequences $\mu_{0,h} = (1/q^{\delta})^{h}\mu_{0}$ and $\mu_{1,h} = (1/q^{\delta})^{h}\mu_{1}$, for
all integers $h \in \mathbb{Z}$. Let $0 \leq p \leq \varsigma-1$. By convention, we put
$\mathfrak{d}_{\varsigma} = \mathfrak{d}_{0}$. We recall that
$$\bar{U}_{\mathfrak{d}_{p},\mathfrak{d}_{p+1}} =
\{ \tau \in \mathbb{C}^{\ast} / \mathrm{arg}(\tau) \in [\mathfrak{d}_{p},\mathfrak{d}_{p+1}] \}.$$
contains at least
one root $q_{l}(m)$ of $P_{m}(\tau)$, for some $0 \leq l \leq kd_{D}-1$, all $m \in \mathbb{R}$ (see Condition 3) in the previous subsection). Notice furthermore
that by construction, all the roots $q_{l}(m)$, for all $0 \leq l \leq kd_{D}-1$, all $m \in \mathbb{R}$ are located
in the annulus $\{ \tau \in \mathbb{C} / \mu_{0}+\rho < |\tau| < \mu_{1} - \rho \}$ for some small positive real number
$\rho>0$.

We first define a domain $P\Omega_{h}^{p}$ which is a
\emph{sectorial square frame} built as follows. We consider the four sectorial annuli
\begin{multline*}
C_{\mu_{1,h}}^{p} = \{ \tau \in \mathbb{C} / \mu_{1,h} \leq |\tau| \leq \hat{Q}\mu_{1,h} \} \cap 
(\bar{U}_{\mathfrak{d}_{p}} \cup \bar{U}_{\mathfrak{d}_{p},\mathfrak{d}_{p+1}} \cup \bar{U}_{\mathfrak{d}_{p+1}}),\\
C_{\mu_{0,h}}^{p} = \{ \tau \in \mathbb{C} / \check{Q}\mu_{0,h} \leq |\tau| \leq \mu_{0,h} \} \cap 
(\bar{U}_{\mathfrak{d}_{p}} \cup \bar{U}_{\mathfrak{d}_{p},\mathfrak{d}_{p+1}} \cup \bar{U}_{\mathfrak{d}_{p+1}}),\\
A_{\mu_{0,h},\mu_{1,h}}^{p} = \{ \tau \in \mathbb{C} / \check{Q}\mu_{0,h} \leq |\tau| \leq \hat{Q}\mu_{1,h} \} \cap
\bar{U}_{\mathfrak{d}_p},\\
A_{\mu_{0,h},\mu_{1,h}}^{p+1} = \{ \tau \in \mathbb{C} / \check{Q}\mu_{0,h} \leq |\tau| \leq \hat{Q}\mu_{1,h} \} \cap
\bar{U}_{\mathfrak{d}_{p+1}}
\end{multline*}
We define
$$ P\Omega_{h}^{p} = C_{\mu_{1,h}}^{p} \cup C_{\mu_{0,h}}^{p} \cup A_{\mu_{0,h},\mu_{1,h}}^{p} \cup
A_{\mu_{0,h},\mu_{1,h}}^{p+1} $$
for all $h \in \mathbb{Z}$, all $0 \leq p \leq \varsigma-1$. See Figure~\ref{figure1}.

\begin{figure}[h]
	\centering
		\includegraphics[width=.5\textwidth]{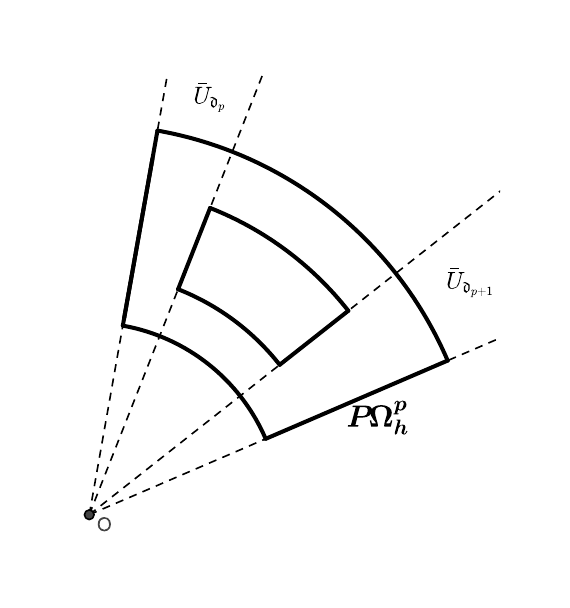}
	\caption{Sectorial square frame $P\Omega_h^p$}
	\label{figure1}
\end{figure}

We also define the union of some of these frames with a disc
$$ R\Omega_{j}^{p} = (\cup_{h=0}^{j} P\Omega_{h}^{p}) \cup \bar{D}(0,\check{Q}\mu_{0,j}) $$
for all integers $j \geq 0$. Notice that due to the assumption (\ref{constraint_mu1_mu0}), we get in
particular that
\begin{equation}
\bar{U}_{\mathfrak{d}_{p}} \cap \bar{D}(0,\hat{Q}\mu_{1}) \subset R\Omega_{j}^p \ \ , \ \ 
\bar{U}_{\mathfrak{d}_{p+1}} \cap \bar{D}(0,\hat{Q}\mu_{1}) \subset R\Omega_{j}^p
\end{equation}
for all $j \geq 0$. Finally, we introduce the \emph{sectorial triangle frame}
\begin{equation}
\Theta_{\hat{Q}\mu_{1}}^{p} = \bar{U}_{\mathfrak{d}_{p}} \cup \bar{U}_{\mathfrak{d}_{p+1}} \cup \bar{A}_{\hat{Q}\mu_{1}}.
\end{equation}

\begin{figure}[h]
	\centering
		\includegraphics[width=.5\textwidth]{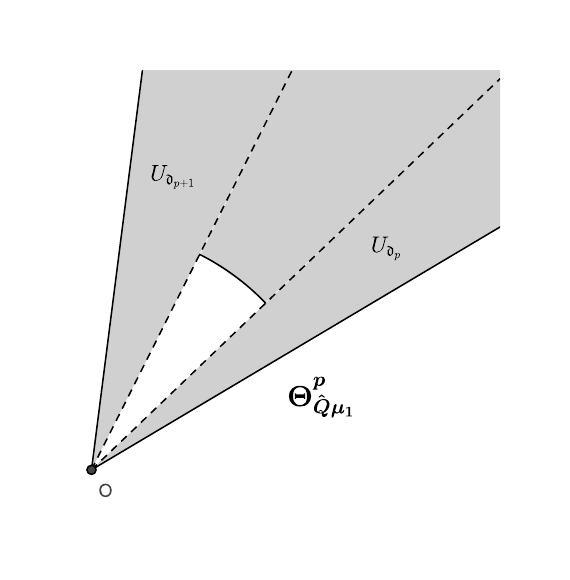}
	\caption{Sectorial triangle frame $\Theta^p_{\hat{Q}\mu_1}$}
	\label{figure2}
\end{figure}

We provide now norm estimates on $q-$difference/convolution operators acting on the Banach spaces of functions introduced in Section~\ref{seccion2}, considered on the above domains $R\Omega_{j}^{p}$ and $\Theta_{\hat{Q}\mu_{1}}^{p}$. Under the hypotheses made in Section~\ref{seccion41}, we
get

\begin{prop} There exists a constant $C_{2}>0$ (depending on $q,\alpha,\delta,R_{l},R_{D},\mu$) such that
\begin{multline}
|| \frac{\tau^{kd_{l}}}{P_{m}(\tau)} \int_{-\infty}^{+\infty} C_{l}(m-m_{1},\epsilon)R_{l}(im_{1})
f(q^{\delta}\tau,m_{1})dm_{1}||_{(k_{1},\beta,\mu,\alpha);P\Omega_{h+1}^{p}}\\
\leq \frac{C_{2}}{C_{P}(r_{Q,R_{D}})^{\frac{1}{kd_{D}}}}
\frac{ |\hat{Q}\mu_{1,h+1}|^{kd_{l}} }{(1 + \check{Q}\mu_{0,h+1})^{kd_{D}-1}}
((q^{\delta}\hat{Q}\mu_{1,h+1} + \tau_{0})(\hat{Q}\mu_{1,h+1} + \tau_{0}))^{\frac{k_{1}\delta}{2}}\\
\times ||f(\tau,m)||_{(k_{1},\beta,\mu,\alpha);P\Omega_{h}^{p}} \times ||C_{l}(m,\epsilon)||_{(\beta,\mu)}
\label{norm_H_Pomega_h}
\end{multline}
for all $h \in \mathbb{Z}$, all $1 \leq l \leq D-1$, for all
$f(\tau,m) \in \mathrm{Exp}_{(k_{1},\beta,\mu,\alpha);P\Omega_{h}^{p}}^{q}$ and
\begin{multline}
|| \frac{\tau^{kd_{l}}}{P_{m}(\tau)} \int_{-\infty}^{+\infty} C_{l}(m-m_{1},\epsilon)R_{l}(im_{1})
f(q^{\delta}\tau,m_{1})dm_{1}||_{(k_{1},\beta,\mu,\alpha);\bar{D}(0,\check{Q}\mu_{0,j+1})}\\
\leq \frac{C_{2}}{C_{P}(r_{Q,R_{D}})^{\frac{1}{kd_{D}}}} |\check{Q}\mu_{0,j+1}|^{kd_{l}}
((q^{\delta}\check{Q}\mu_{0,j+1}+\tau_{0})(\check{Q}\mu_{0,j+1}+\tau_{0}))^{\frac{k_{1}\delta}{2}}\\
\times ||f(\tau,m)||_{(k_{1},\beta,\mu,\alpha);D(0,\check{Q}\mu_{0,j})} \times ||C_{l}(m,\epsilon)||_{(\beta,\mu)}
\label{norm_H_disc_Qmuj}
\end{multline}
for all $j \geq 0$, all $1 \leq l \leq D-1$, all
$f(\tau,m) \in \mathrm{Exp}_{(k_{1},\beta,\mu,\alpha);D(0,\check{Q}\mu_{0,j})}^{q}$. Finally,
there exists a constant $C_{3}>0$ (depending on $q,\alpha,\delta, R_l,\mu,\mu_1,\hat{Q},\tau_0,k_1,k,d_l$) such that
\begin{multline}
|| \frac{\tau^{kd_{l}}}{P_{m}(\tau)} \int_{-\infty}^{+\infty} C_{l}(m-m_{1},\epsilon)R_{l}(im_{1})
f(q^{\delta}\tau,m_{1})dm_{1}||_{(k_{1},\beta,\mu,\alpha);\Theta_{\hat{Q}\mu_{1}}^{p}}\\
\leq \frac{C_{3}}{C_{P}(r_{Q,R_{D}})^{\frac{1}{kd_{D}}}}
||f(\tau,m)||_{(k_{1},\beta,\mu,\alpha);\Theta_{\hat{Q}\mu_{1}}^{p}}
\times ||C_{l}(m,\epsilon)||_{(\beta,\mu)} \label{norm_H_ThetaQmu1}
\end{multline}
for all $1 \leq l \leq D-1$, all $f(\tau,m) \in \mathrm{Exp}_{(k_{1},\beta,\mu,\alpha);\Theta_{\hat{Q}\mu_{1}}^{p}}^{q}$.
\end{prop}
\begin{proof} All the inequalities (\ref{norm_H_Pomega_h}), (\ref{norm_H_disc_Qmuj}) and (\ref{norm_H_ThetaQmu1})
are consequence of Proposition~\ref{prop2}. Indeed, in order to get (\ref{norm_H_Pomega_h}), one can check that
$q^{\delta}P\Omega_{h+1}^{p} = P\Omega_{h}^{p}$ and that
\begin{multline*}
\sup_{\tau \in P\Omega_{h+1}^{p}} \frac{|\tau|^{kd_{l}}}{(1+|\tau|)^{kd_{D}-1}}
((q^{\delta}|\tau| + \tau_{0})(|\tau| + \tau_{0}))^{k_{1}\delta/2} \\
\leq
\frac{ |\hat{Q}\mu_{1,h+1}|^{kd_{l}} }{(1 + \check{Q}\mu_{0,h+1})^{kd_{D}-1}}
((q^{\delta}\hat{Q}\mu_{1,h+1} + \tau_{0})(\hat{Q}\mu_{1,h+1} + \tau_{0}))^{\frac{k_{1}\delta}{2}}.
\end{multline*}
On the other hand, one can get (\ref{norm_H_disc_Qmuj}) as a consequence of the fact that
$q^{\delta}\bar{D}(0,\check{Q}\mu_{0,j+1}) = \bar{D}(0,\check{Q}\mu_{0,j})$ and that
\begin{multline*}
\sup_{\tau \in \bar{D}(0,\check{Q}\mu_{0,j+1})} \frac{|\tau|^{kd_{l}}}{(1+|\tau|)^{kd_{D}-1}}
((q^{\delta}|\tau| + \tau_{0})(|\tau| + \tau_{0}))^{k_{1}\delta/2} \\
\leq |\check{Q}\mu_{0,j+1}|^{kd_{l}}
((q^{\delta}\check{Q}\mu_{0,j+1}+\tau_{0})(\check{Q}\mu_{0,j+1}+\tau_{0}))^{\frac{k_{1}\delta}{2}}.
\end{multline*}
Finally, (\ref{norm_H_ThetaQmu1}) is a consequence of the inclusion $q^{\delta}\Theta_{\hat{Q}\mu_{1}}^{p} \subset
\Theta_{\hat{Q}\mu_{1}}^{p}$, and  
$$
\sup_{\tau \in \Theta_{\hat{Q}\mu_{1}}^{p}} \frac{|\tau|^{kd_{l}}}{(1+|\tau|)^{kd_{D}-1}}
((q^{\delta}|\tau| + \tau_{0})(|\tau| + \tau_{0}))^{k_{1}\delta/2}<\infty, \\
$$
due to the first inequality of (\ref{assum_dD_dl_k1_Delta_l}) holds and from the inequality
$$ ||f(\tau,m)||_{(k_{1},\beta,\mu,\alpha);q^{\delta}\Theta^p_{\hat{Q}\mu_{1}}} \leq
||f(\tau,m)||_{(k_{1},\beta,\mu,\alpha);\Theta^p_{\hat{Q}\mu_{1}}} $$
for any function $f(\tau,m) \in \mathrm{Exp}_{(k_{1},\beta,\mu,\alpha);\Theta^p_{\hat{Q}\mu_{1}}}^{q}$,
which results from Proposition~\ref{prop1}.
\end{proof}

\subsection{Construction of actual solutions on the domains
$R\Omega_{j}^{p}$, $\Theta_{\hat{Q}\mu_{1}}^{p}$ of the problem (\ref{q_diff_conv_init_v_prob})}\label{seccion43}

\begin{prop}\label{prop6} Let us assume the statements (\ref{assum_dD_dl_k1_Delta_l}), (\ref{assum_deg_Q_RD_Rl}),
(\ref{quotient_Q_RD_in_Ann}) hold. Let $\epsilon\in D(0,\epsilon_0)$. In the framework of the layout of Section~\ref{seccion41} and the geometric construction described in Subsection~\ref{seccion42}, there exists $r_{Q,R_{D}}>0$ and $\gamma_{l}>0$ such that if
\begin{equation}
\sup_{\epsilon\in D(0,\epsilon_0)}||C_{l}(m,\epsilon)||_{(\beta,\mu)} \leq \gamma_{l},
\end{equation}
for all $1 \leq l \leq D-1$, the problem
(\ref{q_diff_conv_init_v_prob}) with vanishing initial data has a solution $w_{k}^{p}(\tau,m,\epsilon)$ that
belongs to the Banach
space $\mathrm{Exp}_{(k_{1},\beta,\mu,\alpha);\Theta_{\hat{Q}\mu_{1}}^{p}}^{q}$, for all
$0 \leq p \leq \varsigma-1$. Moreover, we can write $w_{k}^{p}(\tau,m,\epsilon)$ as a sum
\begin{equation}
w_{k}^{p}(\tau,m,\epsilon) = \sum_{j \geq 0} w_{k,j}^{p}(\tau,m,\epsilon) \label{decomp_w_kp}
\end{equation}
where each function $w_{k,j}^{p}(\tau,m,\epsilon)$ satisfies the next properties.\medskip

\noindent 1) For all $j \geq 0$,
the function $w_{k,j}^{p}(\tau,m,\epsilon)$ belongs to
$\mathrm{Exp}_{(k_{1},\beta,\mu,\alpha);\Theta_{\hat{Q}\mu_{1}}^{p}}^{q}$ and there exists a constant
$0 < K_{1} < 1$ (depending on $q,\alpha,\delta,R_l,\mu,\mu_1,\hat{Q},\tau_0,k_1,k,d_l,\epsilon_0,\Delta_l,r_{Q,R_D}$) such that
\begin{equation}
||w_{k,j}^{p}(\tau,m,\epsilon)||_{(k_{1},\beta,\mu,\alpha);\Theta_{\hat{Q}\mu_{1}}^{p}} \leq
K_{1}^{j}||\frac{\psi_{k}(\tau,m,\epsilon)}{P_{m}(\tau)}||_{(k_{1},\beta,\mu,\alpha);\Theta_{\hat{Q}\mu_{1}}^{p}}
\label{norm_w_kj_on_triangle}
\end{equation}
for all $j \geq 0$, and $0 \leq p \leq \varsigma-1$.\medskip

\noindent 2) For all $j \geq 0$, the function $w_{k,j}^{p}(\tau,m,\epsilon)$ belongs to
$\mathrm{Exp}_{(k_{1},\beta,\mu,\alpha);\bar{D}(0,\check{Q}\mu_{0,j})}^{q}$ and there exists a constant
$0<K_{2}<1$ (depending on $q,\alpha,\delta,R_l,\mu,\epsilon_0,\Delta_l,r_{Q,R_D},\gamma,\check{Q},\mu_0,\tau_0,k,k_1,d_l$) such that
\begin{equation}
||w_{k,j}^{p}(\tau,m,\epsilon)||_{(k_{1},\beta,\mu,\alpha);\bar{D}(0,\check{Q}\mu_{0,j})}
\leq K_{2}^{j}(\frac{1}{q^{\delta}})^{\frac{k \min_{l=1}^{D-1}d_{l}}{2}(j+1)j}
||\frac{\psi_{k}(\tau,m,\epsilon)}{P_{m}(\tau)}||_{(k_{1},\beta,\mu,\alpha);\bar{D}(0,\check{Q}\mu_{0})}
\label{norm_w_kj_on_discs}
\end{equation}
for all $j \geq 0$, all $0 \leq p \leq \varsigma-1$.\medskip

\noindent 3) For all $j \geq 0$, all $0 \leq h \leq j$, the function
$w_{k,j}^{p}(\tau,m,\epsilon)$ belongs to $\mathrm{Exp}_{(k_{1},\beta,\mu,\alpha); P\Omega_{h}^{p}}^{q}$ and there exist
constants $0<K_{3}<1$, $0<K_{4}<1$ and $C_{4}>0$\\ (depending on $q,\alpha,\delta,R_l,\mu,\epsilon_0,\Delta_l,r_{Q,R_D},\gamma_l,\hat{Q},\check{Q},\mu_0,\mu_1,\tau_0,k,k_1,d_l$) such that
\begin{multline}
||w_{k,j}^{p}(\tau,m,\epsilon)||_{(k_{1},\beta,\mu,\alpha);P\Omega_{h}^{p}}
\leq C_{4}K_{3}^{j}K_{4}^{h}(\frac{1}{q^{\delta}})^{\frac{k \min_{l=1}^{D-1}d_{l}}{2}(h+1)h}\\
\times
\max( ||\frac{\psi_{k}(\tau,m,\epsilon)}{P_{m}(\tau)} ||_{(k_{1},\beta,\mu,\alpha);\bar{A}_{\hat{Q}\mu_{1}}},
||\frac{\psi_{k}(\tau,m,\epsilon)}{P_{m}(\tau)} ||_{(k_{1},\beta,\mu,\alpha);P\Omega_{0}^{p}} )
\label{norm_w_kj_on_square_frames}
\end{multline}
for all $j \geq 0$, all $0 \leq h \leq j$ and all $0 \leq p \leq \varsigma-1$.
\end{prop}
\begin{proof} Let $\epsilon\in D(0,\epsilon_0)$. We define the linear operator
$$
\mathcal{H}_{\epsilon}(w(\tau,m)) := \sum_{l=1}^{D-1} \epsilon^{\Delta_{l} - kd_{l}}
\frac{(k\tau^{k})^{d_l}}{P_{m}(\tau)} \frac{1}{(2\pi)^{1/2}}
\int_{-\infty}^{+\infty} C_{l}(m-m_{1},\epsilon)R_{l}(im_{1})w(q^{\delta}\tau,m_{1}) dm_{1}
$$
Let $w_{k,0}^{p}(\tau,m,\epsilon) = \frac{\psi_{k}(\tau,m,\epsilon)}{P_{m}(\tau)}$. Due to the lower bound estimates
(\ref{lower_bounds_Pm}) and the hypothesis that
$\psi_{k}(\tau,m,\epsilon) \in \mathrm{Exp}_{(k_{1},\beta,\mu,\alpha);\mathbb{C}}^{q}$, we deduce that
$w_{k,0}^{p}(\tau,m,\epsilon)$ belongs to the spaces 
$\mathrm{Exp}_{(k_{1},\beta,\mu,\alpha);\Theta_{\hat{Q}\mu_{1}}^{p}}^{q}$,
$\mathrm{Exp}_{(k_{1},\beta,\mu,\alpha);\bar{D}(0,\check{Q}\mu_{0,0})}^{q}$ and
$\mathrm{Exp}_{(k_{1},\beta,\mu,\alpha); P\Omega_{0}}^{q}$. In the next step, we put
$$ w_{k,j}^{p}(\tau,m,\epsilon) := \mathcal{H}_{\epsilon}^{(j)}\left( \frac{\psi_{k}(\tau,m,\epsilon)}{P_{m}(\tau)} \right) $$
where $\mathcal{H}_{\epsilon}^{(j)}$ denotes the map $\mathcal{H}_{\epsilon}$ composed $j$ times with itself. We first
show the estimates (\ref{norm_w_kj_on_triangle}). Due to the estimates (\ref{norm_H_ThetaQmu1}),
and for small enough $r_{Q,R_{D}}>0$ and $\gamma_{l}>0$, $0 \leq l \leq D-1$, we get a constant $0 < K_{1} < 1$ with
\begin{multline}
||w_{k,j+1}^{p}(\tau,m,\epsilon)||_{(k_{1},\beta,\mu,\alpha);\Theta_{\hat{Q}\mu_{1}}^{p}}\\
\leq \left( \sum_{l=1}^{D-1} \epsilon_{0}^{\Delta_{l}-kd_{l}}
\frac{C_{3}k^{d_l}}{C_{P}(r_{Q,R_{D}})^{1/(kd_{D})}(2\pi)^{1/2}}
\sup_{\epsilon \in D(0,\epsilon_{0})}||C_{l}(m,\epsilon)||_{(\beta,\mu)} \right)\\
\times ||w_{k,j}^{p}(\tau,m,\epsilon)||_{(k_{1},\beta,\mu,\alpha);\Theta_{\hat{Q}\mu_{1}}^{p}}\\
\leq K_{1}||w_{k,j}^{p}(\tau,m,\epsilon)||_{(k_{1},\beta,\mu,\alpha);\Theta_{\hat{Q}\mu_{1}}^{p}}
\label{norm_w_kj_on_triangle_recursion}
\end{multline}
for all $j \geq 0$. As a result, we deduce (\ref{norm_w_kj_on_triangle}) by induction on $j$ with the help
of (\ref{norm_w_kj_on_triangle_recursion}). Therefore, we deduce that for all $\epsilon \in D(0,\epsilon_{0})$, the
series $\sum_{j \geq 0} w_{k,j}^{p}(\tau,m,\epsilon)$ converges in the Banach space
$\mathrm{Exp}_{(k_{1},\beta,\mu,\alpha);\Theta_{\hat{Q}\mu_{1}}^{p}}^{q}$ to a function $w_{k}^{p}(\tau,m,\epsilon)$.
By construction, $w_{k}^{p}(\tau,m,\epsilon)$ is equal to
\begin{equation}
w_{k}^{p}(\tau,m,\epsilon) = \sum_{j \geq 0}
\mathcal{H}_{\epsilon}^{(j)}( \frac{\psi_{k}(\tau,m,\epsilon)}{P_{m}(\tau)} )=
(\mathrm{id}-\mathcal{H}_{\epsilon})^{-1}( \frac{\psi_{k}(\tau,m,\epsilon)}{P_{m}(\tau)} )
\end{equation}
where $\mathrm{id}$ denotes the identity map. As a result, $w_{k}^{p}(\tau,m,\epsilon)$ solves the equation
(\ref{q_diff_conv_init_v_prob}) and satisfies $w_{k}^{p}(0,m,\epsilon) \equiv 0$.

We deal now with the estimates (\ref{norm_w_kj_on_discs}).
According to (\ref{norm_H_disc_Qmuj}), for sufficiently small $r_{Q,R_{D}}$ and
$\gamma_{l}$, $0 \leq l \leq D-1$, we get a constant $0<K_{2}<1$ with
\begin{multline}
||w_{k,j+1}^{p}(\tau,m,\epsilon)||_{(k_{1},\beta,\mu,\alpha);D(0,\check{Q}\mu_{0,j+1})}\\
\leq \left( \sum_{l=1}^{D-1} \epsilon_{0}^{\Delta_{l} - kd_{l}}
\frac{C_{2}k^{d_l}}{C_{P}(r_{Q,R_{D}})^{1/(kd_{D})}(2\pi)^{1/2}}
\sup_{\epsilon \in D(0,\epsilon_{0})}||C_{l}(m,\epsilon)||_{(\beta,\mu)} \right. \\
\left. \times |\check{Q}\mu_{0,j+1}|^{kd_{l}}
((q^{\delta}\check{Q}\mu_{0,j+1}+\tau_{0})(\check{Q}\mu_{0,j+1} + \tau_{0}))^{\frac{k_{1}\delta}{2}} \right)\\
||w_{k,j}^{p}(\tau,m,\epsilon)||_{(k_{1},\beta,\mu,\alpha);D(0,\check{Q}\mu_{0,j})}\\
\leq K_{2}(\frac{1}{q^{\delta}})^{(j+1) \min_{l=1}^{D-1} kd_{l}}
||w_{k,j}^{p}(\tau,m,\epsilon)||_{(k_{1},\beta,\mu,\alpha);D(0,\check{Q}\mu_{0,j})}
\label{norm_w_kj_on_discs_recursion}
\end{multline}
for all $j \geq 0$. Hence, we can deduce (\ref{norm_w_kj_on_discs}) by induction on $j$ by means of (\ref{norm_w_kj_on_discs_recursion}). In the last step, we consider (\ref{norm_w_kj_on_square_frames}).
Taking into account the estimates (\ref{norm_H_Pomega_h}), for small enough $r_{Q,R_{D}}$ and
$\gamma_{l}$, $0 \leq l \leq D-1$, we get a constant $0<K_{3.1}<1$ with
\begin{multline}
||w_{k,j+1}^{p}(\tau,m,\epsilon)||_{(k_{1},\beta,\mu,\alpha);P\Omega_{h+1}^{p}}\\
\leq \left( \sum_{l=1}^{D-1} \epsilon_{0}^{\Delta_{l} - kd_{l}}
\frac{C_{2}k^{d_l}}{C_{P}(r_{Q,R_{D}})^{1/(kd_{D})}(2\pi)^{1/2}}
\sup_{\epsilon \in D(0,\epsilon_{0})}||C_{l}(m,\epsilon)||_{(\beta,\mu)} \right. \\
\left. \times \frac{|\hat{Q}\mu_{1,h+1}|^{kd_{l}}}{(1 + \check{Q}\mu_{0,h+1})^{kd_{D}-1}}
((q^{\delta}\hat{Q}\mu_{1,h+1} + \tau_{0})(\hat{Q}\mu_{1,h+1}+\tau_{0}))^{\frac{k_{1}\delta}{2}} \right)
||w_{k,j}^{p}(\tau,m,\epsilon)||_{(k_{1},\beta,\mu,\alpha);P\Omega_{h}^{p}}\\
\leq K_{3.1}\frac{\max_{l=1}^{D-1}|\hat{Q}\mu_{1,h+1}|^{kd_{l}}}{(1 + \check{Q}\mu_{0,h+1})^{kd_{D}-1}}
((q^{\delta}\hat{Q}\mu_{1,h+1} + \tau_{0})(\hat{Q}\mu_{1,h+1}+\tau_{0}))^{\frac{k_{1}\delta}{2}}\\
\times
||w_{k,j}^{p}(\tau,m,\epsilon)||_{(k_{1},\beta,\mu,\alpha);P\Omega_{h}^{p}} \label{norm_w_kj_on_square_frames_recursion}
\end{multline}
for all $h \in \mathbb{Z}$, all $j \geq 0$. From (\ref{norm_w_kj_on_square_frames_recursion}),
using induction on $j$, we deduce that
\begin{multline}
||w_{k,j}^{p}(\tau,m,\epsilon)||_{(k_{1},\beta,\mu,\alpha);P\Omega_{h}^{p}}\\
\leq K_{3.1}^{j} \left( \prod_{b=h-j+1}^{h}
\frac{\max_{l=1}^{D-1}|\hat{Q}\mu_{1,b}|^{kd_{l}}}{(1 + \check{Q}\mu_{0,b})^{kd_{D}-1}}
((q^{\delta}\hat{Q}\mu_{1,b} + \tau_{0})(\hat{Q}\mu_{1,b}+\tau_{0}))^{\frac{k_{1}\delta}{2}} \right)\\
\times || \frac{\psi_{k}(\tau,m,\epsilon)}{P_{m}(\tau)} ||_{(k_{1},\beta,\mu,\alpha);P\Omega_{h-j}^{p}}
\label{norm_w_kj_on_square_frames_first}
\end{multline}
for all $j \geq 0$ and all $h\in\mathbb{Z}$. On the one hand, from the assumption
(\ref{assum_dD_dl_k1_Delta_l}), we get a constant $K_{3.2}>0$ such that
\begin{equation}
\frac{\max_{l=1}^{D-1}|\hat{Q}\mu_{1,b}|^{kd_{l}}}{(1 + \check{Q}\mu_{0,b})^{kd_{D}-1}}
((q^{\delta}\hat{Q}\mu_{1,b} + \tau_{0})(\hat{Q}\mu_{1,b}+\tau_{0}))^{\frac{k_{1}\delta}{2}} \leq
K_{3.2} \label{maj_K3.2}
\end{equation}
for all $b < 0$. Observe there are $j-h-1$ elements in the product in (\ref{norm_w_kj_on_square_frames_first}) of this type. On the other hand, we get a constant $K_{3.3}>0$ such that
\begin{equation}
\frac{\max_{l=1}^{D-1}|\hat{Q}\mu_{1,b}|^{kd_{l}}}{(1 + \check{Q}\mu_{0,b})^{kd_{D}-1}}
((q^{\delta}\hat{Q}\mu_{1,b} + \tau_{0})(\hat{Q}\mu_{1,b}+\tau_{0}))^{\frac{k_{1}\delta}{2}} \leq
K_{3.3}(\frac{1}{q^{\delta}})^{bk \min_{l=1}^{D-1} d_{l}} \label{maj_K3.3}
\end{equation}
for all $b \geq 0$. Observe there are $h+1$ elements of this type in the product on (\ref{norm_w_kj_on_square_frames_first}). Besides, since $h\leq j$, we know that the next inclusion must hold 
$P\Omega_{h-j}^p \subset \bar{A}_{\hat{Q}\mu_{1}}$ if $h-j \leq -1$. In particular, we get from
Proposition~\ref{prop1} that
\begin{multline}
|| \frac{\psi_{k}(\tau,m,\epsilon)}{P_{m}(\tau)} ||_{(k_{1},\beta,\mu,\alpha);P\Omega_{h-j}^{p}}\\
\leq \max(
|| \frac{\psi_{k}(\tau,m,\epsilon)}{P_{m}(\tau)} ||_{(k_{1},\beta,\mu,\alpha);\bar{A}_{\hat{Q}\mu_{1}}},
|| \frac{\psi_{k}(\tau,m,\epsilon)}{P_{m}(\tau)} ||_{(k_{1},\beta,\mu,\alpha);P\Omega_{0}^{p}} )
\label{norm_psi_k_maj_on_square_frames}
\end{multline}
Gathering (\ref{norm_w_kj_on_square_frames_first}), (\ref{maj_K3.2}), (\ref{maj_K3.3}) and
(\ref{norm_psi_k_maj_on_square_frames}) one concludes the estimates
\begin{multline}
||w_{k,j}^{p}(\tau,m,\epsilon)||_{(k_{1},\beta,\mu,\alpha);P\Omega_{h}^{p}}\\
 \leq \frac{K_{3.3}}{K_{3.2}} (K_{3.1}K_{3.2})^{j} (\frac{K_{3.3}}{K_{3.2}})^{h}
(\frac{1}{q^{\delta}})^{\frac{h(h+1)}{2}k \min_{l=1}^{D-1} d_{l} }\\
\times \max( ||\frac{\psi_{k}(\tau,m,\epsilon)}{P_{m}(\tau)} ||_{(k_{1},\beta,\mu,\alpha);\bar{A}_{\hat{Q}\mu_{1}}},
||\frac{\psi_{k}(\tau,m,\epsilon)}{P_{m}(\tau)} ||_{(k_{1},\beta,\mu,\alpha);P\Omega_{0}^{p}} ),
\end{multline}
for some $K_{3.2},K_{3.3}>0$.
Now, we can take the constant $K_{3.1}>0$ small enough such that $0<K_{3.1}K_{3.2}<1$ and
we choose $K_{3.2}>0$ large enough such that $K_{3.3}/K_{3.2}<1$. This yields the estimates 
(\ref{norm_w_kj_on_square_frames}).
\end{proof}

\textbf{Remark:} Holomorphy of $\epsilon\mapsto w_{k,j}^p(\tau,m,\epsilon)$ is deduced from holomorphy of $\Psi_{k}(\tau,m,\epsilon)$ with respect to $\epsilon$ and the construction displayed in the proof of the previous result.

\section{Laplace and Fourier transforms}\label{seccion5}

\begin{defin}\label{defi4}
Let $(\mathbb{E},||.||_{\mathbb{E}})$ be a complex Banach space. Let $k \geq 1$ be an integer and let
$w : S_{d,\delta} \rightarrow \mathbb{E}$ be a holomorphic function on the open unbounded sector
$S_{d,\delta} = \{ \tau \in \mathbb{C}^{\ast} : |d - \mathrm{arg}(\tau) | < \delta \} $, continuous on
$S_{d,\delta} \cup \{ 0 \}$ and such that $w(0)=0$. Assume there exist $C>0$ and $K>0$
such that
\begin{equation}
||w(\tau)||_{\mathbb{E}}
\leq C e^{ K|\tau|^{k} } \label{w_exp_growth_k}
\end{equation}
for all $\tau \in S_{d, \delta}$. Then, the $m_{k}-$Laplace transform of $w$ (following the terminology introduced
in \cite{lama}) in the direction $d$ is defined by
$$ \mathcal{L}_{m_{k}}^{d}(w(\tau))(T) = k \int_{L_{\gamma}}
w(u) e^{ - ( u/T )^{k} } \frac{d u}{u},$$
along a half-line $L_{\gamma} = \mathbb{R}_{+}e^{i\gamma} \subset S_{d,\delta} \cup \{ 0 \}$, where $\gamma$ depends on
$T$ and is chosen in such a way that $\cos(k(\gamma - \mathrm{arg}(T))) \geq \delta_{1} > 0$, for some fixed $\delta_{1}$.
The function $\mathcal{L}^{d}_{m_k}(w(\tau))(T)$ is well defined, holomorphic and bounded in any sector
\begin{equation}\label{e690}
S_{d,\theta,R^{1/k}} = \{ T \in \mathbb{C}^{\ast} : |T| < R^{1/k} \ \ , \ \ |d - \mathrm{arg}(T) | < \theta/2 \},
\end{equation}
where $\frac{\pi}{k} < \theta < \frac{\pi}{k} + 2\delta$ and
$0 < R < \delta_{1}/K$.
\end{defin}
\noindent We now state some elementary properties concerning the $m_{k}-$Laplace transform.\medskip
\begin{prop}\label{prop7}

\noindent 1) If $w$ is an entire function on $\mathbb{C}$, with growth estimates (\ref{w_exp_growth_k}) and
with Taylor expansion $w(\tau) = \sum_{n \geq 0} b_{n}\tau^{n}$, then $\mathcal{L}_{m_{k}}^{d}(w(\tau))(T)$ defines
an analytic function near 0 with convergent Taylor expansion $\sum_{n \geq 0} \Gamma(\frac{n}{k})b_{n}T^{n}$.\\
\noindent 2) Let $w$ a function defined as above in Definition 4. Let $q>1$ and $\delta>0$ be real numbers. Then the next
identity
\begin{equation}
 \mathcal{L}^{d}_{m_k}(w(\tau))(q^{\delta}T) = \mathcal{L}^{d}_{m_k}(w(q^{\delta}\tau))(T)
\end{equation}
holds for all $T \in S_{d,\theta,R_{1}^{1/k}}$ for $0 < R_{1} < \delta_{1}/(Kq^{k\delta})$. Moreover, one has
\begin{multline}
\mathcal{L}^{d}_{m_k}(k \tau^{k}w(\tau))(T) = T^{k+1}\partial_{T}(\mathcal{L}^{d}_{m_k}(w(\tau))(T)),\\
T^{m}\mathcal{L}_{m_k}^{d}(w(\tau))(T) = \mathcal{L}_{m_k}^{d}\left( \frac{\tau^{k}}{\Gamma(\frac{m}{k})}
\int_{0}^{\tau^k}(\tau^{k}-s)^{\frac{m}{k}-1} w(s^{1/k}) \frac{ds}{s} \right)(T)
\end{multline}
for every nonnegative integer $m$, and for all $T \in S_{d,\theta,R^{1/k}}$ with $0 < R < \delta_{1}/K$.
\end{prop}
In the following proposition, we recall some properties of the inverse Fourier transform.
\begin{prop}\label{prop8}
Let $f \in E_{(\beta,\mu)}$ with $\beta > 0$, $\mu > 1$. The inverse Fourier transform of $f$ is defined by
$$ \mathcal{F}^{-1}(f)(x) = \frac{1}{ (2\pi)^{1/2} } \int_{-\infty}^{+\infty} f(m) \exp( ixm ) dm $$
for all $x \in \mathbb{R}$. The function $\mathcal{F}^{-1}(f)$ extends to an analytic function on the strip
\begin{equation}
H_{\beta} = \{ z \in \mathbb{C} / |\mathrm{Im}(z)| < \beta \}. \label{strip_H_beta}
\end{equation}
Let $\phi(m) = im f(m) \in E_{(\beta,\mu - 1)}$. Then, we have
\begin{equation}
\partial_{z} \mathcal{F}^{-1}(f)(z) = \mathcal{F}^{-1}(\phi)(z) \label{dz_fourier}
\end{equation}
for all $z \in H_{\beta}$.\\
Let $g \in E_{(\beta,\mu)}$ and let $\psi(m) = \frac{1}{(2\pi)^{1/2}} f \ast g(m)$ be the convolution product of $f$ and $g$, defined for all $m \in \mathbb{R}$. Then, $\psi \in E_{(\beta,\mu)}$. Moreover, we have
\begin{equation}
\mathcal{F}^{-1}(f)(z)\mathcal{F}^{-1}(g)(z) = \mathcal{F}^{-1}(\psi)(z) \label{prod_fourier}
\end{equation}
for all $z \in H_{\beta}$.
\end{prop}
For a detailed proof of these results, we refer to Proposition 5 and Proposition 7 in~\cite{lama}.

\section{$q-$Gevrey asymptotics of order $1/k$ and a $q-$analog of the Ramis-Sibuya theorem}\label{seccion6}

The results contained in this section can be found in detail in~\cite{ma}, so we omit the proofs and refer to the reader to that previous work for the complete details.

\begin{defin} Let $V$ be a bounded open sector centered at 0 in $\mathbb{C}$. Let $(\mathbb{F},||.||_{\mathbb{F}})$ be a complex Banach space.
Let $q>1$ be a real number and $k \geq 1$ be an integer.
We say that a holomorphic function $f : V \rightarrow \mathbb{F}$ admits a formal power series $\hat{f}(\epsilon) =
\sum_{n \geq 0} f_{n} \epsilon^{n} \in \mathbb{F}[[\epsilon]]$ as its $q-$Gevrey asymptotic expansion of order $1/k$ if for every
open subsector $U \subset V$ with $\bar{U} \subset V$, there exist constants $A,C>0$ such that
$$ ||f(\epsilon) - \sum_{n=0}^{N}f_{n} \epsilon^{n}||_{\mathbb{F}} \leq CA^{N+1}q^{\frac{(N+1)N}{2k}}|\epsilon|^{N+1}$$
for all $\epsilon \in U$, and every $N \geq 0$.
\end{defin}

\begin{lemma}\label{lema2} A holomorphic function $f:V \rightarrow \mathbb{F}$ admits the null formal series $\hat{0} \in \mathbb{F}[[\epsilon]]$
as its $q-$Gevrey asymptotic expansion of order $1/k$ if and only for any open subsector $U \subset V$ with
$\bar{U} \subset V$ there exist two constants $K \in \mathbb{R}$, $M>0$
with
\begin{equation}
||f(\epsilon)||_{\mathbb{F}} \leq M\exp( -\frac{k}{2\log(q)} \log^{2}|\epsilon| )|\epsilon|^{K}
\end{equation}
for all $\epsilon \in U$.
\end{lemma}

The classical Ramis-Sibuya theorem is a cohomological criterion that ensures $k-$summability of a given formal series (see
\cite{ba}, Section 4.4 Proposition 2 or \cite{hssi}, Lemma XI-2-6). We recall a version of this theorem in the
framework of $q-$Gevrey asymptotics of order $1/k$.

\begin{defin}\label{goodcovering} Let $\varsigma \geq 2$ be an integer. For all $0 \leq p \leq \varsigma-1$, we consider an open sector
$\mathcal{E}_{p}$ centered at $0$, with radius $\epsilon_{0}$ such that
$\mathcal{E}_{p} \cap \mathcal{E}_{p+1} \neq \emptyset$, for all
$0 \leq p \leq \varsigma-1$ (with the convention that $\mathcal{E}_{\varsigma} = \mathcal{E}_{0})$. Moreover, we assume that
the intersection of any three different elements in $(\mathcal{E}_{p})_{0 \leq p \leq \varsigma}$ is empty and that
$\cup_{p=0}^{\varsigma - 1} \mathcal{E}_{p} = \mathcal{U} \setminus \{ 0 \}$,
where $\mathcal{U}$ is some neighborhood of 0 in $\mathbb{C}$. Such a set of sectors
$\{ \mathcal{E}_{p} \}_{0 \leq p \leq \varsigma - 1}$ is called a good covering in $\mathbb{C}^{\ast}$.
\end{defin}

\noindent {\bf Theorem (q-RS)} {\it Let $(\mathbb{F},||.||_{\mathbb{F}})$ be a Banach space and
$\{ \mathcal{E}_{p} \}_{0 \leq p \leq \varsigma-1}$ be a good covering in $\mathbb{C}^{\ast}$. For all $0 \leq p \leq \varsigma-1$,
let $G_{p}(\epsilon)$ be a holomorphic function from $\mathcal{E}_{p}$ into $\mathbb{F}$ and let the cocycle
$\Delta_{p}(\epsilon) = G_{p+1}(\epsilon) - G_{p}(\epsilon)$ be a holomorphic function from $Z_{p} =
\mathcal{E}_{p+1} \cap \mathcal{E}_{p}$ into $\mathbb{F}$ (with the convention that $\mathcal{E}_{\varsigma}=\mathcal{E}_{0}$
and $G_{\varsigma}=G_{0}$). We make the following further assumptions:\medskip

\noindent {\bf 1)} The functions $G_{p}(\epsilon)$ are bounded as $\epsilon$ tends to 0 on $\mathcal{E}_{p}$, for all
$0 \leq p \leq \varsigma-1$.\medskip

\noindent {\bf 2)} The function $\Delta_{p}(\epsilon)$ is $q-$exponentially flat of order $k$ on $Z_{p}$ for all
$0 \leq p \leq \varsigma-1$,
meaning that there exist two constants $C_{p}^{1} \in \mathbb{R}$ and $C_{p}^{2}>0$ with
\begin{equation}
|| \Delta_{p}(\epsilon) ||_{\mathbb{F}} \leq C_{p}^{2}|\epsilon|^{C_{p}^{1}}\exp( -\frac{k}{2\log(q)}\log^{2}|\epsilon| )
\label{norm_Delta_q_exp_flat}
\end{equation}
for all $\epsilon \in Z_{p}$, all $0 \leq p \leq \varsigma-1$.\medskip

Then, there exists a formal power series $\hat{G}(\epsilon) \in \mathbb{F}[[\epsilon]]$ which is the common
$q-$Gevrey asymptotic expansion of order $1/k$ of the function $G_{p}(\epsilon)$ on $\mathcal{E}_{p}$, for all
$0 \leq p \leq \varsigma-1$.}

\section{Analytic solutions of a differential and $q-$difference
initial value problem with complex parameter and analytic forcing term. Study of their parametric asymptotic expansions}\label{seccion7}

Let $k \geq 1$ and $D \geq 2$ be integers. Let $\delta,k_{1}>0$, $q>1$ and $\alpha$ be real numbers and for $1 \leq l \leq D$, let
$d_{l},\Delta_{l} \geq 0$ be nonnegative integers. We make the assumption that
\begin{equation}
kd_{D} - 1 \geq k_{1}\delta + kd_{l} \ \ , \ \ \Delta_{l} \geq kd_{l} \label{cond_coeff_k_dl_Deltal}
\end{equation}
for all $1 \leq l \leq D-1$. Let $Q(X),R_{l}(X) \in \mathbb{C}[X]$, $1 \leq l \leq D$, be polynomials such that
\begin{equation}
\mathrm{deg}(Q) \geq \mathrm{deg}(R_D) \geq \mathrm{deg}(R_{l}) \ \ , \ \ Q(im) \neq 0 \ \ , \ \
R_{D}(im) \neq 0, \label{cond_polynom_Q_Rl}
\end{equation}
for all $m \in \mathbb{R}$, all $1 \leq l\le  D-1$. We make the additional assumption that there exists a bounded sectorial
annulus
$$ A_{Q,R_{D}} = \{ z \in \mathbb{C} / r_{Q,R_{D}} \leq |z| \leq r_{Q,R_{D}}^{1} \ \ , \ \
|\mathrm{arg}(z) - d_{Q,R_{D}}| \leq \eta_{Q,R_{D}} \} $$
with direction $d_{Q,R_{D}} \in \mathbb{R}$, aperture $2\eta_{Q,R_{D}}>0$ for some radius
$r_{Q,R_{D}},r_{Q,R_{D}}^{1}>0$ such that (\ref{quotient_Q_RD_in_Ann}) holds for all $m \in \mathbb{R}$. Moreover, we assume that the roots $q_l(m)$ of the polynomial $P_m(\tau)=Q(im)-R_D(im)(k\tau^k)^{d_D}$ satisfy conditions (\ref{root_cond_1_q_lm}) and (\ref{root_cond_2_q_lm}).

\begin{defin}\label{defi7} Let $\{ \mathcal{E}_{p} \}_{0 \leq p \leq \varsigma - 1}$ be a good covering in $\mathbb{C}^{\ast}$. Let
$\mathcal{T}$ be an open bounded sector centered at 0 with radius $r_{\mathcal{T}}$ and consider a family of open sectors
$$ U_{\mathfrak{d}_{p},\theta,\epsilon_{0}r_{\mathcal{T}}} =
\{ T \in \mathbb{C}^{\ast} / |T| < \epsilon_{0}r_{\mathcal{T}} \ \ , \ \ |\mathfrak{d}_{p} - \mathrm{arg}(T)| < \theta/2 \} $$
with aperture $\pi/k<\theta < \pi/k + \mathrm{Ap}(U_{\mathfrak{d}_{p}})$, where $\mathfrak{d}_{p} \in \mathbb{R}$ and
$U_{\mathfrak{d}_{p}}$ are unbounded sectors for all $0 \leq p \leq \varsigma-1$,with bisecting direction $\mathfrak{d}_p$ for all $0\le p\le \varsigma-1$.

For all $0 \leq p \leq \varsigma - 1$, for all $t \in \mathcal{T}$, all $\epsilon \in \mathcal{E}_{p}$, we assume that $\epsilon t \in U_{\mathfrak{d}_{p},\theta,\epsilon_{0}r_{\mathcal{T}}}$. Under the previous conditions, we say that the family
$\{ (U_{\mathfrak{d}_{p},\theta,\epsilon_{0}r_{\mathcal{T}}})_{0 \leq p \leq \varsigma-1},\mathcal{T} \}$
is associated to the good covering $\{ \mathcal{E}_{p} \}_{0 \leq p \leq \varsigma - 1}$.
\end{defin}

Let $\mu_0,\mu_1>0$ with $\mu_1>\mu_0$ such that (\ref{constraint_mu1_mu0}) holds.

We consider a family
$\{ (U_{\mathfrak{d}_{p},\theta,\epsilon_{0}r_{\mathcal{T}}})_{0 \leq p \leq \varsigma-1},\mathcal{T} \}$ associated to a given good covering $\{ \mathcal{E}_{p} \}_{0 \leq p \leq \varsigma - 1}$.
We study the following linear initial value problem
\begin{multline}
Q(\partial_{z})u(t,z,\epsilon) = R_{D}(\partial_{z})\epsilon^{kd_{D}}(t^{k+1}\partial_{t})^{d_{D}}u(t,z,\epsilon)\\
+ \sum_{l=1}^{D-1} \epsilon^{\Delta_{l}}(t^{k+1}\partial_{t})^{d_{l}}c_{l}(z,\epsilon)R_{l}(\partial_{z})u(q^{\delta}t,z,\epsilon)
+ f(t,z,\epsilon) \label{main_q_diff_diff_first}
\end{multline}
for initial data $u(0,z,\epsilon) \equiv 0$. The coefficients $c_{l}(z,\epsilon)$ are constructed as follows.
For every $1\le l\le D-1$, a map $m \mapsto C_{l}(m,\epsilon)$, that belongs to the space
$E_{(\beta,\mu)}$ for some $\beta>0$ and $\mu > \mathrm{deg}(R_{D}) + 1$ and depends holomorphically on
$\epsilon$ on $D(0,\epsilon_{0})$. We denote
\begin{equation}
c_{l}(z,\epsilon) = \mathcal{F}^{-1}(m \mapsto C_{l}(m,\epsilon))(z), \label{defin_coeff_c_l_z}
\end{equation}
which represents a bounded holomorphic function on the
strip $H_{\beta'} = \{ z \in \mathbb{C} / |\mathrm{Im}(z)|< \beta' \}$ for any real number $0 < \beta' < \beta$ (
where $\mathcal{F}^{-1}$ denotes the inverse Fourier transform defined in Proposition~\ref{prop8}). The forcing term
$f(t,z,\epsilon)$ is built as follows. Let $F(t,m,\epsilon)$ be the $m_{k}-$Laplace transform
\begin{equation}
F(t,m,\epsilon) = k \int_{L_{\gamma}} \psi_{k}(u,m,\epsilon) \exp(-(\frac{u}{\epsilon t})^{k}) \frac{du}{u} 
\end{equation}
along any halfine $L_{\gamma} = \mathbb{R}_{+}e^{\sqrt{-1}\gamma}$, where
$\psi_{k}(\tau,m,\epsilon)$ is a function that belongs to the Banach space
$\mathrm{Exp}_{(k_{1},\beta,\mu,\alpha);\mathbb{C}}^{q}$ (see Definition~\ref{defi2}) for all $\epsilon \in D(0,\epsilon_{0})$, which is holomorphic with respect to $\epsilon$ in $D(0,\epsilon_0)$.
By construction,
$(t,\epsilon) \mapsto F(t,m,\epsilon)$ defines a bounded holomorphic function w.r.t $(t,\epsilon)$ near
the origin in $\mathbb{C}^{2}$, with values in $E_{(\beta,\mu)}$. In view of the results stated in~\cite{ra}, the fact that $\psi_{k}(\tau,m,\epsilon)$ defines an entire function with values in $E_{(\beta,\mu)}$ with
$q-$exponential growth of order $k_{1}$, entails that its Taylor expansion at $\tau=0$, denoted by
$\sum_{n \geq 1} \psi_{k,n}(m,\epsilon) \tau^{n}$, is such that 
$$  ||\psi_{k,n}(m,\epsilon)||_{(\beta,\mu)} \leq Cq^{-\frac{n(n+1)}{2k_{1}}}A^{n} $$
for well chosen constants $C,A>0$, for all $n \geq 1$, all $\epsilon \in D(0,\epsilon_{0})$. As a result, we can write
$$ F(t,m,\epsilon) = \sum_{n \geq 1} \Gamma(\frac{n}{k}) \psi_{k,n}(m,\epsilon) (\epsilon t)^{n} $$
which defines a $E_{(\beta,\mu)}-$valued analytic function w.r.t $t$ and $\epsilon$ near the origin (actually an
entire function w.r.t $t$). We define the forcing term as the Fourier inverse transform
\begin{equation}
f(t,z,\epsilon) = \mathcal{F}^{-1}(m \mapsto F(t,m,\epsilon))(z) = \frac{k}{(2\pi)^{1/2}}\int_{-\infty}^{+\infty}
\int_{L_{\gamma}} \psi_{k}(u,m,\epsilon) \exp(-(\frac{u}{\epsilon t})^{k}) e^{izm} \frac{du}{u} dm \label{defin_forcing_f}
\end{equation}
which defines a bounded holomorphic function w.r.t $(t,\epsilon)$ near the origin in $\mathbb{C}^{2}$ and w.r.t
$z$ on any strip $H_{\beta'}$ for $0 < \beta' < \beta$.

The next Theorem states the first main result in this section. We construct a family of actual holomorphic solutions of
(\ref{main_q_diff_diff_first}) with vanishing initial data, defined on the sectors $\mathcal{E}_{p}$ w.r.t $\epsilon$.
Moreover, we give bound estimates for the difference between any two neighboring solutions on $\mathcal{E}_{p} \cap \mathcal{E}_{p+1}$ and show that it is at most of $q-$exponential decay of some order
$\kappa>0$.

\begin{theo}\label{teo1} Under the hypotheses stated in this section, there exist $r_{Q,R_{D}}>0$ and $\gamma_{l}>0$ such that if
\begin{equation}
\sup_{\epsilon \in D(0,\epsilon_{0})} ||C_{l}(m,\epsilon)||_{(\beta,\mu)} \leq \gamma_{l} \label{norm_C_l_small}
\end{equation}
for all $1 \leq l \leq D-1$, then for every $0 \leq p \leq \varsigma-1$  there exists a solution
$u_{p}(t,z,\epsilon)$ of (\ref{main_q_diff_diff_first}) with $u_{p}(0,z,\epsilon) \equiv 0$, which defines
a bounded holomorphic function on the domain $\mathcal{T} \times H_{\beta'} \times \mathcal{E}_{p}$ for any given
$0 < \beta' < \beta$ with small enough $r_{\mathcal{T}}$. Moreover, for any given
$0 < \kappa < k \min_{l=1}^{D-1} d_{l}$, one can choose $M_{p}>0$ and $K_{p} \in \mathbb{R}$
(not dependending on $\epsilon$) such that
\begin{equation}
\sup_{t \in \mathcal{T}, z \in H_{\beta'}} |u_{p+1}(t,z,\epsilon) - u_{p}(t,z,\epsilon)| \leq
M_{p}\exp(-\frac{\kappa}{2\log(q)}\log^{2}(|\epsilon|)) |\epsilon|^{K_{p}}
\label{q_exp_small_difference_u_p}
\end{equation}
for all $\epsilon \in \mathcal{E}_{p} \cap \mathcal{E}_{p+1}$, for all $0 \leq p \leq \varsigma-1$ (where
by convention $u_{\varsigma}:=u_{0}$).
\end{theo}
\begin{proof} For the proof, we adopt the notation stated in Section~\ref{seccion42}. Under the assumptions (\ref{cond_coeff_k_dl_Deltal}), (\ref{cond_polynom_Q_Rl}) and bearing in mind that the family of sectors
$\{ (U_{\mathfrak{d}_{p},\theta,\epsilon_{0}r_{\mathcal{T}}})_{0 \leq p \leq \varsigma-1},\mathcal{T} \}$
is associated to the good covering $\{ \mathcal{E}_{p} \}_{0 \leq p \leq \varsigma - 1}$, we get from Proposition~\ref{prop6} the
existence of $r_{Q,R_{D}}>0$ and $\gamma_{l}>0$ such that if (\ref{norm_C_l_small}) holds for all $1 \leq l \leq D-1$,
the equation (\ref{q_diff_conv_init_v_prob}), with vanishing initial data, admits a solution $w_{k}^{p}(\tau,m,\epsilon)$
belonging to the Banach space $\mathrm{Exp}_{(k_{1},\beta,\mu,\alpha);\Theta_{\hat{Q}\mu_{1}}^p}^{q}$ for each
$0 \leq p \leq \varsigma-1$. We apply Fourier and $m_{k}-$Laplace transform of
$w_{k}^{p}(\tau,m,\epsilon)$ in direction $\mathfrak{d}_{p}$ in order to define
\begin{equation}
u_{p}(t,z,\epsilon) = \frac{k}{(2\pi)^{1/2}} \int_{-\infty}^{+\infty} \int_{L_{\gamma_{p}}}
w_{k}^{p}(u,m,\epsilon) \exp( -(\frac{u}{\epsilon t})^{k} ) e^{izm} \frac{du}{u} dm,
\end{equation}
where $L_{\gamma_{p}} = \mathbb{R}_{+}e^{\sqrt{-1}\gamma_{p}} \subset U_{\mathfrak{d}_{p}} \cup \{ 0 \}$
which may depend on $(\epsilon,t)$. By construction, $u_{p}(t,z,\epsilon)$ represents a bounded holomorphic function on
$\mathcal{T} \times H_{\beta'} \times \mathcal{E}_{p}$ provided that $r_{\mathcal{T}}>0$ is small enough. Moreover,
$u_{p}(0,z,\epsilon) \equiv 0$. Using the properties of Fourier inverse transform described in Proposition~\ref{prop8} and $m_{k}-$Laplace transform from Proposition~\ref{prop7}, we deduce that $u_{p}(t,z,\epsilon)$ solves the initial value
problem (\ref{main_q_diff_diff_first}) on $\mathcal{T} \times H_{\beta'} \times \mathcal{E}_{p}$.

We now proceed to the proof of the estimates (\ref{q_exp_small_difference_u_p}). We use the decomposition
(\ref{decomp_w_kp}) in order to write the function $u_{p}(t,z,\epsilon)$ as a sum
\begin{equation}
u_{p}(t,z,\epsilon) = \sum_{j \geq 0} u_{p,j}(t,z,\epsilon)
\end{equation}
where 
$$ u_{p,j}(t,z,\epsilon) = \frac{k}{(2\pi)^{1/2}} \int_{-\infty}^{+\infty} \int_{L_{\gamma_{p}}}
w_{k,j}^{p}(u,m,\epsilon) \exp( -(\frac{u}{\epsilon t})^{k} ) e^{izm} \frac{du}{u} dm. $$
Let $j \geq 0$ and $0 \leq p \leq \varsigma-1$. By construction of the domains $R\Omega_{j}^{p}$ and
$\Theta_{\hat{Q}\mu_{1}}^{p}$ together with the use of a path deformation argument, which is available due to the properties of $w_{k,j}^{p}$ described in Proposition~\ref{prop6}, we can write each difference
$u_{p+1,j}(t,z,\epsilon) - u_{p,j}(t,z,\epsilon)$ as a sum of integrals along halflines and arcs of circles:
\begin{multline}
 u_{p+1,j}(t,z,\epsilon) - u_{p,j}(t,z,\epsilon) \\=
 \frac{k}{(2\pi)^{1/2}} \int_{-\infty}^{+\infty} \int_{C_{\gamma_{p},\gamma_{p+1},\check{Q}\mu_{0,j}}}
 w_{k,j}^{p}(u,m,\epsilon) \exp( -(\frac{u}{\epsilon t})^{k} ) e^{izm} \frac{du}{u} dm\\
 + \left(\sum_{h=0}^{j} \frac{k}{(2\pi)^{1/2}} \int_{-\infty}^{+\infty}
 \int_{L_{\gamma_{p+1},\check{Q}\mu_{0,h},\hat{Q}\mu_{1,h}}}
 w_{k,j}^{p+1}(u,m,\epsilon) \exp( -(\frac{u}{\epsilon t})^{k} ) e^{izm} \frac{du}{u} dm \right)\\
 + \frac{k}{(2\pi)^{1/2}} \int_{-\infty}^{+\infty}
 \int_{L_{\gamma_{p+1},\hat{Q}\mu_{1},\infty}}
 w_{k,j}^{p+1}(u,m,\epsilon) \exp( -(\frac{u}{\epsilon t})^{k} ) e^{izm} \frac{du}{u} dm\\
 - \left( \sum_{h=0}^{j} \frac{k}{(2\pi)^{1/2}} \int_{-\infty}^{+\infty}
 \int_{L_{\gamma_{p},\check{Q}\mu_{0,h},\hat{Q}\mu_{1,h}}}
 w_{k,j}^{p}(u,m,\epsilon) \exp( -(\frac{u}{\epsilon t})^{k} ) e^{izm} \frac{du}{u} dm \right)\\
 - \frac{k}{(2\pi)^{1/2}} \int_{-\infty}^{+\infty}
 \int_{L_{\gamma_{p},\hat{Q}\mu_{1},\infty}}
 w_{k,j}^{p}(u,m,\epsilon) \exp( -(\frac{u}{\epsilon t})^{k} ) e^{izm} \frac{du}{u} dm
 \label{path_deform_decomp_difference_u_p_j}
\end{multline}
where $C_{\gamma_{p},\gamma_{p+1},\check{Q}\mu_{0,j}}$ is an arc of circle with radius
$\check{Q}\mu_{0,j}$ connecting $\check{Q}\mu_{0,j}e^{\sqrt{-1}\gamma_{p}}$ and
$\check{Q}\mu_{0,j}e^{\sqrt{-1}\gamma_{p+1}}$ with a well chosen orientation and the other paths are segments
$$ L_{\gamma_{a},\check{Q}\mu_{0,h},\hat{Q}\mu_{1,h}} = [\check{Q}\mu_{0,h},\hat{Q}\mu_{1,h}]e^{\sqrt{-1}\gamma_{a}}
\ \ , \ \ L_{\gamma_{a},\hat{Q}\mu_{1},\infty} = [\hat{Q}\mu_{1},\infty)e^{\sqrt{-1}\gamma_{a}} $$
for $a=p,p+1$. Observe that all the paths $L_{\gamma_{a},\check{Q}\mu_{0,h},\hat{Q}\mu_{1,h}}$ are concatenated, due to (\ref{constraint_mu1_mu0}).

\begin{figure}[h]
	\centering
		\includegraphics[width=.5\textwidth]{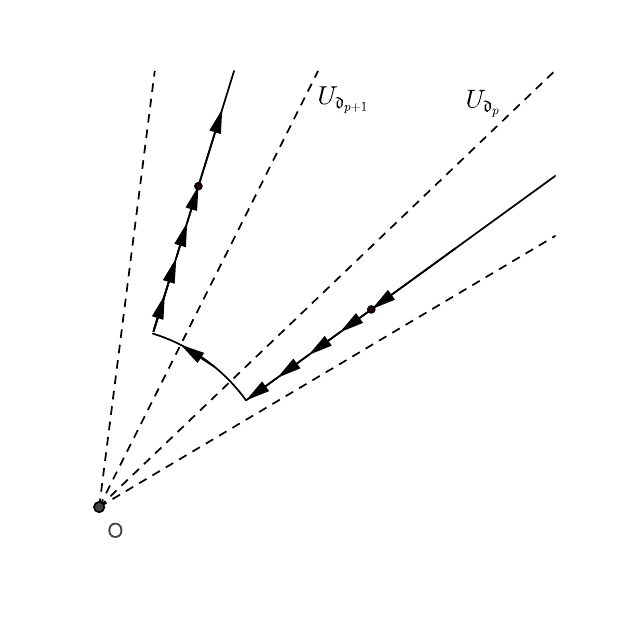}
	\caption{Splitting of the integration path in the proof of Theorem~\ref{teo1}}
	\label{figure3}
\end{figure}

We provide estimates for the first integral
$$
I_{1} = \left| \frac{k}{(2\pi)^{1/2}} \int_{-\infty}^{+\infty} \int_{C_{\gamma_{p},\gamma_{p+1},\check{Q}\mu_{0,j}}}
 w_{k,j}^{p}(u,m,\epsilon) \exp( -(\frac{u}{\epsilon t})^{k} ) e^{izm} \frac{du}{u} dm \right|.
$$
The arc of circle $C_{\gamma_{p},\gamma_{p+1},\check{Q}\mu_{0,j}}$ is chosen in such a way that there exists
$\delta_{1}>0$ with $\cos(k(\theta - \mathrm{arg}(\epsilon t))) \geq \delta_{1}$ for all
$\theta \in [\gamma_{p},\gamma_{p+1}]$, all $t \in \mathcal{T}$, all
$\epsilon \in \mathcal{E}_{p} \cap \mathcal{E}_{p+1}$. From the estimates (\ref{norm_w_kj_on_discs}), we deduce the
existence of a constant $K_{2}>0$ and a constant $M_{1}>0$ (depending on $q,\alpha,\delta, R_l,\mu,\epsilon_0,\Delta_l, r_{Q,R_D},\gamma_l,\check{Q},\mu_0,\tau_0,k,k_1,d_l,\beta,\beta'$) such that
\begin{multline}
I_{1} \leq  \frac{k}{(2\pi)^{1/2}} \int_{-\infty}^{+\infty} \left|
\int_{\gamma_{p}}^{\gamma_{p+1}} |w_{k,j}^{p}(\check{Q}\mu_{0,j}e^{i\theta},m,\epsilon)| \right.\\
\left. \times \exp( -\frac{\cos( k( \theta - \mathrm{arg}(\epsilon t)) )}{|\epsilon t|^{k}}(\check{Q}\mu_{0,j})^{k} )
e^{-m\mathrm{Im}(z)} d\theta \right| dm\\
\leq \frac{k}{(2\pi)^{1/2}}
||\frac{\psi_{k}(\tau,m,\epsilon)}{P_{m}(\tau)}||_{(k_{1},\beta,\mu,\alpha);\bar{D}(0,\check{Q}\mu_{0})}
\int_{-\infty}^{+\infty} e^{-(\beta - \beta')|m|} dm |\gamma_{p+1} - \gamma_{p}|\\
\times |\check{Q}\mu_{0,j}| \exp( \frac{k_{1}}{2}
\frac{\log^{2}(|\check{Q}\mu_{0,j}|+\tau_{0})}{\log(q)} + \alpha\log(|\check{Q}\mu_{0,j}|+\tau_{0}) )\\
\times K_{2}^{j}(\frac{1}{q^{\delta}})^{\frac{k \min_{l=1}^{D-1}d_{l}}{2}(j+1)j}
\exp( - \frac{\delta_{1}}{|\epsilon t|^{k}} (\check{Q}\mu_{0,j})^{k} )\\
\leq M_{1} K_{2}^{j}(\frac{1}{q^{\delta}})^{\frac{k \min_{l=1}^{D-1}d_{l}}{2}(j+1)j}
\exp( - \frac{\delta_{1}(\check{Q}\mu_{0})^{k}}{(r_{\mathcal{T}})^{k}} (\frac{1}{q^{\delta k}})^{j}
\frac{1}{|\epsilon|^{k}} ) \label{I_1_exp_small_order_k_on circle} 
\end{multline}
for all $t \in \mathcal{T}$, all $z \in H_{\beta'}$ and $\epsilon \in \mathcal{E}_{p} \cap \mathcal{E}_{p+1}$.
For $a=p,p+1$ and $h=0,1,\ldots,j$, we give bounds for the integrals
$$ I_{2} = \left| \frac{k}{(2\pi)^{1/2}} \int_{-\infty}^{+\infty}
 \int_{L_{\gamma_{a},\check{Q}\mu_{0,h},\hat{Q}\mu_{1,h}}}
 w_{k,j}^{a}(u,m,\epsilon) \exp( -(\frac{u}{\epsilon t})^{k} ) e^{izm} \frac{du}{u} dm \right|.
$$
We choose the direction $\gamma_{a}$ in a suitable way in order that
$\cos( k(\gamma_{a} - \mathrm{arg}(\epsilon t))) \geq \delta_{1}$ for all $t \in \mathcal{T}$, all
$\epsilon \in \mathcal{E}_{p} \cap \mathcal{E}_{p+1}$. We make use of
(\ref{norm_w_kj_on_square_frames}) which yields constants $0 < K_{3} < 1$ and $C_{4},K_{4}>0$ and a constant
$M_{2}>0$ (depending on $q,\alpha,\delta, R_l,\mu,\epsilon_0,\Delta_l, r_{Q,R_D},\gamma_l,\check{Q},\hat{Q},\mu_0,\mu_1,\tau_0,k,k_1,d_l,\beta,\beta',\psi_k$) with
\begin{multline}
I_{2} \leq \frac{k}{(2\pi)^{1/2}} \int_{-\infty}^{+\infty}
\int_{\check{Q}\mu_{0,h}}^{\hat{Q}\mu_{1,h}}
|w_{k,j}^{a}(re^{\sqrt{-1}\gamma_{a}},m,\epsilon)|\\
\times 
\exp( -\frac{\cos(k(\gamma_{a} - \mathrm{arg}(\epsilon t)))}{|\epsilon t|^{k}} r^{k} )
e^{-m \mathrm{Im}(z)} \frac{dr}{r} dm\\
\leq \frac{C_{4}k}{(2\pi)^{1/2}}\int_{-\infty}^{+\infty} e^{-(\beta - \beta')|m|} dm
\max( ||\frac{\psi_{k}(\tau,m,\epsilon)}{P_{m}(\tau)} ||_{(k_{1},\beta,\mu,\alpha);\bar{A}_{\hat{Q}\mu_{1}}},
||\frac{\psi_{k}(\tau,m,\epsilon)}{P_{m}(\tau)} ||_{(k_{1},\beta,\mu,\alpha);P\Omega_{0}^{p}} )\\
\times (\hat{Q}\mu_{1} - \check{Q}\mu_{0})(\frac{1}{q^{\delta}})^{h}
\sup_{r \in [0,\hat{Q}\mu_{1}]} \exp( \frac{k_{1}}{2}
\frac{\log^{2}(r+\tau_{0})}{\log(q)} + \alpha\log(r+\tau_{0}) )\\
\times
K_{3}^{j}K_{4}^{h}(\frac{1}{q^{\delta}})^{\frac{k \min_{l=1}^{D-1}d_{l}}{2}(h+1)h}
\exp( -\frac{\delta_{1}}{|\epsilon t|^{k}}(\check{Q}\mu_{0,h})^{k} )\\
\leq M_{2} K_{3}^{j}K_{4}^{h}(\frac{1}{q^{\delta}})^{\frac{k \min_{l=1}^{D-1}d_{l}}{2}(h+1)h}
\exp( - \frac{\delta_{1}(\check{Q}\mu_{0})^{k}}{(r_{\mathcal{T}})^{k}} (\frac{1}{q^{\delta k}})^{h}
\frac{1}{|\epsilon|^{k}} ) \label{I_2_exp_small_order_k_on_segment_h}
\end{multline}
for all $t \in \mathcal{T}$, all $z \in H_{\beta'}$ and $\epsilon \in \mathcal{E}_{p} \cap \mathcal{E}_{p+1}$.
Again for $a=p,p+1$, we deal with the third kind of integral
$$ I_{3} = \left| \frac{k}{(2\pi)^{1/2}} \int_{-\infty}^{+\infty}
 \int_{L_{\gamma_{a},\hat{Q}\mu_{1},\infty}}
 w_{k,j}^{a}(u,m,\epsilon) \exp( -(\frac{u}{\epsilon t})^{k} ) e^{izm} \frac{du}{u} dm \right|.
$$
We set the direction $\gamma_{p+1}$ provided the adecquate sense in order that
$\cos( k(\gamma_{a} - \mathrm{arg}(\epsilon t))) \geq \delta_{1}$ for all $t \in \mathcal{T}$, all
$\epsilon \in \mathcal{E}_{p} \cap \mathcal{E}_{p+1}$. Bearing in mind (\ref{norm_w_kj_on_triangle}), we get
a constant $0<K_{1}<1$ and $M_{3}>0$ (depending on $q,\alpha,\delta, R_l,\mu,\mu_1,\hat{Q},,\tau_0,k_1,k,d_l,\epsilon_0,\Delta_l, r_{Q,R_D},\gamma_l,\beta,\beta',\psi_k$) with
\begin{multline}
I_{3} \leq  \frac{k}{(2\pi)^{1/2}} \int_{-\infty}^{+\infty}
\int_{\hat{Q}\mu_{1}}^{+\infty}
|w_{k,j}^{a}(re^{\sqrt{-1}\gamma_{a}},m,\epsilon)|\\
\times 
\exp( -\frac{\cos(k(\gamma_{a} - \mathrm{arg}(\epsilon t)))}{|\epsilon t|^{k}} r^{k} )
e^{-m \mathrm{Im}(z)} \frac{dr}{r} dm\\
\leq \frac{k}{(2\pi)^{1/2}}
||\frac{\psi_{k}(\tau,m,\epsilon)}{P_{m}(\tau)}||_{(k_{1},\beta,\mu,\alpha);\Theta_{\hat{Q}\mu_{1}}^{p}}
\int_{-\infty}^{+\infty} e^{-(\beta - \beta')|m|} dm\\
\times
\int_{\hat{Q}\mu_{1}}^{+\infty} \exp( \frac{k_{1}}{2}
\frac{\log^{2}(r+\tau_{0})}{\log(q)} + \alpha\log(r+\tau_{0}) )
\exp( -\frac{\delta_{1}}{2(\epsilon_{0}r_{\mathcal{T}})}r^{k} ) dr K_{1}^{j}
\exp( -\frac{\delta_{1}}{2|\epsilon t|^{k}}(\hat{Q}\mu_{1})^{k} )\\
\leq M_{3}K_{1}^{j} \exp( - \frac{\delta_{1}(\hat{Q}\mu_{1})^{k}}{2(r_{\mathcal{T}})^{k}} \frac{1}{|\epsilon|^{k}} )
\label{I_3_exp_small_order_k_on_halfline}
\end{multline}
for all $t \in \mathcal{T}$, all $z \in H_{\beta'}$ and $\epsilon \in \mathcal{E}_{p} \cap \mathcal{E}_{p+1}$.

From (\ref{path_deform_decomp_difference_u_p_j}) and gathering the bounds estimates
(\ref{I_1_exp_small_order_k_on circle}), (\ref{I_2_exp_small_order_k_on_segment_h}) and
(\ref{I_3_exp_small_order_k_on_halfline}), we deduce that
\begin{equation}
|u_{p+1}(t,z,\epsilon) - u_{p}(t,z,\epsilon)| \leq
\sum_{j \geq 0} |u_{p+1,j}(t,z,\epsilon) - u_{p,j}(t,z,\epsilon)| \leq J_{1}(\epsilon) + J_{2}(\epsilon) + J_{3}(\epsilon)
\label{maj_difference_u_p_with_J1_J2_J3}
\end{equation}
where
\begin{multline}
J_{1}(\epsilon) = M_{1} \sum_{j \geq 0} K_{2}^{j}(\frac{1}{q^{\delta}})^{\frac{k \min_{l=1}^{D-1}d_{l}}{2}(j+1)j}
\exp( - \frac{\delta_{1}(\check{Q}\mu_{0})^{k}}{(r_{\mathcal{T}})^{k}} (\frac{1}{q^{\delta k}})^{j}
\frac{1}{|\epsilon|^{k}} ),\\
J_{2}(\epsilon) = 2M_{2} \sum_{j \geq 0} \sum_{h=0}^{j}K_{3}^{j}K_{4}^{h}(\frac{1}{q^{\delta}})^{\frac{k \min_{l=1}^{D-1}d_{l}}{2}(h+1)h}
\exp( - \frac{\delta_{1}(\check{Q}\mu_{0})^{k}}{(r_{\mathcal{T}})^{k}} (\frac{1}{q^{\delta k}})^{h}
\frac{1}{|\epsilon|^{k}} ),\\
J_{3}(\epsilon) = 2M_{3} \sum_{j \geq 0}
K_{1}^{j} \exp( - \frac{\delta_{1}(\hat{Q}\mu_{1})^{k}}{2(r_{\mathcal{T}})^{k}} \frac{1}{|\epsilon|^{k}} )
\end{multline}
for all $t \in \mathcal{T}$, all $z \in H_{\beta'}$ and $\epsilon \in \mathcal{E}_{p} \cap \mathcal{E}_{p+1}$.

In order to provide upper estimates for the functions $J_{l}(\epsilon)$ as $\epsilon$ tends to 0, we will need the
following lemma. This lemma is a modified version of the one given by the authors in~\cite{lama2}.
\begin{lemma}\label{lema3} Let $\Delta_{1},\Delta_{2},\Delta_{3}>0$, be real numbers and $0 < \Delta_{4} < 1$. Then, for
any
$$0 < \kappa < 2\Delta_{2}\log^{2}(q)/\log^{2}(\Delta_{4}),$$
there exist a real number $0 < \upsilon < \Delta_{3}$
and two constants $D_{1} \in \mathbb{R}$ and $D_{2}>0$ such that the next estimates
\begin{equation}
\sum_{j \geq 0} \Delta_{1}^{j} q^{-\Delta_{2}j^{2}} \exp( -\Delta_{3} \Delta_{4}^{j} / \epsilon ) \leq
D_{2}\exp( -\frac{\kappa}{2\log(q)} \log^{2}|\epsilon| )|\epsilon|^{D_{1}} \label{q_exp_small_Dirichlet_sum}
\end{equation}
hold for all $\epsilon \in (0,\upsilon]$.
\end{lemma}
\begin{proof} We first recall the next version of Watson's lemma in the context of $q-$Gevrey asymptotics.\medskip

\noindent {\bf Watson's Lemma} {\it Let $b>0$ and $f:[0,b] \mapsto \mathbb{C}$ be a continuous
function having the formal expansion $\sum_{n \geq 0}a_{n}t^{n} \in \mathbb{C}[[t]]$ as its $q-$Gevrey asymptotic expansion of
order $1/\kappa$ at 0, meaning there exist $C,M>0$ such that
$$\left| f(t) - \sum_{n=0}^{N} a_{n}t^{n} \right| \leq CA^{N}q^{\frac{(N+1)N}{2\kappa}}|t|^{N+1},$$
for every $N \geq 0$ and $t \in [0,\delta]$, for some $0<\delta<b$.

Then, the function
$$ I(x) = \int_{0}^{b}f(s)e^{-\frac{s}{x}}ds $$
admits the formal power series $\sum_{n \geq 0} a_{n} n!x^{n+1} \in \mathbb{C}[[x]]$ as its
$q-$Gevrey asymptotic expansion of order $1/\kappa'$ at 0 for any $\kappa'<\kappa$, meaning that
there exist $\tilde{C},\tilde{A}>0$ such that
$$\left| I(x) - \sum_{n=0}^{N-1} a_{n} n! x^{n+1} \right| \leq
\tilde{C}\tilde{A}^{N+1}q^{\frac{(N+1)N}{2\kappa'}}|x|^{N+1},$$
for every $N \geq 1$ and $ x \in [0,\delta']$ for some $0<\delta'<b$.}\medskip

Let $f: [0,+\infty) \mapsto \mathbb{R}$ be a continuously differentiable function.
From Euler-Maclaurin formula, one has
\begin{equation}\label{Euler_Mac_Laurin}
\sum_{j=0}^{n}f(j) = \frac{1}{2}(f(0)+f(n)) + \int_{0}^{n}f(t)dt +
\int_{0}^{n}B_{1}(t-\left\lfloor t\right\rfloor)f'(t)dt,
\end{equation}
for every $n \in \mathbb{N}$, where $B_{1}$ is the Bernoulli polynomial $B_1(s)=s-\frac{1}{2}$.
Here, $\left\lfloor \cdot\right\rfloor$ stands for the floor function.

Let $\epsilon>0$. If we choose $f$ in (\ref{Euler_Mac_Laurin}) to be
$f(s) = \Delta_{1}^{s} q^{-\Delta_{2}s^{2}} \exp( -\Delta_{3} \Delta_{4}^{s} / \epsilon )$, one obtains
\begin{multline*}
\sum_{j=0}^{n} \Delta_{1}^{j} q^{-\Delta_{2}j^{2}} \exp( -\Delta_{3} \Delta_{4}^{j} / \epsilon )
= \frac{1}{2}(e^{-\frac{\Delta_{3}}{\epsilon}} +
\Delta_{1}^{n} q^{-\Delta_{2}n^{2}} \exp( -\Delta_{3} \Delta_{4}^{n} / \epsilon ) )\\
+ \int_{0}^{n} \Delta_{1}^{t} q^{-\Delta_{2}t^{2}} \exp( -\Delta_{3} \Delta_{4}^{t} / \epsilon ) dt\\
+ \int_{0}^{n} B_{1}(t-\left\lfloor t\right\rfloor)
\Delta_{1}^{t} q^{-\Delta_{2}t^{2}} \exp( -\Delta_{3} \Delta_{4}^{t} / \epsilon )
\left( \log(\Delta_{1}) - 2t\Delta_{2}\log(q) - \frac{\Delta_{3}}{\epsilon}\log(\Delta_{4})\Delta_{4}^{t} \right) dt
\end{multline*}
for all $n \geq 0$. Taking the limit when $n$ tends to infinity in the previous expression
we arrive at an equality concerning a convergent series:
\begin{multline}
\sum_{j \geq 0} \Delta_{1}^{j} q^{-\Delta_{2}j^{2}} \exp( -\Delta_{3} \Delta_{4}^{j} / \epsilon )
= \frac{1}{2}e^{-\frac{\Delta_{3}}{\epsilon}} +
\int_{0}^{+\infty} \Delta_{1}^{t} q^{-\Delta_{2}t^{2}} \exp( -\Delta_{3} \Delta_{4}^{t} / \epsilon ) dt\\
+ \int_{0}^{+\infty} B_{1}(t-\left\lfloor t\right\rfloor)
\Delta_{1}^{t} q^{-\Delta_{2}t^{2}} \exp( -\Delta_{3} \Delta_{4}^{t} / \epsilon )
\left( \log(\Delta_{1}) - 2t\Delta_{2}\log(q) - \frac{\Delta_{3}}{\epsilon}\log(\Delta_{4})\Delta_{4}^{t} \right) dt\\
\leq \frac{1}{2}e^{-\frac{\Delta_{3}}{\epsilon}} + L_{1}(\epsilon) + \frac{1}{2}L_{2}(\epsilon)
\label{q_exp_sum_equal_integrals}
\end{multline}
where
\begin{multline*}
L_{1}(\epsilon) =
\int_{0}^{+\infty} \Delta_{1}^{t} q^{-\Delta_{2}t^{2}} \exp( -\Delta_{3} \Delta_{4}^{t} / \epsilon ) dt,\\
L_{2}(\epsilon) = \int_{0}^{+\infty} \Delta_{1}^{t} q^{-\Delta_{2}t^{2}} \exp( -\Delta_{3} \Delta_{4}^{t} / \epsilon )
\left( \log(\Delta_{1}) - 2t\Delta_{2}\log(q) - \frac{\Delta_{3}}{\epsilon}\log(\Delta_{4})\Delta_{4}^{t} \right) dt
\end{multline*}
We now provide upper bound estimates for the integrals $L_{1}(\epsilon)$ and $L_{2}(\epsilon)$. The change of variable
$s=\Delta_{3}\Delta_{4}^{t}$ in $L_{1}(\epsilon)$ and $L_{2}(\epsilon)$ allows us to write
\begin{multline}
L_{1}(\epsilon) = \Delta_{1,2,3,4,q} \int_{0}^{\Delta_{3}}
s^{\frac{\log(\Delta_{1})}{\log(\Delta_{4})} + 2\log(q)\Delta_{2}\frac{\log(\Delta_{3})}{\log^{2}(\Delta_{4})} - 1}
\exp( -\frac{\Delta_{2}\log(q)}{\log^{2}(\Delta_{4})} \log^{2}(s) ) e^{-\frac{s}{\epsilon}} ds,\\
L_{2}(\epsilon) = \Delta_{1,2,3,4,q} \int_{0}^{\Delta_{3}}
s^{\frac{\log(\Delta_{1})}{\log(\Delta_{4})} + 2\log(q)\Delta_{2}\frac{\log(\Delta_{3})}{\log^{2}(\Delta_{4})} - 1}
\left( \log(\Delta_{1}) - \frac{2\Delta_{2}\log(q)}{\log(\Delta_{4})}\log(\frac{s}{\Delta_{3}}) -
\log(\Delta_{4})\frac{s}{\epsilon} \right)\\
\times \exp( -\frac{\Delta_{2}\log(q)}{\log^{2}(\Delta_{4})} \log^{2}(s) ) e^{-\frac{s}{\epsilon}} ds\\
\leq \tilde{L}_{2}(\epsilon) = \Delta_{1,2,3,4,q} \int_{0}^{\Delta_{3}}
s^{\frac{\log(\Delta_{1})}{\log(\Delta_{4})} + 2\log(q)\Delta_{2}\frac{\log(\Delta_{3})}{\log^{2}(\Delta_{4})} - 1}
\left( \log(\Delta_{1}) -
\log(\Delta_{4})\frac{s}{\epsilon} \right)\\
\times \exp( -\frac{\Delta_{2}\log(q)}{\log^{2}(\Delta_{4})} \log^{2}(s) ) e^{-\frac{s}{\epsilon}} ds
\end{multline}
with
$$ \Delta_{1,2,3,4,q} = -\frac{1}{\log(\Delta_{4})}
(\frac{1}{\Delta_{3}})^{\frac{\log(\Delta_{1})}{\log(\Delta_{4})}}
q^{-\Delta_{2}\frac{\log^{2}(\Delta_{3})}{\log^{2}(\Delta_{4})}}.
$$
Combining Lemma\ref{lema2} and Watson's lemma stated above, for any $0<\kappa<2\Delta_{2}\log^{2}(q)/\log^{2}(\Delta_{4})$
one deduces the existence of four constants
$K_{L_1},K_{\tilde{L}_2} \in \mathbb{R}$ and $M_{L_1},M_{\tilde{L}_2}>0$ and a real number $0 < \upsilon < \Delta_{3}$
such that
\begin{multline}
L_{1}(\epsilon) \leq M_{L_1}\exp(-\frac{\kappa}{2\log(q)}\log^{2}(\epsilon)) |\epsilon|^{K_{L_1}},\\
L_{2}(\epsilon) \leq \tilde{L}_{2}(\epsilon) \leq M_{\tilde{L}_2}\exp(-\frac{\kappa}{2\log(q)}\log^{2}(\epsilon))
|\epsilon|^{K_{\tilde{L}_2}} \label{q_exp_small_L1_tildeL2}
\end{multline}
for all $\epsilon \in (0,\upsilon]$. Besides, since $\lim_{\epsilon \rightarrow 0} \epsilon\log^{2}(\epsilon) = 0$,
we see that $\exp(-\Delta_{3}/\epsilon)$ is also $q-$exponen\-tially flat of order $\kappa$ on $(0,\upsilon]$ for
$\upsilon>0$ small enough. Gathering
(\ref{q_exp_sum_equal_integrals}) and (\ref{q_exp_small_L1_tildeL2}) yields the expected estimates
(\ref{q_exp_small_Dirichlet_sum}).
\end{proof}
In the final part of the proof, we furnish bound estimates for the functions $J_{1}(\epsilon)$,$J_{2}(\epsilon)$ and
$J_{3}(\epsilon)$. Concerning $J_{1}(\epsilon)$, a direct use of Lemma~\ref{lema3}, by taking
$$ \Delta_{1} = K_{2}(1/q^{\delta})^{\frac{k}{2}\min_{l=1}^{D-1} d_{l}} \ \ , \ \
\Delta_{2} = \frac{\delta k}{2}\min_{l=1}^{D-1} d_{l} \ \ , \ \
\Delta_{3}=\frac{\delta_{1}(\check{Q}\mu_{0})^{k}}{r_{\mathcal{T}}^{k}} \ \ , \ \ \Delta_{4} = 1/q^{\delta k}
$$
shows that for any $0 < \kappa_{1} < \frac{k}{\delta} \min_{l=1}^{D-1} d_{l}$, one can choose
two constants $K_{J_1} \in \mathbb{R}$ and $M_{J_1}>0$ such that
\begin{equation}
J_{1}(\epsilon) \leq  M_{J_1}\exp(-\frac{\kappa_{1}}{2\log(q)}\log^{2}(|\epsilon|)) |\epsilon|^{K_{J_1}}
\label{J_1_q_exp_flat}
\end{equation}
for all $\epsilon \in \mathcal{E}_{p} \cap \mathcal{E}_{p+1}$, provided that $\epsilon_{0}>0$ is small enough.
By inverting the order of summation, we can rewrite $J_{2}(\epsilon)$ in the form
\begin{multline}
J_{2}(\epsilon) = 2M_{2} \sum_{h \geq 0} (\sum_{j \geq h} K_{3}^{j}) K_{4}^{h}(\frac{1}{q^{\delta}})^{\frac{k \min_{l=1}^{D-1}d_{l}}{2}(h+1)h}
\exp( - \frac{\delta_{1}(\check{Q}\mu_{0})^{k}}{(r_{\mathcal{T}})^{k}} (\frac{1}{q^{\delta k}})^{h}
\frac{1}{|\epsilon|^{k}} )\label{e1192} \\
=
\frac{2M_{2}}{1-K_{3}} \sum_{h \geq 0} (K_{3}K_{4})^{h}(\frac{1}{q^{\delta}})^{\frac{k \min_{l=1}^{D-1}d_{l}}{2}(h+1)h}
\exp( - \frac{\delta_{1}(\check{Q}\mu_{0})^{k}}{(r_{\mathcal{T}})^{k}} (\frac{1}{q^{\delta k}})^{h}
\frac{1}{|\epsilon|^{k}} ).
\end{multline}
Now, a direct application of Lemma~\ref{lema3} with
$$ \Delta_{1} = K_{3}K_{4}(1/q^{\delta})^{\frac{k}{2}\min_{l=1}^{D-1} d_{l}} \ \ , \ \
\Delta_{2} = \frac{\delta k}{2}\min_{l=1}^{D-1} d_{l} \ \ , \ \
\Delta_{3} = \frac{\delta_{1}(\check{Q}\mu_{0})^{k}}{r_{\mathcal{T}}^{k}} \ \ , \ \ \Delta_{4} = 1/q^{\delta k}
$$
shows that for any $0 < \kappa_{2} < \frac{k}{\delta} \min_{l=1}^{D-1} d_{l}$, one can choose
two constants $K_{J_2} \in \mathbb{R}$ and $M_{J_2}>0$ such that
\begin{equation}
J_{2}(\epsilon) \leq  M_{J_2}\exp(-\frac{\kappa_{2}}{2\log(q)}\log^{2}(|\epsilon|)) |\epsilon|^{K_{J_2}}
\label{J_2_q_exp_flat}
\end{equation}
for all $\epsilon \in \mathcal{E}_{p} \cap \mathcal{E}_{p+1}$, taking for granted that $\epsilon_{0}>0$ is small enough.
Finally, again since $\lim_{\epsilon \rightarrow 0} \epsilon^{k}\log^{2}(\epsilon) = 0$, for any 
$0 < \kappa_{3} < \frac{k}{\delta} \min_{l=1}^{D-1} d_{l}$, one can choose
a constant $M_{J_3}>0$ such that
\begin{equation}
J_{3}(\epsilon) \leq  M_{J_3}\exp(-\frac{\kappa_{3}}{2\log(q)}\log^{2}(|\epsilon|))
\label{J_3_q_exp_flat}
\end{equation}
for all $\epsilon \in \mathcal{E}_{p} \cap \mathcal{E}_{p+1}$, for $\epsilon_{0}$ taken small enough.

Finally, bearing in mind the last estimates (\ref{e1192}), (\ref{J_2_q_exp_flat}) and (\ref{J_2_q_exp_flat}),
we deduce from (\ref{maj_difference_u_p_with_J1_J2_J3}) the inequality (\ref{q_exp_small_difference_u_p}). 
\end{proof}

In the second main result of this section, we construct a formal power series in the complex parameter
$\epsilon$ whose coefficients are bounded holomorphic functions on the product of a bounded sector and a strip
that formally solves our main equation (\ref{main_q_diff_diff_first}) and which grants a common $q-$Gevrey asymptotic
expansion of some order $1/\kappa$ of the actual solution $u_{p}$ of (\ref{main_q_diff_diff_first}) build up in
Theorem 1.

\begin{theo}\label{teo2} Under the hypotheses of Theorem~\ref{teo1}, there exists a formal power series
$$ \hat{u}(t,z,\epsilon) = \sum_{m \geq 0} h_{m}(t,z) \frac{\epsilon^{m}}{m!} $$
solution of (\ref{main_q_diff_diff_first}) with coefficients $h_{m}(t,z)$ that belongs to the Banach space
$\mathbb{F}$ of bounded holomorphic functions
on $\mathcal{T} \times H_{\beta'}$ endowed with the sup norm and which is the common $q-$Gevrey asymptotic
expansion of order $1/\kappa$
of the solutions $u_{p}(t,z,\epsilon)$ of (\ref{main_q_diff_diff_first}) seen as holomorphic functions
from $\mathcal{E}_{p}$ into $\mathbb{F}$, for all $0 \leq p \leq \varsigma-1$, for any given
$0 < \kappa < \frac{k}{\delta} \min_{l=1}^{D-1} d_{l}$. Namely, for all $0 \leq p \leq \varsigma-1$, there exists two constants
$A_{p},C_{p}>0$ such that
\begin{equation}
\sup_{t \in \mathcal{T},z \in H_{\beta'}} |u_{p}(t,z,\epsilon) - \sum_{m=0}^{n} h_{m}(t,z) \frac{\epsilon^m}{m!}|
\leq C_{p}A_{p}^{n+1}q^{\frac{(n+1)n}{2\kappa}}|\epsilon|^{n+1}
\end{equation}
for all $n \geq 0$, all $\epsilon \in \mathcal{E}_{p}$.
\end{theo}
\begin{proof} We search for the family of functions $u_{p}(t,z,\epsilon)$, $0 \leq p \leq \varsigma-1$ constructed in Theorem~\ref{teo1}.
For all $0 \leq p \leq \varsigma-1$, we define $G_{p}(\epsilon) := (t,z) \mapsto u_{p}(t,z,\epsilon)$, which is by construction a
holomorphic and bounded function from $\mathcal{E}_{p}$ into the Banach space $\mathbb{F}$ of bounded holomorphic functions on
$\mathcal{T} \times H_{\beta'}$ equipped with the sup norm, where $\mathcal{T}$ is introduced in
Definition~\ref{defi7} and $\beta'>0$ is the width of the strip $H_{\beta'}$ on which the coefficients $c_{l}(z,\epsilon)$ and the forcing term
$f(t,z,\epsilon)$ are holomorphic and bounded with respect to
$z$ (see (\ref{defin_coeff_c_l_z}), (\ref{defin_forcing_f})). Bearing in mind the estimates
(\ref{q_exp_small_difference_u_p}), the cocycle
$\Theta_{p}(\epsilon) = G_{p+1}(\epsilon) - G_{p}(\epsilon)$ is $q-$exponentially flat of order $\kappa$ on
$Z_{p} = \mathcal{E}_{p} \cap \mathcal{E}_{p+1}$, for any $0 \leq p \leq \varsigma-1$, and for any given
$0 < \kappa < \frac{k}{\delta} \min_{l=1}^{D-1} d_{l}$.

Regarding Theorem (q-RS) stated in Section~\ref{seccion6}, we deduce the existence of a formal power series
$\hat{G}(\epsilon) \in \mathbb{F}[[\epsilon]]$ such that the
functions $G_{p}(\epsilon)$ admit $\hat{G}(\epsilon)$ as their $q-$Gevrey asymptotic expansion of order $1/\kappa$ on
$\mathcal{E}_{p}$, for all $0 \leq p \leq \varsigma-1$. We put
$$ \hat{G}(\epsilon ) = \sum_{m \geq 0} h_{m}(t,z) \epsilon^{m}/m! =: \hat{u}(t,z,\epsilon). $$
In the last part of the proof we show that the formal series $\hat{u}(t,z,\epsilon)$ solves the main equation
(\ref{main_q_diff_diff_first}). Since $\hat{G}(\epsilon)$ is the asymptotic expansion of the functions
$G_{p}(\epsilon)$ on $\mathcal{E}_{p}$, we get in particular that
\begin{equation}
\lim_{\epsilon \rightarrow 0, \epsilon \in \mathcal{E}_{p}}
\sup_{t \in \mathcal{T}, z \in H_{\beta'}}|\partial_{\epsilon}^{m}u_{p}(t,z,\epsilon) - h_{m}(t,z)| = 0
\label{limit_deriv_order_m_of_up_epsilon}
\end{equation}
for all $0 \leq p \leq \varsigma-1$, all $m \geq 0$. Now, we choose some $p \in \{ 0, \ldots, \varsigma-1 \}$. By construction, the function
$u_{p}(t,z,\epsilon)$ is a solution of (\ref{main_q_diff_diff_first}). We take derivatives of order $m$ for some $m \geq 0$ with respect to $\epsilon$ on the
left and right hand-side of the equation (\ref{main_q_diff_diff_first}). Using Leibniz rule, we deduce that
$\partial_{\epsilon}^{m}u_{p}(t,z,\epsilon)$ solves the following equation
\begin{multline*}
Q(\partial_{z})\partial_{\epsilon}^{m}u_{p}(t,z,\epsilon) = R_{D}(\partial_{z})
\sum_{m_{1}+m_{2}=m} \frac{m!}{m_{1}!m_{2}!}
\partial_{\epsilon}^{m_1}(\epsilon^{kd_{D}})(t^{k+1}\partial_{t})^{d_{D}}\partial_{\epsilon}^{m_2}u_{p}(t,z,\epsilon)\\
+ \sum_{l=1}^{D-1} \sum_{m_{1}+m_{2}=m} \frac{m!}{m_{1}!m_{2}!}(\partial_{\epsilon}^{m_1}\epsilon^{\Delta_l})
(t^{k+1}\partial_{t})^{d_l}c_{l}(z,\epsilon)R_{l}(\partial_{z})\partial_{\epsilon}^{m_2}u(q^{\delta}t,z,\epsilon) +
\partial_{\epsilon}^{m}f(t,z,\epsilon)
\end{multline*}
for all $m \geq 0$, all $(t,z,\epsilon) \in \mathcal{T} \times H_{\beta'} \times \mathcal{E}_{p}$. We put $c_l(z,\epsilon)=\sum_{m\ge0}c_{l,m}(z)\frac{\epsilon^m}{m!}$. Letting $\epsilon$ tend to zero in the latter formula and making use of
(\ref{limit_deriv_order_m_of_up_epsilon}), we get the recursion on the coefficients $h_{m}(t,z)$,
\begin{multline}
Q(\partial_{z})h_{m}(t,z) = R_{D}(\partial_{z})\frac{m!}{(m-kd_{D})!}(t^{k+1}\partial_{t})^{d_D}h_{m-kd_{D}}(t,z)\\
+ \sum_{l=1}^{D-1}\sum_{m_1+m_2=m-\Delta_l} \frac{m!}{(m_1!m_2!}(t^{k+1}\partial_{t})^{d_l}c_{l,m_1}(z)R_{l}(\partial_{z})
h_{m_2}(q^{\delta}t,z) + (\partial_{\epsilon}^{m}f)(t,z,0) \label{recursion_h_m}
\end{multline}
for all $m \geq \max( \max_{1 \leq l \leq D-1} \Delta_{l}, kd_{D})$, all $(t,z) \in \mathcal{T} \times H_{\beta'}$.
Since the functions $c_l(z,\epsilon)$ and $f(t,z,\epsilon)$ are analytic w.r.t. $\epsilon$ near 0, their Taylor expansion are given by
\begin{equation}
c_l(z,\epsilon)=\sum_{m\ge0}\frac{(\partial_\epsilon c_{l})(z,0)}{m!}\epsilon^m,\qquad f(t,z,\epsilon) = \sum_{m \geq 0} \frac{(\partial_{\epsilon}^{m}f)(t,z,0)}{m!}\epsilon^{m} \label{Taylor_f}
\end{equation}
respectively, for all $\epsilon \in D(0,\epsilon_{0})$ and $z \in H_{\beta'}$. Finally, a direct inspection on the recursion
(\ref{recursion_h_m}) and the expansion (\ref{Taylor_f}) allows us to conclude that the formal series
$\hat{u}(t,z,\epsilon) = \sum_{m \geq 0} h_{m}(t,z)\epsilon^{m}/m!$ solves the equation (\ref{main_q_diff_diff_first}).
\end{proof}

\section{Analytic solutions of an auxiliary convolution initial value problem}\label{seccion8}

In this section, we preserve the values in the layout of the problem considered in Section~\ref{seccion41} and the notation for the geometric construction and domains introduced in Section~\ref{seccion42}. We also consider the solution $w^p_k(\tau,m,\epsilon)$ of the problem (\ref{q_diff_conv_init_v_prob}) determined in Section~\ref{seccion43}. We also fix a value of $\nu>0$.

Let $\mathbf{D} \geq 2$ be an integer. For
$1 \leq l \leq \mathbf{D}$, let
$\mathbf{d}_{l}$,$\boldsymbol{\delta}_{l}$,$\mathbf{\Delta}_{l} \geq 0$ be nonnegative integers.
We assume that 
\begin{equation}
1 = \boldsymbol{\delta}_{1} \ \ , \ \ \boldsymbol{\delta}_{l} < \boldsymbol{\delta}_{l+1},  \label{assum_delta_l_bf_Borel}
\end{equation}
for all $1 \leq l \leq \mathbf{D}-1$. We also make the assumption that
\begin{equation}
\mathbf{d}_{\mathbf{D}} = (\boldsymbol{\delta}_{\mathbf{D}}-1)(k+1) \ \ , \ \
\mathbf{d}_{l} > (\boldsymbol{\delta}_{l}-1)(k+1) \ \ , \ \
\mathbf{\Delta}_{l} - \mathbf{d}_{l} + \boldsymbol{\delta}_{l} - 1 \geq 0 \ \ , \ \
\mathbf{\Delta}_{\mathbf{D}} = \mathbf{d}_{\mathbf{D}} - \boldsymbol{\delta}_{\mathbf{D}} + 1
\label{assum_dl_delta_l_Delta_l_bf_Borel}
\end{equation}
for all $1 \leq l \leq \mathbf{D}-1$. We define integers $\mathbf{d}_{l,k} \geq 0$ satisfying
\begin{equation}
\mathbf{d}_{l} + k + 1 = \boldsymbol{\delta}_{l}(k+1) + \mathbf{d}_{l,k} \label{defin_d_l_k_bold_Borel}
\end{equation}
for all $1 \leq l \leq \mathbf{D}$. Let $\mathbf{Q}(X),\mathbf{R}_{l}(X) \in \mathbb{C}[X]$,
$0 \leq l \leq \mathbf{D}$, be polynomials such that
\begin{equation}
\mathrm{deg}(\mathbf{Q}) \geq \mathrm{deg}(\mathbf{R}_{\mathbf{D}}) \geq \mathrm{deg}(\mathbf{R}_{l}) \ \ , \ \
\mathbf{Q}(im) \neq 0 \ \ , \ \ \mathbf{R}_{\mathbf{D}}(im) \neq 0 \label{assum_Q_Rl_bf_Borel}
\end{equation}
for all $m \in \mathbb{R}$, all $0 \leq l \leq \mathbf{D}-1$. We consider $\mu>\deg(\mathbf{R}_{\mathbf{D}})+1$ and we make the additional assumption that there exists
an unbounded sector
$$ S_{\mathbf{Q},\mathbf{R}_{\mathbf{D}}} = \{ z \in \mathbb{C} /
|z| \geq r_{\mathbf{Q},\mathbf{R}_{\mathbf{D}}} \ \ , \ \ |\mathrm{arg}(z) - d_{\mathbf{Q},\mathbf{R}_{\mathbf{D}}}|
\leq \eta_{\mathbf{Q},\mathbf{R}_{\mathbf{D}}} \} $$
with bisecting direction $d_{\mathbf{Q},\mathbf{R}_{\mathbf{D}}} \in \mathbb{R}$, aperture
$2\eta_{\mathbf{Q},\mathbf{R}_{\mathbf{D}}}>0$ for some radius $r_{\mathbf{Q},\mathbf{R}_{\mathbf{D}}}>0$
such that
\begin{equation}
\frac{\mathbf{Q}(im)}{\mathbf{R}_{\mathbf{D}}(im)} \in S_{\mathbf{Q},\mathbf{R}_{\mathbf{D}}} \label{quotient_Q_RD_in_S_bf}
\end{equation} 
for all $m \in \mathbb{R}$. We factorize the polynomial $\mathbf{P}_{m}(\tau) = \mathbf{Q}(im)k -
\mathbf{R}_{\mathbf{D}}(im)k^{\boldsymbol{\delta}_{\mathbf{D}}}
\tau^{(\boldsymbol{\delta}_{\mathbf{D}}-1)k}$ in the form
\begin{equation}
\mathbf{P}_{m}(\tau) = -\mathbf{R}_{\mathbf{D}}(im)k^{\boldsymbol{\delta}_{\mathbf{D}}}
\Pi_{l=0}^{(\boldsymbol{\delta}_{\mathbf{D}}-1)k-1} (\tau - \mathbf{q}_{l}(m)) \label{factor_P_m_bf}
\end{equation}
where
\begin{multline}
\mathbf{q}_{l}(m) = (\frac{|\mathbf{Q}(im)|}{|\mathbf{R}_{\mathbf{D}}(im)|
k^{\boldsymbol{\delta}_{\mathbf{D}}-1}})^{\frac{1}{(\boldsymbol{\delta}_{\mathbf{D}}-1)k}}\\
\times \exp( \sqrt{-1}( \mathrm{arg}( \frac{\mathbf{Q}(im)}{\mathbf{R}_{\mathbf{D}}(im)
k^{\boldsymbol{\delta}_{\mathbf{D}}-1}}) \frac{1}{(\boldsymbol{\delta}_{\mathbf{D}}-1)k} +
\frac{2\pi l}{(\boldsymbol{\delta}_{\mathbf{D}}-1)k} ) ) \label{defin_roots_bf}
\end{multline}
for all $0 \leq l \leq (\boldsymbol{\delta}_{\mathbf{D}}-1)k-1$, all $m \in \mathbb{R}$. For every $0\le p\le \varsigma -1$ we consider a finite set of real numbers $\mathfrak{e}_{p',p}$, $0 \leq p' \leq \chi_p-1$. We choose a family of unbounded
sectors $S_{\mathfrak{e}_{p',p}}$ centered at 0 with bisecting
direction $\mathfrak{e}_{p',p}$, a small closed disc $\bar{D}(0,\rho)$ and we prescribe the sector
$S_{\mathbf{Q},\mathbf{R}_{\mathbf{D}}}$ in such a way that the following conditions hold.\medskip

\noindent 1) There exists a constant $\mathbf{M}_{1}>0$ such that
\begin{equation}
|\tau - \mathbf{q}_{l}(m)| \geq \mathbf{M}_{1}(1 + |\tau|) \label{root_cond_1_bf}
\end{equation}
for all $0 \leq l \leq (\boldsymbol{\delta}_{\mathbf{D}}-1)k-1$, all $m \in \mathbb{R}$,
all $\tau \in S_{\mathfrak{e}_{p',p}} \cup \bar{D}(0,\rho)$, for all $0 \leq p' \leq \chi_p -1$.\medskip

\noindent 2) There exists a constant $\mathbf{M}_{2}>0$ such that
\begin{equation}
|\tau - \mathbf{q}_{l_0}(m)| \geq \mathbf{M}_{2}|\mathbf{q}_{l_0}(m)| \label{root_cond_2_bf}
\end{equation}
for some $l_{0} \in \{0,\ldots,(\boldsymbol{\delta}_{\mathbf{D}}-1)k-1 \}$, all $m \in \mathbb{R}$, all
$\tau \in S_{\mathfrak{e}_{p',p}} \cup \bar{D}(0,\rho)$, for all $0 \leq p' \leq \chi_p -1$.\medskip

\noindent 3) We make the assumption that for all $0 \leq p' \leq \chi_p-1$, one has
\begin{equation}\label{e1374}
S_{\mathfrak{e}_{p',p}} \subset U_{\mathfrak{d}_{p}},
\end{equation}
where $U_{\mathfrak{d}_p}$ are the ones fixed in Section~\ref{seccion41}.
Let $0 \leq p \leq \varsigma-1$. By construction
of the roots (\ref{defin_roots_bf}) in the factorization (\ref{factor_P_m_bf}) and using the lower bound estimates
(\ref{root_cond_1_bf}), (\ref{root_cond_2_bf}), we get a constant $C_{\mathbf{P}}>0$ such that
\begin{multline}
|\mathbf{P}_{m}(\tau)| \geq \mathbf{M}_{1}^{(\boldsymbol{\delta}_{\mathbf{D}}-1)k-1}\mathbf{M}_{2}
|\mathbf{R}_{\mathbf{D}}(im)|k^{\boldsymbol{\delta}_{\mathbf{D}}}
(\frac{|\mathbf{Q}(im)|}{|\mathbf{R}_{\mathbf{D}}(im)|k^{\boldsymbol{\delta}_{\mathbf{D}}-1}})^{
\frac{1}{(\boldsymbol{\delta}_{\mathbf{D}}-1)k}} (1+|\tau|)^{(\boldsymbol{\delta}_{\mathbf{D}}-1)k-1}\\
\geq \mathbf{M}_{1}^{(\boldsymbol{\delta}_{\mathbf{D}}-1)k-1}\mathbf{M}_{2}
\frac{k^{\boldsymbol{\delta}_{\mathbf{D}}}}{(k^{\boldsymbol{\delta}_{\mathbf{D}}-1})^{
\frac{1}{(\boldsymbol{\delta}_{\mathbf{D}}-1)k}} }
(r_{\mathbf{Q},\mathbf{R}_{\mathbf{D}}})^{\frac{1}{(\boldsymbol{\delta}_{\mathbf{D}}-1)k}}
|\mathbf{R}_{\mathbf{D}}(im)| \\
\times (\min_{x \geq 0}
\frac{(1+x)^{(\boldsymbol{\delta}_{\mathbf{D}}-1)k-1}}{(1+x^{k})^{(\boldsymbol{\delta}_{\mathbf{D}}-1)
- \frac{1}{k}}}) (1 + |\tau|^{k})^{(\boldsymbol{\delta}_{\mathbf{D}}-1)-\frac{1}{k}}\\
= C_{\mathbf{P}} (r_{\mathbf{Q},\mathbf{R}_{\mathbf{D}}})^{
\frac{1}{(\boldsymbol{\delta}_{\mathbf{D}}-1)k}} |\mathbf{R}_{\mathbf{D}}(im)|
(1+|\tau|^{k})^{(\boldsymbol{\delta}_{\mathbf{D}}-1)-\frac{1}{k}}
\label{low_bounds_P_m_bf}
\end{multline}
for all $\tau \in S_{\mathfrak{e}_{p',p}} \cup \bar{D}(0,\rho)$, all $m \in \mathbb{R}$, all $0 \leq p' \leq \chi_p - 1$.

We consider a function $m \mapsto \mathbf{C}_{0,0}(m,\epsilon)$ that belongs to the Banach space
$E_{(\beta,\mu)}$ and which
depends holomorphically on $\epsilon \in D(0,\epsilon_{0})$. Let $\mathbf{c}_{0,0}(\epsilon)$ and
$\mathbf{c}_{\mathbf{F}}(\epsilon)$ be bounded holomorphic functions on $D(0,\epsilon_{0})$ which vanish at the origin
$\epsilon=0$. For all $0 \leq p \leq \varsigma-1$, we consider the function $w_{k}^{p}(u,m,\epsilon)$
constructed in Proposition~\ref{prop6} that solves the problem (\ref{q_diff_conv_init_v_prob}). In this section we focus on the following convolution problem
\begin{multline}
v_{k}^{\mathfrak{e}_{p',p},\mathfrak{d}_{p}}(\tau,m,\epsilon)\\
= \sum_{1 \leq h \leq \boldsymbol{\delta}_{\mathbf{D}}-1} A_{\boldsymbol{\delta}_{\mathbf{D}},h}
\frac{\mathbf{R}_{\mathbf{D}}(im)}{\mathbf{P}_{m}(\tau)\Gamma(\boldsymbol{\delta}_{\mathbf{D}}-h)}\int_{0}^{\tau^{k}}
(\tau^{k}-s)^{\boldsymbol{\delta}_{\mathbf{D}}-h-1}
k^{h} s^{h} v_{k}^{\mathfrak{e}_{p',p},\mathfrak{d}_{p}}(s^{1/k},m,\epsilon) \frac{ds}{s}\\
+ \sum_{l=1}^{\mathbf{D}-1} \frac{\mathbf{R}_{l}(im)}{\mathbf{P}_{m}(\tau)}
\left( \epsilon^{\mathbf{\Delta}_{l}-\mathbf{d}_{l}+\boldsymbol{\delta}_{l}-1}
\frac{1}{\Gamma( \frac{\mathbf{d}_{l,k}}{k} )} \right.
\int_{0}^{\tau^{k}} (\tau^{k}-s)^{\frac{\mathbf{d}_{l,k}}{k}-1}k^{\boldsymbol{\delta}_l}s^{\boldsymbol{\delta}_l}
v_{k}^{\mathfrak{e}_{p',p},\mathfrak{d}_{p}}(s^{1/k},m,\epsilon) \frac{ds}{s}\\
+ \sum_{1 \leq h \leq \boldsymbol{\delta}_{l}-1} A_{\boldsymbol{\delta}_{l},h}
\epsilon^{\mathbf{\Delta}_{l}-\mathbf{d}_{l}+\boldsymbol{\delta}_{l}-1}
\frac{1}{\Gamma( \frac{\mathbf{d}_{l,k}}{k} + \boldsymbol{\delta}_{l}-h)} \int_{0}^{\tau^{k}}
\left. (\tau^{k}-s)^{\frac{\mathbf{d}_{l,k}}{k}+\boldsymbol{\delta}_{l}-h-1}k^{h}s^{h}
v_{k}^{\mathfrak{e}_{p',p},\mathfrak{d}_{p}}(s^{1/k},m,\epsilon) \frac{ds}{s} \right)\\
+ \epsilon^{-1}\frac{1}{\mathbf{P}_{m}(\tau)\Gamma(1 + \frac{1}{k})}\\
\times \int_{0}^{\tau^{k}}
(\tau^{k}-s)^{1/k} \frac{\mathbf{c}_{0,0}(\epsilon)}{(2\pi)^{1/2}}
( \int_{-\infty}^{+\infty} \mathbf{C}_{0,0}(m-m_{1},\epsilon)\mathbf{R}_{0}(im_{1})
v_{k}^{\mathfrak{e}_{p',p},\mathfrak{d}_{p}}(s^{1/k},m_{1},\epsilon) dm_{1} )\frac{ds}{s}\\
+ \epsilon^{-1} \mathbf{c}_{\mathbf{F}}(\epsilon) \frac{1}{\mathbf{P}_{m}(\tau)\Gamma(1 + \frac{1}{k})}\int_{0}^{\tau^{k}}
(\tau^{k}-s)^{1/k} w_{k}^{p}(s^{1/k},m,\epsilon) \frac{ds}{s}. \label{k_Borel_equation_analytic_bf}
\end{multline}

Here, $A_{\boldsymbol{\delta}_{l},h}\in\mathbb{C}$ for every $1\le l\le \boldsymbol{\delta}_{l}-1$ and $1\le l\le \mathbf{D}$. Our goal in this section is to construct actual holomorphic solutions
$v_{k}^{\mathfrak{e}_{p',p},\mathfrak{d}_{p}}(\tau,m,\epsilon)$ of this problem which belong to Banach spaces
$F_{(\nu,\beta,\mu,k);\bar{\Omega}}$ for some starshaped domains, $\Omega$, for $0 \leq p' \leq \chi_p-1$,
$0 \leq p \leq \varsigma-1$.

\begin{prop}\label{prop9} Under the assumption that
\begin{equation}
\boldsymbol{\delta}_{\mathbf{D}} \geq \boldsymbol{\delta}_{l} + \frac{1}{k} \label{deltaD_deltal_constraints_bf}
\end{equation}
for all $1 \leq l \leq \mathbf{D}-1$, there exists a radius $r_{\mathbf{Q},\mathbf{R}_{\mathbf{D}}}>0$ and constants
$\boldsymbol{\varsigma}_{0,0},\boldsymbol{\varsigma}_{0},\boldsymbol{\varsigma}_{F}>0$ such that if
\begin{equation}
\sup_{\epsilon \in D(0,\epsilon_{0})} \left| \frac{\mathbf{c}_{0,0}(\epsilon)}{\epsilon} \right|
\leq \boldsymbol{\varsigma}_{0,0} \ \ , \ \
||\mathbf{C}_{0,0}(m,\epsilon)||_{(\beta,\mu)} \leq \boldsymbol{\varsigma}_{0} \ \ , \ \
\sup_{\epsilon \in D(0,\epsilon_{0})} \left| \frac{\mathbf{c}_{\mathbf{F}}(\epsilon)}{\epsilon} \right|
\leq \boldsymbol{\varsigma}_{\mathbf{F}} \label{varsigma00_C00_cF_small_bf}
\end{equation}
for all $\epsilon \in D(0,\epsilon_{0})$, the equation (\ref{k_Borel_equation_analytic_bf}) has a unique solution
$v_{k}^{\mathfrak{e}_{p',p},\mathfrak{d}_{p}}(\tau,m,\epsilon)$ that belongs to the Banach space
$F_{(\nu,\beta,\mu,k);\overline{S_{\mathfrak{e}_{p',p}}}}$ (see Definition~\ref{defi3}) for all $0 \leq p \leq \varsigma-1$, all
$0 \leq p' \leq \chi_p-1$. Moreover, we can write
$v_{k}^{\mathfrak{e}_{p',p},\mathfrak{d}_{p}}(\tau,m,\epsilon)$ as a sum
\begin{equation}
v_{k}^{\mathfrak{e}_{p',p},\mathfrak{d}_{p}}(\tau,m,\epsilon) = \sum_{j \geq 0}
v_{k,j}^{\mathfrak{e}_{p',p},\mathfrak{d}_{p}}(\tau,m,\epsilon) 
\end{equation}
where each function $v_{k,j}^{\mathfrak{e}_{p',p},\mathfrak{d}_{p}}(\tau,m,\epsilon)$ satisfies the next features.\medskip

\noindent 1) For all $j \geq 0$, the function $v_{k,j}^{\mathfrak{e}_{p',p},\mathfrak{d}_{p}}(\tau,m,\epsilon)$ belongs
to $F_{(\nu,\beta,\mu,k);\overline{S_{\mathfrak{e}_{p',p}}}}$ and there exist constants $C_{5}>0$, $0 < K_{5} < 1$ such that
\begin{equation}\label{e1465}
|| v_{k,j}^{\mathfrak{e}_{p',p},\mathfrak{d}_{p}}(\tau,m,\epsilon) ||_{(\nu,\beta,\mu,k);\overline{S_{\mathfrak{e}_{p',p}}}}
\leq C_{5}K_{5}^{j}
\end{equation}
for all $j \geq 0$, all $\epsilon \in D(0,\epsilon_{0})$.\medskip

\noindent 2) For all $j \geq 0$, the function
$v_{k,j}^{\mathfrak{e}_{p',p},\mathfrak{d}_{p}}(\tau,m,\epsilon)$ belongs to 
$F_{(\nu,\beta,\mu,k);\bar{D}(0,\check{Q}\mu_{0,j})}$ and there exist constants $C_{6}>0$, $0<K_{6}<1$ such that
\begin{equation}\label{e1474}
|| v_{k,j}^{\mathfrak{e}_{p',p},\mathfrak{d}_{p}}(\tau,m,\epsilon) ||_{(\nu,\beta,\mu,k);\bar{D}(0,\check{Q}\mu_{0,j})}
\leq C_{6}K_{6}^{j}(\frac{1}{q^{\delta}})^{\frac{k \min_{l=1}^{D-1}d_{l}}{2}(j+1)j}
\end{equation}
for all $j \geq 0$, all $\epsilon \in D(0,\epsilon_{0})$.\medskip

\noindent 3) For all $j \geq 0$, all $0 \leq h \leq j$, the function
$v_{k,j}^{\mathfrak{e}_{p',p},\mathfrak{d}_{p}}(\tau,m,\epsilon)$ belongs to
$F_{(\nu,\beta,\mu,k);\overline{S_{\mathfrak{e}_{p',p}}} \cap \bar{D}(0,\hat{Q}\mu_{1,h})}$ and there exist constants
$C_{7}>0$, $0 < K_{7} <1$ such that
\begin{equation}\label{e1484}
|| v_{k,j}^{\mathfrak{e}_{p',p},\mathfrak{d}_{p}}(\tau,m,\epsilon)
||_{(\nu,\beta,\mu,k);\overline{S_{\mathfrak{e}_{p',p}}} \cap \bar{D}(0,\hat{Q}\mu_{1,h})}\\
\leq C_{7} K_{7}^{j} (\frac{1}{q^{\delta}})^{\frac{k \min_{l=1}^{D-1}d_{l}}{2}(h+1)h}
\end{equation}
for all $j \geq 0$, all $0 \leq h \leq j$ and all $\epsilon \in D(0,\epsilon_{0})$.

The constants $C_5,K_5,C_6,K_6,C_7,K_7$ depend on the geometric elements and constants determined by the problem under study, and do not depend on $\epsilon$.
\end{prop}
\begin{proof} For all $j \geq 0$ and $\epsilon\in D(0,\epsilon_0)$, we introduce the next linear operators
\begin{multline}
 \mathcal{G}_{\epsilon,j}(w(\tau,m))\\
 = \sum_{1 \leq h \leq \boldsymbol{\delta}_{\mathbf{D}}-1} A_{\boldsymbol{\delta}_{\mathbf{D}},h}
\frac{\mathbf{R}_{\mathbf{D}}(im)}{\mathbf{P}_{m}(\tau)\Gamma(\boldsymbol{\delta}_{\mathbf{D}}-h)}\int_{0}^{\tau^{k}}
(\tau^{k}-s)^{\boldsymbol{\delta}_{\mathbf{D}}-h-1}
(k^{h} s^{h} w(s^{1/k},m)) \frac{ds}{s}\\
+ \sum_{l=1}^{\mathbf{D}-1} \frac{\mathbf{R}_{l}(im)}{\mathbf{P}_{m}(\tau)}
\left( \epsilon^{\mathbf{\Delta}_{l}-\mathbf{d}_{l}+\boldsymbol{\delta}_{l}-1}
\frac{1}{\Gamma( \frac{\mathbf{d}_{l,k}}{k} )} \right.
\int_{0}^{\tau^{k}} (\tau^{k}-s)^{\frac{\mathbf{d}_{l,k}}{k}-1}(k^{\boldsymbol{\delta}_l}s^{\boldsymbol{\delta}_l}
w(s^{1/k},m)) \frac{ds}{s}\\
+ \sum_{1 \leq h \leq \boldsymbol{\delta}_{l}-1} A_{\boldsymbol{\delta}_{l},h}
\epsilon^{\mathbf{\Delta}_{l}-\mathbf{d}_{l}+\boldsymbol{\delta}_{l}-1}
\frac{1}{\Gamma( \frac{\mathbf{d}_{l,k}}{k} + \boldsymbol{\delta}_{l}-h)} \int_{0}^{\tau^{k}}
\left. (\tau^{k}-s)^{\frac{\mathbf{d}_{l,k}}{k}+\boldsymbol{\delta}_{l}-h-1}(k^{h}s^{h}
w(s^{1/k},m)) \frac{ds}{s} \right)\\
+ \epsilon^{-1}\frac{1}{\mathbf{P}_{m}(\tau)\Gamma(1 + \frac{1}{k})}\\
\times \int_{0}^{\tau^{k}}
(\tau^{k}-s)^{1/k} \frac{\mathbf{c}_{0,0}(\epsilon)}{(2\pi)^{1/2}}
( \int_{-\infty}^{+\infty} \mathbf{C}_{0,0}(m-m_{1},\epsilon)\mathbf{R}_{0}(im_{1})
w(s^{1/k},m_{1}) dm_{1} )\frac{ds}{s}\\
+ \epsilon^{-1} \mathbf{c}_{\mathbf{F}}(\epsilon) \frac{1}{\mathbf{P}_{m}(\tau)\Gamma(1 + \frac{1}{k})}\int_{0}^{\tau^{k}}
(\tau^{k}-s)^{1/k} w_{k,j}^{p}(s^{1/k},m,\epsilon) \frac{ds}{s}.
\end{multline}
where $w_{k,j}^{p}(\tau,m,\epsilon)$ are the functions constructed in Proposition~\ref{prop6}.

We depart from a lemma which provides estimates of the functions $w_{k,j}^{p}(\tau,m,\epsilon)$ within the Banach
spaces $F_{(\nu,\beta,\mu,k);\bar{\Omega}}$ for suitable starshaped domains $\Omega$.

\begin{lemma}\label{lema4} 1) The function $w_{k,j}^{p}(\tau,m,\epsilon)$ belongs to the space
$F_{(\nu,\beta,\mu,k);\overline{S}_{\mathfrak{e}_{p',p}}}$ and moreover there exists a constant
$C_{\Theta_{\hat{Q}\mu_{1}}^{p},k,k_{1}}>0$ (depending on $q,\alpha,\delta,R_l,\mu,\mu_1,\hat{Q},\tau_0,k_1,k,d_l,\epsilon_0,\Delta_l,r_{Q,R_D},\nu$)) with
\begin{equation}
||w_{k,j}^{p}(\tau,m,\epsilon)||_{(\nu,\beta,\mu,k);\overline{S}_{\mathfrak{e}_{p',p}}} \leq
C_{\Theta_{\hat{Q}\mu_{1}}^{p},k,k_{1}}K_{1}^{j}
||\frac{\psi_{k}(\tau,m,\epsilon)}{P_{m}(\tau)}||_{(k_{1},\beta,\mu,\alpha);\Theta_{\hat{Q}\mu_{1}}^{p}}
\label{compare_norm_qexp_norm_exp_wkj_on_sector}
\end{equation}
for all $j \geq 0$, where the constant $0<K_{1}<1$ is that in (\ref{norm_w_kj_on_triangle}).\\
\noindent 2) The function $w_{k,j}^{p}(\tau,m,\epsilon)$ belongs to the space
$F_{(\nu,\beta,\mu,k);\bar{D}(0,\check{Q}\mu_{0,j})}$ and furthermore there exists a constant
$C_{\bar{D}(0,\check{Q}\mu_{0}),k,k_{1}}>0$ (depending on $q,\alpha,\delta,R_l,\mu,\epsilon_0,\Delta_l,r_{Q,R_D},\gamma,\check{Q},\mu_0,\tau_0,k,k_1,d_l,\nu$) with
\begin{multline}
||w_{k,j}^{p}(\tau,m,\epsilon)||_{(\nu,\beta,\mu,k);\bar{D}(0,\check{Q}\mu_{0,j})} \leq
C_{\bar{D}(0,\check{Q}\mu_{0}),k,k_{1}}K_{2}^{j}(\frac{1}{q^{\delta}})^{\frac{k \min_{l=1}^{D-1}d_{l}}{2}(j+1)j}\\
\times ||\frac{\psi_{k}(\tau,m,\epsilon)}{P_{m}(\tau)}||_{(k_{1},\beta,\mu,\alpha);\bar{D}(0,\check{Q}\mu_{0})}
\label{compare_norm_qexp_norm_exp_wkj_on_disc_j}
\end{multline}
for all $j \geq 0$, where the constant $0<K_{2}<1$ is defined in (\ref{norm_w_kj_on_discs}).\\
\noindent 3) Let $0\le h\le j$. The function $w_{k,j}^{p}(\tau,m,\epsilon)$ belongs to the space
$F_{(\nu,\beta,\mu,k);\overline{S_{\mathfrak{e}_{p',p}}} \cap \bar{D}(0,\hat{Q}\mu_{1,h})}$. There exists $C_{\overline{S_{\mathfrak{e}_{p',p}}} \cap \bar{D}(0,\hat{Q}\mu_{1}),k,k_{1}}>0$ (dep. on $q,\alpha,\delta,R_l,\mu,\epsilon_0,\Delta_l,r_{Q,R_D},\gamma_l,\hat{Q},\check{Q},\mu_0,\mu_1,\tau_0,k,k_1,d_l,\nu$) with
\begin{multline}
||w_{k,j}^{p}(\tau,m,\epsilon)||_{(\nu,\beta,\mu,k);\overline{S_{\mathfrak{e}_{p',p}}} \cap \bar{D}(0,\hat{Q}\mu_{1,h})}\\
\leq C_{\overline{S_{\mathfrak{e}_{p',p}}} \cap \bar{D}(0,\hat{Q}\mu_{1}),k,k_{1}}
(\max(K_{2},K_{3}))^{j}(\frac{1}{q^{\delta}})^{\frac{k \min_{l=1}^{D-1}d_{l}}{2}(h+1)h}\\
\times
\max( ||\frac{\psi_{k}(\tau,m,\epsilon)}{P_{m}(\tau)} ||_{(k_{1},\beta,\mu,\alpha);\bar{D}(0,\check{Q}\mu_{0})},
||\frac{\psi_{k}(\tau,m,\epsilon)}{P_{m}(\tau)} ||_{(k_{1},\beta,\mu,\alpha);\bar{A}_{\hat{Q}\mu_{1}}},
||\frac{\psi_{k}(\tau,m,\epsilon)}{P_{m}(\tau)} ||_{(k_{1},\beta,\mu,\alpha);P\Omega_{0}^{p}} )
\label{compare_norm_qexp_norm_exp_wkj_on_disc_h_and_sector}
\end{multline}
for all $j \geq 0$ and $0 \leq h \leq j$, where the constant $0<K_{3}<1$ is determined in
(\ref{norm_w_kj_on_square_frames}).
\end{lemma}
\begin{proof} We first show 1). By definition of the norms involved we get
\begin{multline}
||w_{k,j}^{p}(\tau,m,\epsilon)||_{(\nu,\beta,\mu,k);\overline{S_{\mathfrak{e}_{p',p}}}}\\
= \sup_{\tau \in \overline{S_{\mathfrak{e}_{p',p}}},m \in \mathbb{R}} \{ (1 + |\tau|^{2k})e^{-\nu |\tau|^{k}}
\exp( \frac{k_{1}}{2} \frac{\log^{2}(|\tau|+\tau_{0})}{\log(q)} + \alpha \log(|\tau|+\tau_{0}) )\}\\
\times \{(1+|m|)^{\mu}e^{\beta|m|}\frac{1}{|\tau|}
\exp( \frac{-k_{1}}{2} \frac{\log^{2}(|\tau|+\tau_{0})}{\log(q)} - \alpha \log(|\tau|+\tau_{0}) )
|w_{k,j}^{p}(\tau,m,\epsilon)|\}\\
\leq \sup_{\tau \in \Theta_{\hat{Q}\mu_{1}}^{p}} \{ (1 + |\tau|^{2k})e^{-\nu |\tau|^{k}}
\exp( \frac{k_{1}}{2} \frac{\log^{2}(|\tau|+\tau_{0})}{\log(q)} + \alpha \log(|\tau|+\tau_{0}) )\}
||w_{k,j}^{p}(\tau,m,\epsilon)||_{(k_{1},\beta,\mu,\alpha);\Theta_{\hat{Q}\mu_{1}}^{p}}
\label{bounds_norm_qexp_wjkp_on_sector}
\end{multline}
since $\overline{S_{\mathfrak{e}_{p',p}}} \subset \Theta_{\hat{Q}\mu_{1}}^{p}$ (see~(\ref{e1374}). Hence, the inequality
(\ref{compare_norm_qexp_norm_exp_wkj_on_sector}) is a consequence of the estimates (\ref{norm_w_kj_on_triangle}).\\
The proof of 2) is analogous to the one given above by replacing the sector
$\overline{S_{\mathfrak{e}_{p',p}}}$ by the disc $\bar{D}(0,\check{Q}\mu_{0,j})$ in the norm estimates
(\ref{bounds_norm_qexp_wjkp_on_sector}) and by using the inequality
(\ref{norm_w_kj_on_discs}) instead of (\ref{norm_w_kj_on_triangle}).\\
We now give the proof for 3). Due to (\ref{constraint_mu1_mu0}) we can make the following splitting
\begin{equation}
\overline{S_{\mathfrak{e}_{p',p}}} \cap \bar{D}(0,\hat{Q}\mu_{1,h}) =
(\overline{S_{\mathfrak{e}_{p',p}}} \cap \bar{D}(0,\check{Q}\mu_{0,j})) \bigcup
(\bigcup_{l=h}^{j} ( \{ \tau \in \mathbb{C} / \check{Q}\mu_{0,l} \leq |\tau| \leq \hat{Q}\mu_{1,l} \} \cap 
\overline{S_{\mathfrak{e}_{p',p}}} ))
\end{equation}
for all $0 \leq h \leq j$. Using a similar factorization as in (\ref{bounds_norm_qexp_wjkp_on_sector}), we deduce a
constant\\ $\hat{C}_{\overline{S_{\mathfrak{e}_{p',p}}} \cap \bar{D}(0,\hat{Q}\mu_{1}),k,k_{1}}>0$ (depending on $q,\alpha,\delta,R_l,\mu,\epsilon_0,\Delta_l,r_{Q,R_D},\gamma_l,\hat{Q},\check{Q},\mu_0,\mu_1,\tau_0,k,k_1,d_l,\nu$)
such that
\begin{multline}
||w_{k,j}^{p}(\tau,m,\epsilon)||_{(\nu,\beta,\mu,k);\overline{S_{\mathfrak{e}_{p',p}}} \cap \bar{D}(0,\hat{Q}\mu_{1,h})}
\leq \hat{C}_{\overline{S_{\mathfrak{e}_{p',p}}} \cap \bar{D}(0,\hat{Q}\mu_{1}),k,k_{1}}\\
\times
\max_{l=h}^{j} \left(
||w_{k,j}^{p}(\tau,m,\epsilon)||_{(k_{1},\beta,\mu,\alpha);\overline{S_{\mathfrak{e}_{p',p}}} \cap \bar{D}(0,\check{Q}\mu_{0,j})},
||w_{k,j}^{p}(\tau,m,\epsilon)||_{(k_{1},\beta,\mu,\alpha);\{ \tau \in \mathbb{C} / \check{Q}\mu_{0,l} \leq |\tau| \leq \hat{Q}\mu_{1,l} \} \cap 
\overline{S_{\mathfrak{e}_{p',p}}} } \right)\\
\leq \hat{C}_{\overline{S_{\mathfrak{e}_{p',p}}} \cap \bar{D}(0,\hat{Q}\mu_{1}),k,k_{1}}
\max_{l=h}^{j} \left(
||w_{k,j}^{p}(\tau,m,\epsilon)||_{(k_{1},\beta,\mu,\alpha);\bar{D}(0,\check{Q}\mu_{0,j})},
||w_{k,j}^{p}(\tau,m,\epsilon)||_{(k_{1},\beta,\mu,\alpha);P\Omega_{l}^{p}} \right)
\end{multline}
by taking into account the inclusions
$$ \overline{S_{\mathfrak{e}_{p',p}}} \cap \bar{D}(0,\check{Q}\mu_{0,j}) \subset 
\bar{D}(0,\check{Q}\mu_{0,j})\ \ , \ \ \{ \tau \in \mathbb{C} / \check{Q}\mu_{0,l} \leq |\tau| \leq \hat{Q}\mu_{1,l} \} \cap
\overline{S_{\mathfrak{e}_{p',p}}} \subset P\Omega_{l}^{p}
$$
for all $h \leq l \leq j$. From the estimates (\ref{norm_w_kj_on_discs}) and
(\ref{norm_w_kj_on_square_frames}), we deduce that
\begin{multline}
||w_{k,j}^{p}(\tau,m,\epsilon)||_{(\nu,\beta,\mu,k);\overline{S_{\mathfrak{e}_{p',p}}} \cap \bar{D}(0,\hat{Q}\mu_{1,h})}
\leq \hat{C}_{\overline{S_{\mathfrak{e}_{p',p}}} \cap \bar{D}(0,\hat{Q}\mu_{1}),k,k_{1}}\\
\times
\max_{l=h}^{j} \left( K_{2}^{j}
(\frac{1}{q^{\delta}})^{\frac{k \min_{l=1}^{D-1}d_{l}}{2}(j+1)j} \right. \\
\times ||\frac{\psi_{k}(\tau,m,\epsilon)}{P_{m}(\tau)}||_{(k_{1},\beta,\mu,\alpha);\bar{D}(0,\check{Q}\mu_{0})},
C_{4}K_{3}^{j}K_{4}^{l}(\frac{1}{q^{\delta}})^{\frac{k (\min_{\ell=1}^{D-1}d_{\ell})}{2}(l+1)l} \\
\times \left.  \max( ||\frac{\psi_{k}(\tau,m,\epsilon)}{P_{m}(\tau)} ||_{(k_{1},\beta,\mu,\alpha);\bar{A}_{\hat{Q}\mu_{1}}},
||\frac{\psi_{k}(\tau,m,\epsilon)}{P_{m}(\tau)} ||_{(k_{1},\beta,\mu,\alpha);P\Omega_{0}^{p}}) \right)\\
\leq \hat{C}_{\overline{S_{\mathfrak{e}_{p',p}}} \cap \bar{D}(0,\hat{Q}\mu_{1}),k,k_{1}}\\
\times
\max( ||\frac{\psi_{k}(\tau,m,\epsilon)}{P_{m}(\tau)}||_{(k_{1},\beta,\mu,\alpha);\bar{D}(0,\check{Q}\mu_{0})},
||\frac{\psi_{k}(\tau,m,\epsilon)}{P_{m}(\tau)} ||_{(k_{1},\beta,\mu,\alpha);\bar{A}_{\hat{Q}\mu_{1}}},
||\frac{\psi_{k}(\tau,m,\epsilon)}{P_{m}(\tau)} ||_{(k_{1},\beta,\mu,\alpha);P\Omega_{0}^{p}} )\\
\times \max( 1, C_{4}) (\max(K_{2},K_{3}))^{j}
(\frac{1}{q^{\delta}})^{\frac{k (\min_{l=1}^{D-1}d_{l})}{2}(h+1)h}
\end{multline}
for all $j \geq 0$ and $0 \leq h \leq j$.
\end{proof}
\begin{lemma}\label{lema5} Let us assume the constraints (\ref{deltaD_deltal_constraints_bf}) hold. Let $j \geq 0$ and
$0 \leq h \leq j$ and let $\Omega$ be one of the
starshaped domains
$S_{\mathfrak{e}_{p',p}} \cap D(0,\hat{Q}\mu_{1,h})$, $D(0,\check{Q}\mu_{0,j})$ or $S_{\mathfrak{e}_{p',p}}$. Then the linear
map $\mathcal{G}_{\epsilon,j}$ satisfies the next properties. We denote
$$ \Omega_{w_{k,j}^{p}} =
\sup_{\epsilon \in D(0,\epsilon_{0})}||w_{k,j}^{p}(\tau,m,\epsilon)||_{(\nu,\beta,\mu,k);\bar{\Omega}},
$$
which turns out to be finite in view of Lemma~\ref{lema4}. For a well chosen (large enough) radius $r_{\mathbf{Q},\mathbf{R}_{\mathbf{D}}}>0$, (small enough) constants $\boldsymbol{\varsigma}_{0,0},\boldsymbol{\varsigma}_{0},\boldsymbol{\varsigma}_{\mathbf{F}}>0$
such that (\ref{varsigma00_C00_cF_small_bf}) holds, there exists $\upsilon>0$ such that\medskip

\noindent {\bf i)} The inclusion
\begin{equation}
\mathcal{G}_{\epsilon,j}(\bar{B}(0,\upsilon \Omega_{w_{k,j}^{p}}))
\subset \bar{B}(0,\upsilon \Omega_{w_{k,j}^{p}} ) \label{G_epsil_j_inclusion}
\end{equation}
holds where $\bar{B}(0,\upsilon \Omega_{w_{k,j}^{p}})$ stands for the closed disc of radius
$\upsilon\Omega_{w_{k,j}^{p}}>0$ centered at 0 in
$F_{(\nu,\beta,\mu,k);\bar{\Omega}}$, for all $\epsilon \in D(0,\epsilon_{0})$.\medskip

\noindent {\bf ii)} We have
\begin{equation}
|| \mathcal{G}_{\epsilon,j}(w_{1}) - \mathcal{G}_{\epsilon,j}(w_{2}) ||_{(\nu,\beta,\mu,k);\bar{\Omega}}
\leq \frac{1}{2} ||w_{1} - w_{2}||_{(\nu,\beta,\mu,k);\bar{\Omega}} \label{G_epsil_j_shrink}
\end{equation}
for all $w_{1},w_{2} \in \bar{B}(0,\upsilon\Omega_{w_{k,j}^{p}})$, and all $\epsilon \in D(0,\epsilon_{0})$.
\end{lemma}
\begin{proof} We first give proof of (\ref{G_epsil_j_inclusion}). Let $\upsilon>0$ and take
$w(\tau,m) \in \bar{B}(0,\upsilon \Omega_{w_{k,j}^{p}})$. For $0 \leq h \leq \boldsymbol{\delta}_{\mathbf{D}}-1$, by means of Proposition~\ref{prop3} and (\ref{low_bounds_P_m_bf}), (\ref{deltaD_deltal_constraints_bf}), we get
\begin{multline}
|| A_{\boldsymbol{\delta}_{\mathbf{D}},h}
\frac{\mathbf{R}_{\mathbf{D}}(im)}{\mathbf{P}_{m}(\tau)\Gamma(\boldsymbol{\delta}_{\mathbf{D}}-h)}\int_{0}^{\tau^{k}}
(\tau^{k}-s)^{\boldsymbol{\delta}_{\mathbf{D}}-h-1}
k^{h} s^{h} w(s^{1/k},m) \frac{ds}{s} ||_{(\nu,\beta,\mu,k);\bar{\Omega}}\\
\leq \frac{ |A_{\boldsymbol{\delta}_{\mathbf{D}},h}| k^{h} E_{2} }{
C_{\mathbf{P}}(r_{\mathbf{Q},\mathbf{R}_{\mathbf{D}}})^{
\frac{1}{(\boldsymbol{\delta}_{\mathbf{D}}-1)k}}\Gamma(\boldsymbol{\delta}_{\mathbf{D}}-h)}
||w(\tau,m)||_{(\nu,\beta,\mu,k);\bar{\Omega}} \leq \frac{ |A_{\boldsymbol{\delta}_{\mathbf{D}},h}| k^{h} E_{2} }{
C_{\mathbf{P}}(r_{\mathbf{Q},\mathbf{R}_{\mathbf{D}}})^{
\frac{1}{(\boldsymbol{\delta}_{\mathbf{D}}-1)k}}\Gamma(\boldsymbol{\delta}_{\mathbf{D}}-h)}
\upsilon \Omega_{w_{k,j}^{p}}. \label{G_epsilon_j_in_ball_1}
\end{multline}
With the help of Proposition~\ref{prop3} and due to the assumptions (\ref{assum_dl_delta_l_Delta_l_bf_Borel}),
(\ref{assum_Q_Rl_bf_Borel}) and (\ref{deltaD_deltal_constraints_bf}), for $1 \leq l \leq \mathbf{D}-1$, we also deduce that
\begin{multline}
|| \frac{\mathbf{R}_{l}(im)}{\mathbf{P}_{m}(\tau)}
\epsilon^{\mathbf{\Delta}_{l}-\mathbf{d}_{l}+\boldsymbol{\delta}_{l}-1}
\frac{1}{\Gamma( \frac{\mathbf{d}_{l,k}}{k} )}
\int_{0}^{\tau^{k}} (\tau^{k}-s)^{\frac{\mathbf{d}_{l,k}}{k}-1}k^{\boldsymbol{\delta}_l}s^{\boldsymbol{\delta}_l}
w(s^{1/k},m) \frac{ds}{s}||_{(\nu,\beta,\mu,k);\bar{\Omega}}\\
\leq \frac{k^{\boldsymbol{\delta}_{l}} E_{2}}{\Gamma( \frac{\mathbf{d}_{l,k}}{k} )
C_{\mathbf{P}}(r_{\mathbf{Q},\mathbf{R}_{\mathbf{D}}})^{
\frac{1}{(\boldsymbol{\delta}_{\mathbf{D}}-1)k}} }
\epsilon_{0}^{\mathbf{\Delta}_{l}-\mathbf{d}_{l}+\boldsymbol{\delta}_{l}-1}
\sup_{m \in \mathbb{R}} \left| \frac{\mathbf{R}_{l}(im)}{\mathbf{R}_{\mathbf{D}}(im)} \right|
||w(\tau,m)||_{(\nu,\beta,\mu,k);\bar{\Omega}}\\
\leq \frac{k^{\boldsymbol{\delta}_{l}} E_{2}}{\Gamma( \frac{\mathbf{d}_{l,k}}{k} )
C_{\mathbf{P}}(r_{\mathbf{Q},\mathbf{R}_{\mathbf{D}}})^{
\frac{1}{(\boldsymbol{\delta}_{\mathbf{D}}-1)k}} }
\epsilon_{0}^{\mathbf{\Delta}_{l}-\mathbf{d}_{l}+\boldsymbol{\delta}_{l}-1}
\sup_{m \in \mathbb{R}} \left| \frac{\mathbf{R}_{l}(im)}{\mathbf{R}_{\mathbf{D}}(im)} \right| \upsilon \Omega_{w_{k,j}^{p}}
\label{G_epsilon_j_in_ball_2}
\end{multline}
together with
\begin{multline}
|| \frac{\mathbf{R}_{l}(im)}{\mathbf{P}_{m}(\tau)}A_{\boldsymbol{\delta}_{l},h}
\epsilon^{\mathbf{\Delta}_{l}-\mathbf{d}_{l}+\boldsymbol{\delta}_{l}-1}
\frac{1}{\Gamma( \frac{\mathbf{d}_{l,k}}{k} + \boldsymbol{\delta}_{l}-h)}\\
\times \int_{0}^{\tau^{k}}
(\tau^{k}-s)^{\frac{\mathbf{d}_{l,k}}{k}+\boldsymbol{\delta}_{l}-h-1}k^{h}s^{h}
w(s^{1/k},m) \frac{ds}{s} ||_{(\nu,\beta,\mu,k);\bar{\Omega}}\\
\leq \frac{|A_{\boldsymbol{\delta}_{l},h}|k^{h}E_{2}}{\Gamma( \frac{\mathbf{d}_{l,k}}{k} + \boldsymbol{\delta}_{l}-h)
C_{\mathbf{P}}(r_{\mathbf{Q},\mathbf{R}_{\mathbf{D}}})^{
\frac{1}{(\boldsymbol{\delta}_{\mathbf{D}}-1)k}} }
\epsilon_{0}^{\mathbf{\Delta}_{l}-\mathbf{d}_{l}+\boldsymbol{\delta}_{l}-1}
\sup_{m \in \mathbb{R}} \left| \frac{\mathbf{R}_{l}(im)}{\mathbf{R}_{\mathbf{D}}(im)} \right|
||w(\tau,m)||_{(\nu,\beta,\mu,k);\bar{\Omega}}\\
\leq \frac{|A_{\boldsymbol{\delta}_{l},h}|k^{h}E_{2}}{\Gamma( \frac{\mathbf{d}_{l,k}}{k} + \boldsymbol{\delta}_{l}-h)
C_{\mathbf{P}}(r_{\mathbf{Q},\mathbf{R}_{\mathbf{D}}})^{
\frac{1}{(\boldsymbol{\delta}_{\mathbf{D}}-1)k}} }
\epsilon_{0}^{\mathbf{\Delta}_{l}-\mathbf{d}_{l}+\boldsymbol{\delta}_{l}-1}
\sup_{m \in \mathbb{R}} \left| \frac{\mathbf{R}_{l}(im)}{\mathbf{R}_{\mathbf{D}}(im)} \right| \upsilon \Omega_{w_{k,j}^{p}}
\label{G_epsilon_j_in_ball_3}
\end{multline}
for all $1\le l\le \mathbf{D}$, $1 \leq h \leq \boldsymbol{\delta}_{l}-1$. Using Proposition 4 under the hypothesis
(\ref{assum_Q_Rl_bf_Borel}), we can write
\begin{multline}
|| \epsilon^{-1}\frac{1}{\mathbf{P}_{m}(\tau)\Gamma(1 + \frac{1}{k})}\\
\times \int_{0}^{\tau^{k}}
(\tau^{k}-s)^{1/k} \frac{\mathbf{c}_{0,0}(\epsilon)}{(2\pi)^{1/2}}
( \int_{-\infty}^{+\infty} \mathbf{C}_{0,0}(m-m_{1},\epsilon)\mathbf{R}_{0}(im_{1})
w(s^{1/k},m_{1}) dm_{1} )\frac{ds}{s} ||_{(\nu,\beta,\mu,k);\bar{\Omega}}\\
\leq \frac{\boldsymbol{\varsigma}_{0,0}\boldsymbol{\varsigma}_{0}E_{3}}{(2\pi)^{1/2}\Gamma(1 + \frac{1}{k})
C_{\mathbf{P}}(r_{\mathbf{Q},\mathbf{R}_{\mathbf{D}}})^{
\frac{1}{(\boldsymbol{\delta}_{\mathbf{D}}-1)k}} }||w(\tau,m)||_{(\nu,\beta,\mu,k);\bar{\Omega}}\\
\leq \frac{\boldsymbol{\varsigma}_{0,0}\boldsymbol{\varsigma}_{0}E_{3}}{(2\pi)^{1/2}\Gamma(1 + \frac{1}{k})
C_{\mathbf{P}}(r_{\mathbf{Q},\mathbf{R}_{\mathbf{D}}})^{
\frac{1}{(\boldsymbol{\delta}_{\mathbf{D}}-1)k}} } \upsilon \Omega_{w_{k,j}^{p}}. \label{G_epsilon_j_in_ball_4}
\end{multline}
Finally, according again to Proposition 3, we obtain
\begin{multline}
|| \epsilon^{-1} \mathbf{c}_{\mathbf{F}}(\epsilon) \frac{1}{\mathbf{P}_{m}(\tau)\Gamma(1 + \frac{1}{k})}\int_{0}^{\tau^{k}}
(\tau^{k}-s)^{1/k} w_{k,j}^{p}(s^{1/k},m,\epsilon) \frac{ds}{s} ||_{(\nu,\beta,\mu,k);\bar{\Omega}}\\
\leq \frac{\boldsymbol{\varsigma}_{\mathbf{F}}E_{2}}{\Gamma(1 + \frac{1}{k})
C_{\mathbf{P}}(r_{\mathbf{Q},\mathbf{R}_{\mathbf{D}}})^{
\frac{1}{(\boldsymbol{\delta}_{\mathbf{D}}-1)k}} \mathrm{min}_{m \in \mathbb{R}} |\mathbf{R}_{\mathbf{D}}(im)| }
||w_{k,j}^{p}(\tau,m,\epsilon)||_{(\nu,\beta,\mu,k);\bar{\Omega}}\\
\leq \frac{\boldsymbol{\varsigma}_{\mathbf{F}}E_{2}}{\Gamma(1 + \frac{1}{k})
C_{\mathbf{P}}(r_{\mathbf{Q},\mathbf{R}_{\mathbf{D}}})^{
\frac{1}{(\boldsymbol{\delta}_{\mathbf{D}}-1)k}} \mathrm{min}_{m \in \mathbb{R}} |\mathbf{R}_{\mathbf{D}}(im)| }
\Omega_{w_{k,j}^{p}}. \label{G_epsilon_j_in_ball_5}
\end{multline}
We select now $\upsilon,r_{\mathbf{Q},\mathbf{R}_{\mathbf{D}}}>0$ and
$\boldsymbol{\varsigma}_{\mathbf{F}}>0$ in such a way that
\begin{multline}
\sum_{1 \leq h \leq \boldsymbol{\delta}_{\mathbf{D}}-1} \frac{ |A_{\boldsymbol{\delta}_{\mathbf{D}},h}| k^{h} E_{2} }{
C_{\mathbf{P}}(r_{\mathbf{Q},\mathbf{R}_{\mathbf{D}}})^{
\frac{1}{(\boldsymbol{\delta}_{\mathbf{D}}-1)k}}\Gamma(\boldsymbol{\delta}_{\mathbf{D}}-h)} \upsilon\\
+ \sum_{l=1}^{\mathbf{D}-1} \frac{k^{\boldsymbol{\delta}_{l}} E_{2}}{\Gamma( \frac{\mathbf{d}_{l,k}}{k} )
C_{\mathbf{P}}(r_{\mathbf{Q},\mathbf{R}_{\mathbf{D}}})^{
\frac{1}{(\boldsymbol{\delta}_{\mathbf{D}}-1)k}} }
\epsilon_{0}^{\mathbf{\Delta}_{l}-\mathbf{d}_{l}+\boldsymbol{\delta}_{l}-1}
\sup_{m \in \mathbb{R}} \left| \frac{\mathbf{R}_{l}(im)}{\mathbf{R}_{\mathbf{D}}(im)} \right| \upsilon\\
+ \sum_{l=1}^{\mathbf{D}-1} \sum_{1 \leq h \leq \boldsymbol{\delta}_{l}-1}
\frac{|A_{\boldsymbol{\delta}_{l},h}|k^{h}E_{2}}{\Gamma( \frac{\mathbf{d}_{l,k}}{k} + \boldsymbol{\delta}_{l}-h)
C_{\mathbf{P}}(r_{\mathbf{Q},\mathbf{R}_{\mathbf{D}}})^{
\frac{1}{(\boldsymbol{\delta}_{\mathbf{D}}-1)k}} }
\epsilon_{0}^{\mathbf{\Delta}_{l}-\mathbf{d}_{l}+\boldsymbol{\delta}_{l}-1}
\sup_{m \in \mathbb{R}} \left| \frac{\mathbf{R}_{l}(im)}{\mathbf{R}_{\mathbf{D}}(im)} \right| \upsilon\\
+ \frac{\boldsymbol{\varsigma}_{0,0}\boldsymbol{\varsigma}_{0}E_{3}}{(2\pi)^{1/2}\Gamma(1 + \frac{1}{k})
C_{\mathbf{P}}(r_{\mathbf{Q},\mathbf{R}_{\mathbf{D}}})^{
\frac{1}{(\boldsymbol{\delta}_{\mathbf{D}}-1)k}} } \upsilon\\
+ \frac{\boldsymbol{\varsigma}_{\mathbf{F}}E_{2}}{\Gamma(1 + \frac{1}{k})
C_{\mathbf{P}}(r_{\mathbf{Q},\mathbf{R}_{\mathbf{D}}})^{
\frac{1}{(\boldsymbol{\delta}_{\mathbf{D}}-1)k}} \mathrm{min}_{m \in \mathbb{R}} |R_{D}(im)| } \leq \upsilon.
\label{constraint_G_epsilon_j_in_ball}
\end{multline}
Gathering all the norm estimates (\ref{G_epsilon_j_in_ball_1}), (\ref{G_epsilon_j_in_ball_2}),
(\ref{G_epsilon_j_in_ball_3}), (\ref{G_epsilon_j_in_ball_4}) and (\ref{G_epsilon_j_in_ball_5})
under the constraint (\ref{constraint_G_epsilon_j_in_ball}), one gets (\ref{G_epsil_j_inclusion}).\medskip

We now check the second property (\ref{G_epsil_j_shrink}). Let $w_{1}(\tau,m),w_{2}(\tau,m)$ be in
$\bar{B}(0,\upsilon\Omega_{w_{k,j}^{p}})$. From the above estimates (\ref{G_epsilon_j_in_ball_1}), (\ref{G_epsilon_j_in_ball_2}),
(\ref{G_epsilon_j_in_ball_3}), (\ref{G_epsilon_j_in_ball_4}) and (\ref{G_epsilon_j_in_ball_5}), one can affirm that
\begin{multline}
|| A_{\boldsymbol{\delta}_{\mathbf{D}},h}
\frac{\mathbf{R}_{\mathbf{D}}(im)}{\mathbf{P}_{m}(\tau)\Gamma(\boldsymbol{\delta}_{\mathbf{D}}-h)}\int_{0}^{\tau^{k}}
(\tau^{k}-s)^{\boldsymbol{\delta}_{\mathbf{D}}-h-1}
k^{h} s^{h}\\
\times (w_{1}(s^{1/k},m)-w_{2}(s^{1/k},m)) \frac{ds}{s} ||_{(\nu,\beta,\mu,k);\bar{\Omega}}\\
\leq \frac{ |A_{\boldsymbol{\delta}_{\mathbf{D}},h}| k^{h} E_{2} }{
C_{\mathbf{P}}(r_{\mathbf{Q},\mathbf{R}_{\mathbf{D}}})^{
\frac{1}{(\boldsymbol{\delta}_{\mathbf{D}}-1)k}}\Gamma(\boldsymbol{\delta}_{\mathbf{D}}-h)}
||w_{1}(\tau,m)-w_{2}(\tau,m)||_{(\nu,\beta,\mu,k);\bar{\Omega}} \label{G_epsilon_j_shrink_1}
\end{multline}
for all $1 \leq h \leq \boldsymbol{\delta}_{\mathbf{D}}-1$, 
\begin{multline}
|| \frac{\mathbf{R}_{l}(im)}{\mathbf{P}_{m}(\tau)}
\epsilon^{\mathbf{\Delta}_{l}-\mathbf{d}_{l}+\boldsymbol{\delta}_{l}-1}
\frac{1}{\Gamma( \frac{\mathbf{d}_{l,k}}{k} )}
\int_{0}^{\tau^{k}} (\tau^{k}-s)^{\frac{\mathbf{d}_{l,k}}{k}-1}k^{\boldsymbol{\delta}_l}s^{\boldsymbol{\delta}_l}\\
\times (w_{1}(s^{1/k},m) - w_{2}(s^{1/k},m)) \frac{ds}{s}||_{(\nu,\beta,\mu,k);\bar{\Omega}}\\
\leq \frac{k^{\boldsymbol{\delta}_{l}} E_{2}}{\Gamma( \frac{\mathbf{d}_{l,k}}{k} )
C_{\mathbf{P}}(r_{\mathbf{Q},\mathbf{R}_{\mathbf{D}}})^{
\frac{1}{(\boldsymbol{\delta}_{\mathbf{D}}-1)k}} }
\epsilon_{0}^{\mathbf{\Delta}_{l}-\mathbf{d}_{l}+\boldsymbol{\delta}_{l}-1}
\sup_{m \in \mathbb{R}} \left| \frac{\mathbf{R}_{l}(im)}{\mathbf{R}_{\mathbf{D}}(im)} \right|
||w_{1}(\tau,m) - w_{2}(\tau,m)||_{(\nu,\beta,\mu,k);\bar{\Omega}}\\
\label{G_epsilon_j_shrink_2}
\end{multline}
for all $1 \leq l \leq \mathbf{D}-1$ with
\begin{multline}
|| \frac{\mathbf{R}_{l}(im)}{\mathbf{P}_{m}(\tau)}A_{\boldsymbol{\delta}_{l},h}
\epsilon^{\mathbf{\Delta}_{l}-\mathbf{d}_{l}+\boldsymbol{\delta}_{l}-1}
\frac{1}{\Gamma( \frac{\mathbf{d}_{l,k}}{k} + \boldsymbol{\delta}_{l}-h)}\\
\times \int_{0}^{\tau^{k}}
(\tau^{k}-s)^{\frac{\mathbf{d}_{l,k}}{k}+\boldsymbol{\delta}_{l}-h-1} k^{h}s^{h}
(w_{1}(s^{1/k},m)-w_{2}(s^{1/k},m)) \frac{ds}{s} ||_{(\nu,\beta,\mu,k);\bar{\Omega}}\\
\leq \frac{|A_{\boldsymbol{\delta}_{l},h}|k^{h}E_{2}}{\Gamma( \frac{\mathbf{d}_{l,k}}{k} + \boldsymbol{\delta}_{l}-h)
C_{\mathbf{P}}(r_{\mathbf{Q},\mathbf{R}_{\mathbf{D}}})^{
\frac{1}{(\boldsymbol{\delta}_{\mathbf{D}}-1)k}} }
\epsilon_{0}^{\mathbf{\Delta}_{l}-\mathbf{d}_{l}+\boldsymbol{\delta}_{l}-1}
\sup_{m \in \mathbb{R}} \left| \frac{\mathbf{R}_{l}(im)}{\mathbf{R}_{\mathbf{D}}(im)} \right|\\
\times ||w_{1}(\tau,m) - w_{2}(\tau,m)||_{(\nu,\beta,\mu,k);\bar{\Omega}}\\
\label{G_epsilon_j_shrink_3}
\end{multline}
for all $1 \leq l \leq \mathbf{D}-1$, $1 \leq h \leq \boldsymbol{\delta}_{l}-1$ and
\begin{multline}
|| \epsilon^{-1}\frac{1}{\mathbf{P}_{m}(\tau)\Gamma(1 + \frac{1}{k})}\\
\times \int_{0}^{\tau^{k}}
(\tau^{k}-s)^{1/k} \frac{\mathbf{c}_{0,0}(\epsilon)}{(2\pi)^{1/2}}
( \int_{-\infty}^{+\infty} \mathbf{C}_{0,0}(m-m_{1},\epsilon)\mathbf{R}_{0}(im_{1})\\
\times (w_{1}(s^{1/k},m_{1}) - w_{2}(s^{1/k},m_{1})) dm_{1} )\frac{ds}{s} ||_{(\nu,\beta,\mu,k);\bar{\Omega}}\\
\leq \frac{\boldsymbol{\varsigma}_{0,0}\boldsymbol{\varsigma}_{0}E_{3}}{(2\pi)^{1/2}\Gamma(1 + \frac{1}{k})
C_{\mathbf{P}}(r_{\mathbf{Q},\mathbf{R}_{\mathbf{D}}})^{
\frac{1}{(\boldsymbol{\delta}_{\mathbf{D}}-1)k}} }||w_{1}(\tau,m) - w_{2}(\tau,m)||_{(\nu,\beta,\mu,k);\bar{\Omega}}
. \label{G_epsilon_j_shrink_4}
\end{multline}
We fix $r_{\mathbf{Q},\mathbf{R}_{\mathbf{D}}}>0$ such that
inequality
\begin{multline}
\sum_{1 \leq h \leq \boldsymbol{\delta}_{\mathbf{D}}-1} \frac{ |A_{\boldsymbol{\delta}_{\mathbf{D}},h}| k^{h} E_{2} }{
C_{\mathbf{P}}(r_{\mathbf{Q},\mathbf{R}_{\mathbf{D}}})^{
\frac{1}{(\boldsymbol{\delta}_{\mathbf{D}}-1)k}}\Gamma(\boldsymbol{\delta}_{\mathbf{D}}-h)}\\
+ \sum_{l=1}^{\mathbf{D}-1} \frac{k^{\boldsymbol{\delta}_{l}} E_{2}}{\Gamma( \frac{\mathbf{d}_{l,k}}{k} )
C_{\mathbf{P}}(r_{\mathbf{Q},\mathbf{R}_{\mathbf{D}}})^{
\frac{1}{(\boldsymbol{\delta}_{\mathbf{D}}-1)k}} }
\epsilon_{0}^{\mathbf{\Delta}_{l}-\mathbf{d}_{l}+\boldsymbol{\delta}_{l}-1}
\sup_{m \in \mathbb{R}} \left| \frac{\mathbf{R}_{l}(im)}{\mathbf{R}_{\mathbf{D}}(im)} \right| \\
+ \sum_{l=1}^{\mathbf{D}-1} \sum_{1 \leq h \leq \boldsymbol{\delta}_{l}-1}
\frac{|A_{\boldsymbol{\delta}_{l},h}|k^{h}E_{2}}{\Gamma( \frac{\mathbf{d}_{l,k}}{k} + \boldsymbol{\delta}_{l}-h)
C_{\mathbf{P}}(r_{\mathbf{Q},\mathbf{R}_{\mathbf{D}}})^{
\frac{1}{(\boldsymbol{\delta}_{\mathbf{D}}-1)k}} }
\epsilon_{0}^{\mathbf{\Delta}_{l}-\mathbf{d}_{l}+\boldsymbol{\delta}_{l}-1}
\sup_{m \in \mathbb{R}} \left| \frac{\mathbf{R}_{l}(im)}{\mathbf{R}_{\mathbf{D}}(im)} \right| \\
+ \frac{\boldsymbol{\varsigma}_{0,0}\boldsymbol{\varsigma}_{0}E_{3}}{(2\pi)^{1/2}\Gamma(1 + \frac{1}{k})
C_{\mathbf{P}}(r_{\mathbf{Q},\mathbf{R}_{\mathbf{D}}})^{
\frac{1}{(\boldsymbol{\delta}_{\mathbf{D}}-1)k}} } \leq \frac{1}{2}
\label{constraint_G_epsilon_j_shrink}
\end{multline}
holds. Bearing in mind the above estimates (\ref{G_epsilon_j_shrink_1}), (\ref{G_epsilon_j_shrink_2}),
(\ref{G_epsilon_j_shrink_3}) and (\ref{G_epsilon_j_shrink_4}) under the restriction
(\ref{constraint_G_epsilon_j_shrink}), we deduce (\ref{G_epsil_j_shrink}).

Finally, we choose $\upsilon,r_{\mathbf{Q},\mathbf{R}_{\mathbf{D}}}>0$ and
$\boldsymbol{\varsigma}_{\mathbf{F}}>0$ such that both
(\ref{constraint_G_epsilon_j_in_ball}) and (\ref{constraint_G_epsilon_j_shrink}) are fulfilled. This concludes the proof.
\end{proof}
Let $j \geq 0$ and $0 \leq h \leq j$ and $\Omega$ be one of the
starshaped domains
$S_{\mathfrak{e}_{p',p}} \cap D(0,\hat{Q}\mu_{1,h})$, $D(0,\check{Q}\mu_{0,j})$ or $S_{\mathfrak{e}_{p',p}}$. We consider
the ball $\bar{B}(0,\upsilon \Omega_{w_{k,j}^{p}}) \subset
F_{(\nu,\beta,\mu,k);\bar{\Omega}}$ constructed in Lemma~\ref{lema5} which is a complete
metric space for the norm $||.||_{(\nu,\beta,\mu,k);\bar{\Omega}}$. We get that
$\mathcal{G}_{\epsilon,j}$ is a
contractive map from $\bar{B}(0,\upsilon \Omega_{w_{k,j}^{p}})$ into itself. Due to the classical contractive
mapping theorem, we deduce that
the map $\mathcal{G}_{\epsilon,j}$ has a unique fixed point denoted
$v_{k,j}^{\mathfrak{e}_{p',p},\mathfrak{d}_{p}}(\tau,m,\epsilon)$ meaning that
\begin{equation}
\mathcal{G}_{\epsilon,j}(v_{k,j}^{\mathfrak{e}_{p',p},\mathfrak{d}_{p}}(\tau,m,\epsilon)) =
v_{k,j}^{\mathfrak{e}_{p',p},\mathfrak{d}_{p}}(\tau,m,\epsilon) \label{fixed_point_G_epsilon_j}
\end{equation}
that belongs to $\bar{B}(0,\upsilon \Omega_{w_{k,j}^{p}})$, for all $\epsilon \in D(0,\epsilon_{0})$.
Moreover, the function
$v_{k,j}^{\mathfrak{e}_{p',p},\mathfrak{d}_{p}}(\tau,m,\epsilon)$ depends holomorphically
on $\epsilon$ in $D(0,\epsilon_{0})$. In particular, from the estimates
(\ref{compare_norm_qexp_norm_exp_wkj_on_sector}) we deduce that
\begin{multline}
||v_{k,j}^{\mathfrak{e}_{p',p},\mathfrak{d}_{p}}(\tau,m,\epsilon)||_{(\nu,\beta,\mu,k);\overline{S_{\mathfrak{e}_{p',p}}}}
\leq \upsilon \sup_{\epsilon \in D(0,\epsilon_{0})}
||w_{k,j}^{p}(\tau,m,\epsilon)||_{(\nu,\beta,\mu,k);\overline{S_{\mathfrak{e}_{p',p}}}}\\
\leq \upsilon C_{\Theta_{\hat{Q}\mu_{1}}^{p},k,k_{1}}K_{1}^{j}
\sup_{\epsilon \in D(0,\epsilon_{0})}
||\frac{\psi_{k}(\tau,m,\epsilon)}{P_{m}(\tau)}||_{(k_{1},\beta,\mu,\alpha);\Theta_{\hat{Q}\mu_{1}}^{p}}
\label{norm_exp_v_kj_ep'_dp_on_unbounded_sectors}
\end{multline}
for all $j \geq 0$, all $\epsilon \in D(0,\epsilon_{0})$. Furthermore, regarding (\ref{compare_norm_qexp_norm_exp_wkj_on_disc_j}), we also get that
\begin{multline}
||v_{k,j}^{\mathfrak{e}_{p',p},\mathfrak{d}_{p}}(\tau,m,\epsilon)||_{(\nu,\beta,\mu,k);\bar{D}(0,\check{Q}\mu_{0,j})}
\leq \upsilon \sup_{\epsilon \in D(0,\epsilon_{0})}
||w_{k,j}^{p}(\tau,m,\epsilon)||_{(\nu,\beta,\mu,k);\bar{D}(0,\check{Q}\mu_{0,j})}\\
\leq \upsilon C_{\bar{D}(0,\check{Q}\mu_{0}),k,k_{1}}
K_{2}^{j}(\frac{1}{q^{\delta}})^{\frac{k \min_{l=1}^{D-1}d_{l}}{2}(j+1)j} \sup_{\epsilon \in D(0,\epsilon_{0})}
||\frac{\psi_{k}(\tau,m,\epsilon)}{P_{m}(\tau)}||_{(k_{1},\beta,\mu,\alpha);\bar{D}(0,\check{Q}\mu_{0})}
\end{multline}
for all $j \geq 0$, all $\epsilon \in D(0,\epsilon_{0})$. In addition to that, as a consequence of
(\ref{compare_norm_qexp_norm_exp_wkj_on_disc_h_and_sector}), we obtain
\begin{multline}
||v_{k,j}^{\mathfrak{e}_{p',p},\mathfrak{d}_{p}}(\tau,m,\epsilon)
||_{(\nu,\beta,\mu,k);\overline{S_{\mathfrak{e}_{p',p}}} \cap \bar{D}(0,\hat{Q}\mu_{1,h})}
\leq \upsilon \sup_{\epsilon \in D(0,\epsilon_{0})}
||w_{k,j}^{p}(\tau,m,\epsilon)||_{(\nu,\beta,\mu,k);\overline{S_{\mathfrak{e}_{p',p}}} \cap \bar{D}(0,\hat{Q}\mu_{1,h})}\\
\leq \upsilon C_{\overline{S_{\mathfrak{e}_{p',p}}} \cap \bar{D}(0,\hat{Q}\mu_{1}),k,k_{1}}
(\max(K_{2},K_{3}))^{j}(\frac{1}{q^{\delta}})^{\frac{k \min_{l=1}^{D-1}d_{l}}{2}(h+1)h}\\
\times \left\{\sup_{\epsilon \in D(0,\epsilon_{0})}
\max( ||\frac{\psi_{k}(\tau,m,\epsilon)}{P_{m}(\tau)} ||_{(k_{1},\beta,\mu,\alpha);\bar{D}(0,\check{Q}\mu_{0})},
||\frac{\psi_{k}(\tau,m,\epsilon)}{P_{m}(\tau)} ||_{(k_{1},\beta,\mu,\alpha);\bar{A}_{\hat{Q}\mu_{1}}},\right.\\
\left.||\frac{\psi_{k}(\tau,m,\epsilon)}{P_{m}(\tau)} ||_{(k_{1},\beta,\mu,\alpha);P\Omega_{0}^{p}} )\right\}
\end{multline}
for all $j \geq 0$, all $0 \leq h \leq j$ and all $\epsilon \in D(0,\epsilon_{0})$.\medskip

We recall (see Proposition~\ref{prop6}) that the function $w_{k}^{p}(\tau,m,\epsilon)$ can be written as the sum
$w_{k}^{p}(\tau,m,\epsilon) = \sum_{j \geq 0} w_{k,j}^{p}(\tau,m,\epsilon)$, which converges in the Banach space
$\mathrm{Exp}_{(k_{1},\beta,\mu,\alpha);\Theta_{\hat{Q}\mu_{1}}^{p}}^q$. From the estimates
(\ref{compare_norm_qexp_norm_exp_wkj_on_sector}), we can affirm that this sum also converges in the Banach space
$F_{(\nu,\beta,\mu,k);\overline{S_{\mathfrak{e}_{p',p}}}}$. On the other hand, we can define
\begin{equation}
v_{k}^{\mathfrak{e}_{p',p},\mathfrak{d}_{p}}(\tau,m,\epsilon) := \sum_{j \geq 0}
v_{k,j}^{\mathfrak{e}_{p',p},\mathfrak{d}_{p}}(\tau,m,\epsilon), 
\end{equation}
which turn out to be a convergent series in the Banach space $F_{(\nu,\beta,\mu,k);\overline{S_{\mathfrak{e}_{p',p}}}}$
as above due to the upper bounds (\ref{norm_exp_v_kj_ep'_dp_on_unbounded_sectors}). Furthermore, from
(\ref{fixed_point_G_epsilon_j}), we see that
$v_{k}^{\mathfrak{e}_{p',p},\mathfrak{d}_{p}}(\tau,m,\epsilon)$ satisfies the convolution equation
$$  v_{k}^{\mathfrak{e}_{p',p},\mathfrak{d}_{p}}(\tau,m,\epsilon) =
\sum_{j \geq 0} \mathcal{G}_{\epsilon,j}(v_{k,j}^{\mathfrak{e}_{p',p},\mathfrak{d}_{p}}(\tau,m,\epsilon))
$$
which entails that $v_{k}^{\mathfrak{e}_{p',p},\mathfrak{d}_{p}}$ solves the problem
(\ref{k_Borel_equation_analytic_bf}) in the Banach space
$F_{(\nu,\beta,\mu,k);\overline{S_{\mathfrak{e}_{p',p}}}}$. Proposition~\ref{prop9} follows from here.
\end{proof}

\section{Analytic solutions of a differential initial value Cauchy problem}\label{seccion9}

Let $k \geq 1$ be the integer defined above in Section~\ref{seccion41} and let $\mathbf{D} \geq 2$ be an integer. For

$1 \leq l \leq \mathbf{D}$, let
$\mathbf{d}_{l}$,$\boldsymbol{\delta}_{l}$,$\mathbf{\Delta}_{l} \geq 0$ be nonnegative integers.
We assume that 
\begin{equation}
1 = \boldsymbol{\delta}_{1} \ \ , \ \ \boldsymbol{\delta}_{l} < \boldsymbol{\delta}_{l+1},  \label{assum_delta_l_bf}
\end{equation}
for all $1 \leq l \leq \mathbf{D}-1$. We also make the assumption that
\begin{equation}
\mathbf{d}_{\mathbf{D}} = (\boldsymbol{\delta}_{\mathbf{D}}-1)(k+1) \ \ , \ \
\mathbf{d}_{l} > (\boldsymbol{\delta}_{l}-1)(k+1) \ \ , \ \
\mathbf{\Delta}_{l} - \mathbf{d}_{l} + \boldsymbol{\delta}_{l} - 1 \geq 0 \ \ , \ \
\mathbf{\Delta}_{\mathbf{D}} = \mathbf{d}_{\mathbf{D}} - \boldsymbol{\delta}_{\mathbf{D}} + 1
\label{assum_dl_delta_l_Delta_l_bf}
\end{equation}
for all $1 \leq l \leq \mathbf{D}-1$. We define $\boldsymbol{d}_{l,k}$ in the form (\ref{defin_d_l_k_bold_Borel}), for all $1 \leq l \leq \mathbf{D}-1$. Let $\mathbf{Q}(X),\mathbf{R}_{l}(X) \in \mathbb{C}[X]$,
$0 \leq l \leq \mathbf{D}$, be polynomials such that
\begin{equation}
\mathrm{deg}(\mathbf{Q}) \geq \mathrm{deg}(\mathbf{R}_{\mathbf{D}}) \geq \mathrm{deg}(\mathbf{R}_{l}) \ \ , \ \
\mathbf{Q}(im) \neq 0 \ \ , \ \ \mathbf{R}_{\mathbf{D}}(im) \neq 0 \label{assum_Q_Rl_bf}
\end{equation}
for all $m \in \mathbb{R}$, all $0 \leq l \leq \mathbf{D}-1$. We assume these polynomials satisfy the conditions stated at the beginning of Section~\ref{seccion8}. Let $\beta,\mu>0$ be the integers considered in Section~\ref{seccion8}.
We consider a function $m \mapsto \mathbf{C}_{0,0}(m,\epsilon)$ that belong to the Banach space $E_{(\beta,\mu)}$ and which
depends holomorphically on $\epsilon \in D(0,\epsilon_{0})$. Let $\mathbf{c}_{0,0}(\epsilon)$ and
$\mathbf{c}_{\mathbf{F}}(\epsilon)$ be bounded holomorphic functions on $D(0,\epsilon_{0})$ which vanish at the origin
$\epsilon=0$.

For all $0 \leq p \leq \varsigma-1$, we consider the sector $U_{\mathfrak{d}_{p},\theta,\epsilon_{0}r_{\mathcal{T}}}$
described in Definition~\ref{defi7} and we set the function
\begin{equation}\label{e1963}
U_{p}(T,m,\epsilon) = k\int_{L_{\gamma_{p}}} w_{k}^{p}(u,m,\epsilon) e^{-(\frac{u}{T})^{k}} \frac{du}{u} 
\end{equation}
for all $T \in U_{\mathfrak{d}_{p},\theta,\epsilon_{0}r_{\mathcal{T}}}$, $m \in \mathbb{R}$ and
$\epsilon \in \mathcal{E}_{p}$, which is the $m_{k}-$Laplace transform along a halfline
$L_{\gamma_{p}} = \mathbb{R}_{+}e^{\sqrt{-1}\gamma_{p}} \subset U_{\mathfrak{d}_{p}} \cup \{ 0 \}$
of the function $w_{k}^{p}(u,m,\epsilon)$
constructed in Proposition~\ref{prop6} that solves the problem (\ref{q_diff_conv_init_v_prob}). Provided the construction and geometric assumptions on the polynomial $\mathbf{P}_m(\tau)$ stated in Section~\ref{seccion8} hold, for every $0\le p\le \varsigma -1$, we consider a finite set of real numbers $\mathfrak{e}_{p',p}$, $0 \leq p' \leq \chi_p-1$ and a family of unbounded
sectors $S_{\mathfrak{e}_{p',p}}$ centered at 0 with bisecting
direction $\mathfrak{e}_{p',p}$, a small closed disc $\bar{D}(0,\rho)$ in such a way that the sector $S_{\mathbf{Q},\mathbf{R}_{\mathbf{D}}}$ satisfies the geometric assumptions 1), 2) and 3) in Section~\ref{seccion8}. We consider the next
linear initial value problem
\begin{multline}
\mathbf{Q}(im)(\partial_{T}Y_{\mathfrak{e}_{p',p},\mathfrak{d}_{p}}(T,m,\epsilon) ) =
\sum_{l=1}^{\mathbf{D}} \mathbf{R}_{l}(im) \epsilon^{\mathbf{\Delta}_{l} - \mathbf{d}_{l} + \boldsymbol{\delta}_{l} - 1}
T^{\mathbf{d}_{l}} \partial_{T}^{\boldsymbol{\delta}_l}Y_{\mathfrak{e}_{p',p},\mathfrak{d}_{p}}(T,m,\epsilon)\\
+ \epsilon^{-1}\frac{\mathbf{c}_{0,0}(\epsilon)}{(2\pi)^{1/2}}\int_{-\infty}^{+\infty}\mathbf{C}_{0,0}(m-m_{1},\epsilon)
\mathbf{R}_{0}(im_{1})Y_{\mathfrak{e}_{p',p},\mathfrak{d}_{p}}(T,m_{1},\epsilon) dm_{1}
+ \epsilon^{-1}\mathbf{c}_{\mathbf{F}}(\epsilon)U_{p}(T,m,\epsilon)
\label{SCP_bf}
\end{multline}
for given initial data $Y_{\mathfrak{e}_{p',p},\mathfrak{d}_{p}}(0,m,\epsilon) = 0$.

\begin{prop}\label{prop1982}
Let $0\le p\le \varsigma-1$ and $0\le p'\le \chi_p-1$. Assume that (\ref{deltaD_deltal_constraints_bf}) holds for all $1\le l\le \mathbf{D}-1$, and $r_{\mathbf{Q},\mathbf{R}_{\mathbf{D}}},\boldsymbol{\varsigma}_{0,0},\boldsymbol{\varsigma}_{0},\boldsymbol{\varsigma}_{F}>0$ are small constants such that (\ref{varsigma00_C00_cF_small_bf}) holds. Then, the problem (\ref{SCP_bf}) under initial data $Y_{\mathfrak{e}_{p',p},\mathfrak{d}_{p}}(0,m,\epsilon)=0$ admits a solution $Y_{\mathfrak{e}_{p',p},\mathfrak{d}_{p}}(T,m,\epsilon)$, holomorphic and bounded with respect to $T$ in the sector $S_{\mathfrak{e}_{p',p},\theta, R^{1/k}}$ (see (\ref{e690})), where $\frac{\pi}{k}<\theta<\frac{\pi}{k}+2\delta$, for some $\delta$, and some $R>0$, continuous with respect to $m$ in $\mathbb{R}$, and holomorphic with respect to $\epsilon$ in $D(0,\epsilon_0)$. 
\end{prop}
\begin{proof}
Using the formula (8.7) from \cite{taya}, p. 3630, we can expand the operators
$T^{\boldsymbol{\delta}_{l}(k+1)} \partial_{T}^{\boldsymbol{\delta}_l}$ in the form
\begin{equation}
T^{\boldsymbol{\delta}_{l}(k+1)} \partial_{T}^{\boldsymbol{\delta}_l} = (T^{k+1}\partial_{T})^{\boldsymbol{\delta}_l} +
\sum_{1 \leq p \leq \boldsymbol{\delta}_{l}-1} A_{\boldsymbol{\delta}_{l},p} T^{k(\boldsymbol{\delta}_{l}-p)}
(T^{k+1}\partial_{T})^{p}
\label{expand_op_diff_bf}
\end{equation}
where $A_{\boldsymbol{\delta}_{l},p}$, $p=1,\ldots,\boldsymbol{\delta}_{l}-1$ are real numbers, for all
$1 \leq l \leq \mathbf{D}$. Multiplying the equation (\ref{SCP_bf}) by $T^{k+1}$ and using
(\ref{expand_op_diff_bf}), (\ref{defin_d_l_k_bold_Borel}) we can rewrite the equation (\ref{SCP_bf}) in the form
\begin{multline}
\mathbf{Q}(im)( T^{k+1}\partial_{T}Y_{\mathfrak{e}_{p',p},\mathfrak{d}_{p}}(T,m,\epsilon) ) \\
= \sum_{l=1}^{\mathbf{D}} \mathbf{R}_{l}(im)\left( \epsilon^{\mathbf{\Delta}_{l} - \mathbf{d}_{l} +
\boldsymbol{\delta}_{l} - 1}T^{\mathbf{d}_{l,k}}(T^{k+1}\partial_{T})^{\boldsymbol{\delta}_l}
Y_{\mathfrak{e}_{p',p},\mathfrak{d}_{p}}(T,m,\epsilon) \right.
\\+ \sum_{1 \leq h \leq \boldsymbol{\delta}_{l}-1} A_{\boldsymbol{\delta}_{l},h}
\left. \epsilon^{\mathbf{\Delta}_{l}-\mathbf{d}_{l}+\boldsymbol{\delta}_{l}-1} T^{k(\boldsymbol{\delta}_{l}-h)
+ \mathbf{d}_{l,k}}(T^{k+1}\partial_{T})^{h}Y_{\mathfrak{e}_{p',p},\mathfrak{d}_{p}}(T,m,\epsilon) \right)\\
+ \epsilon^{-1}T^{k+1}
\frac{\mathbf{c}_{0,0}(\epsilon)}{(2\pi)^{1/2}}\int_{-\infty}^{+\infty} \mathbf{C}_{0,0}(m-m_{1},\epsilon)
\mathbf{R}_{0}(im_{1})Y_{\mathfrak{e}_{p',p},\mathfrak{d}_{p}}(T,m_{1},\epsilon) dm_{1}
+ \epsilon^{-1}\mathbf{c}_{\mathbf{F}}(\epsilon)T^{k+1}U_{p}(T,m,\epsilon).
\label{SCP_irregular_bf}
\end{multline}

We consider the function $v_k^{\mathfrak{e}_{p',p},\mathfrak{d}_{p}}(\tau,m,\epsilon)$, unique solution of (\ref{k_Borel_equation_analytic_bf}), which belongs to the Banach space $F_{(\nu,\beta,\mu,k);\overline{S_{\mathfrak{e}_{p',p}}}}$, for every $\epsilon\in D(0,\epsilon_0)$. In view of the growth properties at infinity of the elements belonging to such Banach space, one is allowed to apply Laplace transform with respect to $\tau$ along a well chosen direction and define
\begin{equation}\label{e2015}
Y_{\mathfrak{e}_{p',p},\mathfrak{d}_{p}}(T,m,\epsilon)=k\int_{L_{\gamma_{p',p}}}v_{k}^{\mathfrak{e}_{p',p},\mathfrak{d}_p}(u,m,\epsilon)e^{-\left(\frac{u}{T}\right)^{k}}\frac{du}{u},
\end{equation}
for every $\epsilon\in D(0,\epsilon_0)$, $m\in\mathbb{R}$. More precisely, one may choose $$L_{\gamma_{p',p}}=\mathbb{R}_{+}e^{i\gamma_{p'}}\subseteq S_{\mathfrak{e}_{p',p}}\cup\{0\}\subseteq U_{\mathfrak{d}_{p}},$$
where $\gamma_{p',p}$ depends on $T$ and it is chosen in such a way that $\cos(k(\gamma_{p',p}-\arg(T)))\ge\delta_1$, for some $\delta_1>0$. From Definition~\ref{defi4} one gets that $Y_{\mathfrak{e}_{p',p},\mathfrak{d}_{p}}(T,m,\epsilon)$ is a holomorphic and bounded with respect to $T$ in the sector $S_{\mathfrak{e}_{p',p},\theta, R^{1/k}}$, where $\frac{\pi}{k}<\theta<\frac{\pi}{k}+2\delta$, for some $\delta$, and $0<R<\delta_1/\nu$, continuous with respect to $m$ in $\mathbb{R}$. From the construction of $v_{k}^{\mathfrak{e}_{p',p},\mathfrak{d}_p}(u,m,\epsilon)$ we get that $\epsilon\mapsto Y_{\mathfrak{e}_{p',p},\mathfrak{d}_{p}}(T,m,\epsilon)$ is a holomorphic function in $D(0,\epsilon_0)$. The properties of Laplace transform stated in Proposition~\ref{prop7}, and the integral  representation of $U_p(T,m,\epsilon)$ allow us to conclude that $Y_{\mathfrak{e}_{p',p},\mathfrak{d}_{p}}(T,m,\epsilon)$ solves (\ref{SCP_irregular_bf}).  
\end{proof}

\section{Analytic solutions of a differential initial value Cauchy problem with analytic forcing term on sectors and
with complex parameter}\label{seccion10}

In this section, we preserve the constructions and values of the parameters and functions adopted in the last one. We also assume that the conditions on the elements appearing in the construction of problem (\ref{main_q_diff_diff_first}) in Section~\ref{seccion7} hold. We also preserve the construction of the good covering $\{\mathcal{E}_{p}\}_{0\le p\le \varsigma-1}$, and the family $\{(U_{\mathfrak{d}_{p},\theta,\epsilon_0 r_{\mathcal{T}}})_{0\le p\le \varsigma -1},\mathcal{T}\}$ associated to the good covering defined in Section~\ref{seccion7}. 

Let $\{ \mathcal{E}_{p',p}\}_{\substack{0\le p\le \varsigma-1\\0\le p'\le \chi_p-1}}$ be a good covering in $\mathbb{C}^{\star}$ such that $\mathcal{E}_{p',p}\subseteq \mathcal{E}_{p}$ for every $0\le p\le \varsigma-1$ and all $0\le p'\le \chi_p-1$. We also fix an open and bounded sector with vertex at the origin, $\tilde{\mathcal{T}}$, with finite radius $r_{\mathcal{T}}$, such that $\tilde{\mathcal{T}}\subseteq\mathcal{T}$. Let 
$$S_{\mathfrak{e}_{p',p},\theta,\epsilon_0r_{\mathcal{T}}}=\{T\in\mathbb{C}^{\star}:|T|\le \epsilon_0r_{\mathcal{T}},|\mathfrak{e}_{p',p}-\arg(T)|<\frac{\theta}{2}\},$$
where $\frac{\pi}{k}<\theta<\frac{\pi}{k}+Ap(S_{\mathfrak{e}_{p',p}})$, with $Ap(S_{\mathfrak{e}_{p',p}})$ stands for the aperture of $S_{\mathfrak{e}_{p',p}}$, the sector which was chosen in Section~\ref{seccion8} satisfying (\ref{root_cond_1_bf}), (\ref{root_cond_2_bf}), (\ref{e1374}) and (\ref{low_bounds_P_m_bf}). In addition to that, we assume that for every $0\le p\le \varsigma-1$ and all $0\le p'\le \chi_p-1$ one has $\epsilon t\in S_{\mathfrak{e}_{p',p},\theta,\epsilon_0r_\mathcal{T}}$, for every $t\in\mathcal{\tilde{T}}$ and $\epsilon\in\mathcal{E}_{p',p}$, by reducing the aperture of $\tilde{\mathcal{T}}$, if necessary. Under all these assumptions, the set $\{(S_{\mathfrak{e}_{p',p},\theta,\epsilon_0r_\mathcal{T}})_{\substack{0\le p\le \varsigma -1\\0\le p'\le \chi_p-1}},\tilde{\mathcal{T}}\}$ is associated to the good covering $\{ \mathcal{E}_{p',p}\}_{\substack{0\le p\le \varsigma-1\\0\le p'\le \chi_p-1}}$. 

The function $\mathbf{C}_{0,0}(m,\epsilon)$, constructed in Section~\ref{seccion9} is such that one is allowed to apply inverse Fourier transform with respect to $m$ (see Definition~\ref{defi1}). We define
$$\mathbf{c}_0(z,\epsilon):=\mathcal{F}^{-1}\left(m\mapsto\mathbf{C}_{0,0}(m,\epsilon)\right)(z),$$
for every $(z,\epsilon)\in H_{\beta'}\times D(0,\epsilon_0)$, for any $0<\beta'<\beta$. Indeed, $\mathbf{c}_{0}(z,\epsilon)$ turns out to be a  holomorphic function on $H_{\beta'}\times D(0,\epsilon_0)$, for any $0<\beta'<\beta$.

For every $0\le p\le \varsigma-1$ and $0\le p'\le \chi_p-1$, we consider the following initial value problem
\begin{multline}
\mathbf{Q}(\partial_z)\partial_{t}y_{p',p}(t,z,\epsilon)  =\epsilon^{(\boldsymbol{\delta}_{\mathbf{D}}-1)(k+1)-\boldsymbol{\delta}_{\mathbf{D}}+1}t^{(\boldsymbol{\delta}_{\mathbf{D}}-1)(k+1)}\partial_t^{\boldsymbol{\delta}_{\mathbf{D}}}\mathbf{R}_{\mathbf{D}}(\partial_z)y_{p',p}(t,z,\epsilon)\\
+\sum_{l=1}^{\mathbf{D}-1}\epsilon^{\boldsymbol{\Delta}_l}t^{\mathbf{d}_l}\partial_t^{\boldsymbol{\delta}_l}\mathbf{R}_l(\partial_z)y_{p',p}(t,z,\epsilon)+\mathbf{c}_{0,0}(\epsilon)\mathbf{c}_{0}(z,\epsilon)
\mathbf{R}_{0}(\partial_z)y_{p',p}(t,z,\epsilon)+\mathbf{c}_{\mathbf{F}}(\epsilon)u_{p}(t,z,\epsilon),
\label{SCP_bf_2}
\end{multline}
for initial data $y_{p',p}(0,z,\epsilon)\equiv0$, where the constants $\mathbf{D},\boldsymbol{\delta}_l,\mathbf{d}_l,\boldsymbol{\Delta}_l$, for $1\le l \le \mathbf{D}$, the polynomials $\mathbf{Q},\mathbf{R}_l$ for $0\le l\le \mathbf{D}$, the function $\mathbf{c}_{0,0}$ are those considered in Section~\ref{seccion8}, and $u_p(t,z,\epsilon)$ is the solution of (\ref{main_q_diff_diff_first}), with $u_p(0,z,\epsilon)\equiv 0$, constructed in Theorem~\ref{teo1}.

\begin{theo}\label{teo3}
Under the hypotheses described in this section, assume moreover that (\ref{deltaD_deltal_constraints_bf}) holds. There exist positive constants $r_{\mathbf{Q},\mathbf{R}_{\mathbf{D}}},\boldsymbol{\varsigma}_{0,0},\boldsymbol{\varsigma}_{0},\boldsymbol{\varsigma}_{F}>0$ such that if
\begin{equation}
\sup_{\epsilon \in D(0,\epsilon_{0})} \left| \frac{\mathbf{c}_{0,0}(\epsilon)}{\epsilon} \right|
\leq \boldsymbol{\varsigma}_{0,0} \ \ , \ \
||\mathbf{C}_{0,0}(m,\epsilon)||_{(\beta,\mu)} \leq \boldsymbol{\varsigma}_{0} \ \ , \ \
\sup_{\epsilon \in D(0,\epsilon_{0})} \left| \frac{\mathbf{c}_{\mathbf{F}}(\epsilon)}{\epsilon} \right|
\leq \boldsymbol{\varsigma}_{\mathbf{F}} \label{varsigma00_C00_cF_small_bf_2}
\end{equation}
for all $\epsilon \in D(0,\epsilon_{0})$, then, for every $0\le p\le \varsigma-1$ and $0\le p'\le \chi_p-1$, there exists a solution $y_{p',p}(t,z,\epsilon)$ of equation (\ref{SCP_bf_2}) with $y_{p',p}(0,z,\epsilon)\equiv0$, which defines a holomorphic and bounded function on the domain $\tilde{\mathcal{T}}\times H_{\beta'}\times\mathcal{E}_{p',p}$, for every $0<\beta'<\beta$. In addition to that, two different estimates hold:
1) Let $0\le p\le \varsigma-1$ and $0\le p',p''\le \chi_p$ with $p'\neq p''$ and such that $\mathcal{E}_{p',p}\cap\mathcal{E}_{p'',p}\neq\emptyset$. Then, there exist $\tilde{C},\tilde{C}_2>0$, which do not depend on $\epsilon$ such that
$$\sup_{t\in\tilde{T},z\in H_{\beta'}}|y_{p',p}(t,z,\epsilon)-y_{p'',p}(t,z,\epsilon)|\le \tilde{C}\exp\left(-\frac{\tilde{C}_{2}}{|\epsilon|^k}\right),$$
for all $0<\beta'<\beta$, and every $\epsilon\in\mathcal{E}_{p',p}\cap\mathcal{E}_{p'',p}$.

2) Let $0\le p,p'\le \varsigma-1$ with $p\neq p'$, and let $0\le p_1\le \chi_p$, $0\le p_2\le \chi_{p'}$ such that $\mathcal{E}_{p_1,p}\cap\mathcal{E}_{p_2,p'}\neq\emptyset$. Then, there exist $\tilde{C}_3>0$, $\tilde{C}_4\in\mathbb{R}$ which do not depend on $\epsilon$ such that for every $0<\kappa<\frac{k}{\delta}\min_{l=1}^{D-1}d_l$, one has

$$\sup_{t\in\tilde{T},z\in H_{\beta'}}|y_{p_1,p}(t,z,\epsilon)-y_{p_2,p'}(t,z,\epsilon)|\le \tilde{C}_3\exp\left(-\frac{\kappa}{2\log(q)}\log^2(|\epsilon|)\right)|\epsilon|^{\tilde{C}_4},$$

for all $0<\beta'<\beta$, and every $\epsilon\in\mathcal{E}_{p,p_1}\cap\mathcal{E}_{p',p_2}$.

\end{theo}

\textbf{Remark:} The existence of a unique formal solution and related asymptotics is left for the last main result in the last section for the sake of clarity.

\begin{proof}
Let $0\le p\le \varsigma-1$ and $0\le p'\le \chi_p-1$. One is allowed to apply Proposition~\ref{prop1982} to the problem (\ref{SCP_bf}) in order to get a solution $Y_{\mathfrak{e}_{p',p},\mathfrak{d}_{p}}(T,m,\epsilon)$ of such problem under initial data $Y_{\mathfrak{e}_{p',p},\mathfrak{d}_{p}}(0,m,\epsilon)\equiv0$. Such function is holomorphic and bounded with respect to $T$ in the sector $S_{\mathfrak{e}_{p',p},\theta, R^{1/k}}$, with $\frac{\pi}{k}<\theta<\frac{\pi}{k}+2\delta$, for some $\delta>0$, and some $R>0$, continuous with respect to $m$ in $\mathbb{R}$, and holomorphic with respect to $\epsilon$ in $D(0,\epsilon_0)$. Moreover, in view of (\ref{e2015}), it holds that 
$$|Y_{\mathfrak{e}_{p',p},\mathfrak{d}_{p}}(T,m,\epsilon)|\le k\int_{0}^{\infty}|v_{k}^{\mathfrak{e}_{p',p},\mathfrak{d}_p}(re^{i\gamma_{p',p}},m,\epsilon)|e^{-\left(\frac{r}{|T|}\right)^{k}\cos(k(\gamma_{p',p}-\arg(T)))}\frac{du}{u},$$
$$\le C_1\frac{1}{(1+|m|)^{\mu}}\exp(-\beta|m|)\int_{0}^{\infty}\frac{1}{1+r^{2k}}\exp((\nu-\delta_1/r_{\mathcal{T}})r^{k})dr\le C_2\exp(-\beta|m|)\frac{1}{(1+|m|)^{\mu}},$$
for some $C_1,C_2>0$, and all $(T,m,\epsilon)\in S_{\mathfrak{e}_{p',p},\theta, R^{1/k}}\times \mathbb{R}\times D(0,\epsilon_0)$. This and the choice of the set $\{(S_{\mathfrak{e}_{p',p}})_{\substack{0\le p\le \varsigma -1\\0\le p'\le \chi_p-1}},\tilde{\mathcal{T}}\}$, which is associated to the good covering $\{ \mathcal{E}_{p',p}\}_{\substack{0\le p\le \varsigma-1\\0\le p'\le \chi_p-1}}$, allow us to affirm that, for every $0\le p\le \varsigma -1$ and all $0\le p'\le \chi_p-1$, the function
$$y_{p',p}(t,z,\epsilon):=\mathcal{F}^{-1}\left(m\mapsto Y_{\mathfrak{e}_{p',p},\mathfrak{d}_p}(t\epsilon,m,\epsilon)\right)(z)$$
is holomorphic and bounded on the domain $\tilde{\mathcal{T}}\times H_{\beta'}\times\mathcal{E}_{p',p}$, for every $0<\beta'<\beta$.

We recall that, for every $0\le p\le \varsigma-1$, the function $u_p(t,z,\epsilon)$ is the solution of (\ref{main_q_diff_diff_first}), and it is obtained by applying Laplace and inverse Fourier transformation of $w_k^p(\tau,m,\epsilon)$, solution of (\ref{q_diff_conv_init_v_prob}). Regarding the forcing term, from the definition of the function $U_p(T,m,\epsilon)$ (see~(\ref{e1963})), being the Laplace transform of the function $w_k^p(\tau,m,\epsilon)$ along $\gamma_p$, and regarding the properties of inverse Fourier transform (see Proposition~\ref{prop8}) we conclude that $y_{p',p}(t,z,\epsilon)$ solves (\ref{SCP_bf_2}), with initial data $y_{p',p}(t,z,\epsilon)$. 

We now give proof to the second part of the enunciate, in which we determine upper bounds for the difference of consecutive solutions. We refer to consecutive solutions when dealing with solutions of the problem defined in consecutive sectors in the good covering determined with respect to the perturbation parameter $\epsilon$. At this point, we have to make a difference between consecutive solutions related to a common $p$, for $0\le p\le \varsigma-1$ (and differing on the value of $p'$, for $0\le p'\le \chi_p-1$) and consecutive solutions of the problem coming from different $p$.

Let $0\le p\le \varsigma-1$ and $0\le p',p''\le \chi_p-1$ with $p\neq p''$ such that $\mathcal{E}_{p',p}\cap\mathcal{E}_{p'',p}\neq\emptyset$. Let $0\le \beta'<\beta$. We aim to give upper bounds of the difference
\begin{equation}\label{e2075}\sup_{t\in\mathcal{\tilde{T}},z\in H_{\beta'}}|y_{p',p}(t,z,\epsilon)-y_{p'',p}(t,z,\epsilon)|.
\end{equation}

At this point, we think is important to take a look back to the geometric arrangement of the problem. For every $0\le p\le \varsigma-1$, the function $w_k^p(\tau,m,\epsilon)$ belongs to the vector space $\mathrm{Exp}_{(k,\beta,\mu,\alpha);\Theta_{\hat{Q}\mu_1}^p}^q$. In particular, $w_k^p(\tau,m,\epsilon)$ is a holomorphic function in $U_{\mathfrak{d}_{p}}$ (which contains both $S_{\mathfrak{e}_{p',p}}$ and $S_{\mathfrak{e}_{p'',p}}$, in which the directions of integration $\gamma_{p',p}$ and $\gamma_{p'',p}$ belong). 

Let $r_{p,p',p''}>0$ be such that the roots of the polynomial $\mathbf{P}_m(\tau)$ which are contained in $U_{\mathfrak{d}_p}$ have modulus larger than $2r_{p,p',p''}$, and consider an auxiliary finite sector $S_{p,p',p'',r_{p,p',p''}}$ of complex numbers satisfying that the set $\{\tau\in (S_{\mathfrak{e}_{p',p}}\cup S_{\mathfrak{e}_{p'',p}})\cap D(0,\frac{3}{2}r_{p,p',p''})\}\subseteq \bar{S_{p,p',p'',r_{p,p',p''}}}$. One can consider the problem (\ref{k_Borel_equation_analytic_bf}) and give rise to a solution of such problem in the Banach space $G_{\beta,\mu;\overline{S}_{p,p',p'',r_{p,p',p''}}}$ of continuous functions $(\tau,m)\mapsto h(\tau,m)$ on $\overline{S}_{p,p',p'',r_{p,p',p''}}\times \mathbb{R}$, which are holomorphic with respect to $\tau$ on $S_{p,p',p'',r_{p,p',p''}}$, and such that
$$\left\|h(\tau,m)\right\|_{\beta,\mu;S_{p,p',p'',r_{p,p',p''}}}=\sup_{\tau\in\overline{S}_{p,p',p'',r_{p,p',p''}},m\in\mathbb{R}}\frac{(1+|m|)^{\mu}\exp(\beta|m|)}{|\tau|}|h(\tau,m)|<\infty.$$

Observe that this Banach space is a restriction of that in Definition~\ref{defi3}, with a compact set $\overline{\Omega}$. The function $w_{k}^{p}(\tau,m,\epsilon)$ belongs to $G_{\beta,\mu;\overline{S}_{p,p',p'',r_{p,p',p''}}}$, and is holomorphic with respect to $\epsilon$ in $D(0,\epsilon_0)$. An analogous study of the problem (\ref{k_Borel_equation_analytic_bf}) by means of a fixed point argument as that stated in the first part of the proof in Proposition~\ref{prop9} within this novel Banach spaces allow us to affirm that the functions $v_{k}^{\mathfrak{e}_{p',p},\mathfrak{d}_{p}}$ and $v_{k}^{\mathfrak{e}_{p'',p},\mathfrak{d}_{p}}$ are analytically extended to $S_{p,p',p'',r_{p,p',p''}}$ (also with respect to $\epsilon$ in $D(0,\epsilon_0)$), and can be prolonged into $S_{\mathfrak{e}_{p',p}}$ and $S_{\mathfrak{e}_{p'',p}}$ respectively. We omit the details of the proof which is exact to the corresponding one of Proposition~\ref{prop9}. 

In this situation, by means of a deformation in the integration path (\ref{e2075}) is estimated from above in the following terms. See Figure~\ref{figure4} for an sketch of the deformation and the splitting of the integration path.

\begin{figure}[h]
	\centering
		\includegraphics[width=.5\textwidth]{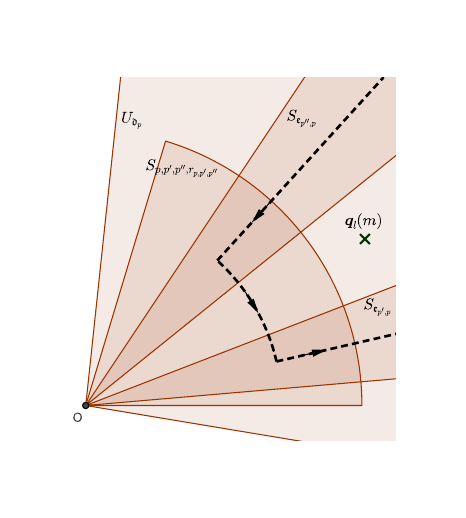}
	\caption{Proof of Theorem~\ref{teo3}. First case.}
	\label{figure4}
\end{figure}

$$\sup_{t\in\mathcal{\tilde{T}},z\in H_{\beta'}}|y_{p',p}(t,z,\epsilon)-y_{p'',p}(t,z,\epsilon)|\le I_1+I_2+I_3,$$
where
$$I_1=\left|\frac{k}{(2\pi)^{1/2}}\int_{-\infty}^{\infty}\int_{C_{\gamma_{p'',p},\gamma_{p',p},r_{p,p',p''}}}v_{k}^{\mathfrak{e}_{p'',p},\mathfrak{d}_{p}}(u,m,\epsilon)e^{-\left(\frac{u}{\epsilon t}\right)^k}e^{izm}\frac{du}{u}dm\right|,$$
with $C_{\gamma_{p'',p},\gamma_{p',p},r_{p,p',p''}}$ is an arc of circle with radius $r_{p,p',p''}$, connecting $r_{p,p',p''}e^{\sqrt{-1}\gamma_{p'',p}}$ with $r_{p,p',p''}e^{\sqrt{-1}\gamma_{p',p}}$ with an adequate orientation,
$$I_{2}=\left|\frac{k}{(2\pi)^{1/2}}\int_{-\infty}^{\infty}\int_{L_{\gamma_{p'',p},r_{p,p',p''},\infty}}v_{k}^{\mathfrak{e}_{p'',p},\mathfrak{d}_{p}}(u,m,\epsilon)e^{-\left(\frac{u}{\epsilon t}\right)^k}e^{izm}\frac{du}{u}dm\right|,$$
$$I_{3}=\left|\frac{k}{(2\pi)^{1/2}}\int_{-\infty}^{\infty}\int_{L_{\gamma_{p',p},r_{p,p',p''},\infty}}v_{k}^{\mathfrak{e}_{p',p},\mathfrak{d}_{p}}(u,m,\epsilon)e^{-\left(\frac{u}{\epsilon t}\right)^k}e^{izm}\frac{du}{u}dm\right|,$$
with $L_{\gamma_{a,p},r_{p,p',p''},\infty}=[r_{p,p',p''},\infty)e^{\sqrt{-1}\gamma_{a,p}}$, for $a=p',p''$. From the fact that $v_{k}^{\mathfrak{e}_{a,p},\mathfrak{d}_{p}}(u,m,\epsilon)$ belongs to the spaces $F_{(\nu,\beta,\mu,k);\overline{S_{\mathfrak{e}_{a,p}}}}$ and $G_{\beta,\mu;\overline{S}_{p,p',p'',r_{p,p',p''}}}$ for $a=p',p''$, one can follow analogous arguments as in (\ref{I_1_exp_small_order_k_on circle}) and (\ref{I_3_exp_small_order_k_on_halfline}). Indeed, one has
$$ I_1\le \tilde{C}_!\left\|v_{k}^{\mathfrak{e}_{p'',p},\mathfrak{d}_{p}}(u,m,\epsilon)\right\|_{\beta,\mu;S_{p,p',p'',r_{p,p',p''}}}\int_{\gamma_{p',p}}^{\gamma_{p'',p}}e^{-\left(\nu-\frac{\delta_1}{2\epsilon_0^k(r_{\mathcal{T}})^k}\right)r_{p,p',p''}^k}d\theta\exp\left(-\frac{\tilde{C}_2}{|\epsilon|^k}\right),$$
and
$$I_2\le \tilde{C}_1\left\|v_{k}^{\mathfrak{e}_{p'',p},\mathfrak{d}_{p}}(u,m,\epsilon)\right\|_{(\nu,\beta,\mu,k);\overline{S_{\mathfrak{e}_{p'',p}}}}\int_{r_{p,p',p''}}^{\infty}e^{-\left(\nu-\frac{\delta_1}{2\epsilon_0^k(r_{\mathcal{T}})^k}\right)|u|}d|u|\exp\left(-\frac{\tilde{C}_2}{|\epsilon|^k}\right),$$
with analogous estimates concerning $I_3$, for every $\epsilon\in\mathcal{E}_{p',p}\cap\mathcal{E}_{p'',p}$, $t\in\tilde{\mathcal{T}}$ and $z\in H_{\beta'}$. Here, $\tilde{C}_1=\frac{k}{(2\pi)^{1/2}}\int_{-\infty}^{\infty}e^{-(\beta-\beta')|m|}dm$, and $\tilde{C}_{2}=\frac{r_{p,p',p''}^k}{2(r_{\mathcal{T}})^k}$. The first part of the result follows from here.

We now assume that $0\le p,p'\le \varsigma-1$ with $p\neq p'$, and $0\le p_1\le \chi_p$, $0\le p_2\le \chi_{p'}$ such that $\mathcal{E}_{p_1,p}\cap\mathcal{E}_{p_2,p'}\neq\emptyset$. We proceed to estimate the difference $$\sup_{t\in\tilde{\mathcal{T}},z\in H_{\beta'}}|y_{p_1,p}(t,z,\epsilon)-y_{p_2,p'}(t,z,\epsilon)|,$$
for every $\epsilon\in \mathcal{E}_{p_1,p}\cap\mathcal{E}_{p_2,p'}$, and for any $0<\beta'<\beta$. We follow analogous arguments as in the proof of Theorem~\ref{teo1}, so we omit most of the details.

In view of Proposition~\ref{prop9}, one is able to write
$$y_{p_1,p}(t,z,\epsilon)=\sum_{j\ge0}y_{p_1,p,j}(t,z,\epsilon),\qquad y_{p_2,p'}(t,z,\epsilon)=\sum_{j\ge0}y_{p_2,p',j}(t,z,\epsilon),$$
where 
$$y_{p_1,p,j}=\frac{k}{(2\pi)^{1/2}}\int_{-\infty}^{\infty}\int_{L_{\gamma_{p_1,p}}}v_{k,j}^{\mathfrak{e}_{p_1,p},\mathfrak{d}_{p}}(u,m,\epsilon)e^{-\left(\frac{u}{\epsilon t}\right)^{k}}e^{izm}\frac{du}{u}dm,$$
and
$$y_{p_2,p',j}=\frac{k}{(2\pi)^{1/2}}\int_{-\infty}^{\infty}\int_{L_{\gamma_{p_2,p'}}}v_{k,j}^{\mathfrak{e}_{p_2,p'},\mathfrak{d}_{p'}}(u,m,\epsilon)e^{-\left(\frac{u}{\epsilon t}\right)^{k}}e^{izm}\frac{du}{u}dm,$$
where $L_{\gamma_{p_1,p}}$ and $L_{\gamma_{p_2,p'}}$ are defined above. We can deform the integration path in the shape of the proof of Theorem~\ref{teo1} (see (\ref{path_deform_decomp_difference_u_p_j}) and Figure~\ref{figure3}) to write
\begin{multline}
 y_{p_1,p,j}(t,z,\epsilon) - y_{p_2,p',j}(t,z,\epsilon) \\=
 \frac{k}{(2\pi)^{1/2}} \int_{-\infty}^{+\infty} \int_{C_{\gamma_{p_1,p},\gamma_{p_2,p'},\check{Q}\mu_{0,j}}}
 v_{k,j}^{\mathfrak{e}_{p_1,p},\mathfrak{d}_p}(u,m,\epsilon) \exp( -(\frac{u}{\epsilon t})^{k} ) e^{izm} \frac{du}{u} dm\\
 + \left(\sum_{h=0}^{j} \frac{k}{(2\pi)^{1/2}} \int_{-\infty}^{+\infty}
 \int_{L_{\gamma_{p_1,p},\check{Q}\mu_{0,h},\hat{Q}\mu_{1,h}}}
 v_{k,j}^{\mathfrak{e}_{p_1,p},\mathfrak{d}_p}(u,m,\epsilon) \exp( -(\frac{u}{\epsilon t})^{k} ) e^{izm} \frac{du}{u} dm \right)\\
 + \frac{k}{(2\pi)^{1/2}} \int_{-\infty}^{+\infty}
 \int_{L_{\gamma_{p_1,p},\hat{Q}\mu_{1},\infty}}
 v_{k,j}^{\mathfrak{e}_{p_1,p},\mathfrak{d}_p}(u,m,\epsilon) \exp( -(\frac{u}{\epsilon t})^{k} ) e^{izm} \frac{du}{u} dm\\
 - \left( \sum_{h=0}^{j} \frac{k}{(2\pi)^{1/2}} \int_{-\infty}^{+\infty}
 \int_{L_{\gamma_{p_2,p'},\check{Q}\mu_{0,h},\hat{Q}\mu_{1,h}}}
 v_{k,j}^{\mathfrak{e}_{p_2,p'},\mathfrak{d}_{p'}}(u,m,\epsilon) \exp( -(\frac{u}{\epsilon t})^{k} ) e^{izm} \frac{du}{u} dm \right)\\
 - \frac{k}{(2\pi)^{1/2}} \int_{-\infty}^{+\infty}
 \int_{L_{\gamma_{p_2,p'},\hat{Q}\mu_{1},\infty}}
 v_{k,j}^{\mathfrak{e}_{p_2,p'},\mathfrak{d}_p'}(u,m,\epsilon) \exp( -(\frac{u}{\epsilon t})^{k} ) e^{izm} \frac{du}{u} dm
 \label{path_deform_decomp_difference_u_p_j_2}
\end{multline}
The integral
$$\tilde{I}_1:=\left|\frac{k}{(2\pi)^{1/2}} \int_{-\infty}^{+\infty} \int_{C_{\gamma_{p_1,p},\gamma_{p_2,p'},\check{Q}\mu_{0,j}}}
 v_{k,j}^{\mathfrak{e}_{p_1,p},\mathfrak{d}_p}(u,m,\epsilon) \exp( -(\frac{u}{\epsilon t})^{k} ) e^{izm} \frac{du}{u} dm\right|$$
is estimated from above like in (\ref{I_1_exp_small_order_k_on circle}), and by means of (\ref{e1474}).
$$\tilde{I}_1\le \tilde{C}_{6} K_{6}^{j}(\frac{1}{q^{\delta}})^{\frac{k \min_{l=1}^{D-1}d_{l}}{2}(j+1)j}
\exp( - \frac{\delta_{1}(\check{Q}\mu_{0})^{k}}{(r_{\mathcal{T}})^{k}} (\frac{1}{q^{\delta k}})^{j}
\frac{1}{|\epsilon|^{k}} ).$$
The estimates for
$$\tilde{I}_2:= \left|\frac{k}{(2\pi)^{1/2}} \int_{-\infty}^{+\infty}
 \int_{L_{\gamma_{p_1,p},\check{Q}\mu_{0,h},\hat{Q}\mu_{1,h}}}
 v_{k,j}^{\mathfrak{e}_{p_1,p},\mathfrak{d}_p}(u,m,\epsilon) \exp( -(\frac{u}{\epsilon t})^{k} ) e^{izm} \frac{du}{u} dm\right|$$
are analogous to those in (\ref{I_2_exp_small_order_k_on_segment_h}), by means of (\ref{e1484}). At this point, we choose $r_{\mathcal{T}}>0$ with $\nu<\frac{\delta_1}{2\epsilon_0^kr_{\mathcal{T}^k}}$. Then,
\begin{align*}
\tilde{I}_2&\le \tilde{C}_{7.1}K_7^j (\frac{1}{q^{\delta}})^{\frac{k \min_{l=1}^{D-1}d_{l}}{2}(h+1)h} \int_{\check{Q}\mu_{0,h}}^{\hat{Q}\mu_{1,h}}\exp\left(\nu-\frac{\delta_1}{2(\epsilon_0r_{\mathcal{T})}^k}r^{k}\right)dr \exp\left( - \frac{\delta_{1}(\check{Q}\mu_{0})^{k}}{2(r_{\mathcal{T}})^{k}} (\frac{1}{q^{\delta k}})^{h}
\frac{1}{|\epsilon|^{k}} \right)\\
&\le \tilde{C}_{7} K_7^j (\frac{1}{q^{\delta}})^{\frac{k \min_{l=1}^{D-1}d_{l}}{2}(h+1)h}\exp\left( - \frac{\delta_{1}(\check{Q}\mu_{0})^{k}}{2(r_{\mathcal{T}})^{k}} (\frac{1}{q^{\delta k}})^{h}
\frac{1}{|\epsilon|^{k}} \right),
\end{align*}
for some $\tilde{C}_{7.1},\tilde{C}_7>0$. The term in which the integral along $L_{\gamma_{p_2,p'},\check{Q}\mu_{0,h},\hat{Q}\mu_{1,h}}$ can be treated in the same manner. Finally, in view of (\ref{e1465}) and following the same steps as in (\ref{I_2_exp_small_order_k_on_segment_h}), we conclude that
$$\tilde{I}_{3}=\left|\frac{k}{(2\pi)^{1/2}} \int_{-\infty}^{+\infty}
 \int_{L_{\gamma_{p_1,p},\hat{Q}\mu_{1},\infty}}
 v_{k,j}^{\mathfrak{e}_{p_1,p},\mathfrak{d}_p}(u,m,\epsilon) \exp( -(\frac{u}{\epsilon t})^{k} ) e^{izm} \frac{du}{u} dm\right|$$
is such that
\begin{align*}
I_3&\le \tilde{C}_{5.1}K_5^j\int_{\hat{Q}\mu_1}^{\infty}\exp\left(\nu- \frac{\delta_{1}}{2(\epsilon_0r_{\mathcal{T}})^{k}}r^k  \right) dr\exp\left(- \frac{\delta_{1}(\hat{Q}\mu_{1})^{k}}{2(r_{\mathcal{T}})^{k}} 
\frac{1}{|\epsilon|^{k}} \right)\\
&\le \tilde{C}_5K_5^j \exp\left(- \frac{\delta_{1}(\hat{Q}\mu_{1})^{k}}{2(r_{\mathcal{T}})^{k}} \frac{1}{|\epsilon|^{k}}\right), 
\end{align*}
for some $\tilde{C}_{5.1},\tilde{C}_5>0$. The above estimates on $\tilde{I}_{j}$, for $j=1,2,3$, and Lemma~\ref{lema3} allow us to conclude the result following the same arguments as in the last part of the proof in Theorem~\ref{teo1}.

\end{proof}

\section{Analytic solutions of a $q-$difference-differential initial value problem}\label{seccion11}

This section is devoted to the asymptotic behavior of the solutions of the problem (\ref{SCP_bf_2}), under null initial data. The main result in this direction is obtained by a two-level version of Ramis-Sibuya theorem involving both $q-$Gevrey and Gevrey estimates.

\subsection{Ramis-Sibuya theorem theorem in two levels (Gevrey and $q-$Gevrey)}\label{seccion111}

In Section~\ref{seccion6}, we mentioned a $q-$Gevrey version of Ramis-Sibuya theorem. In this section, we provide a novel two-level approach concerning both Gevrey and $q-$Gevrey asymptotics in sectors. A two-level version of Ramis-Sibuya theorem has already been studied when considering two different Gevrey asymptotics (see~\cite{lamaade16, lama0}), two different $q-$Gevrey asymptotics (see~\cite{lamaq}), two levels concerning more general asymptotics, associated with strongly regular sequences (see~\cite{lamasa0})   or in the case of a Gevrey level and the $1^+$ level (see~\cite{ma0}).

\begin{defin}
Let $(\mathbb{F},\left\|\cdot\right\|_{\mathbb{F}})$ be a complex Banach space and $\mathcal{E}$ be an open and bounded sector with vertex at 0. We also consider a positive real number $\alpha$.

We say that a function $f:\mathcal{E}\to\mathbb{F}$, holomorphic on $\mathcal{E}$, admits a formal power series $\hat{f}(\epsilon)=\sum_{k\ge0}a_{k}\epsilon^{k}\in\mathbb{F}[[\epsilon]]$ as its $1/\alpha-$Gevrey asymptotic expansion if, for any closed proper subsector $\mathcal{W}\subseteq\mathcal{E}$ with vertex at the origin, there exist $C,M>0$ such that
$$\left\|f(\epsilon)-\sum_{k=0}^{N-1}a_{k}\epsilon^k\right\|_{\mathbb{F}}\le CM^{N}N!^{1/\alpha}|\epsilon|^{N},$$
for every $N\ge 1$, and all $\epsilon\in\mathcal{W}$.
\end{defin}

\begin{theo}[Gevrey and $q$-Gevrey Ramis Sibuya Theorem]\label{teo4}
Let $(\mathbb{E},\left\|\cdot\right\|_{\mathbb{E}})$ be a complex Banach space, and let $\{ \mathcal{E}_{p',p}\}_{\substack{0\le p\le \varsigma-1\\0\le p'\le \chi_p-1}}$ be a good covering in $\mathbb{C}^{\star}$. We assume $G_{p',p}:\mathcal{E}_{p',p}\to\mathbb{E}$ is a holomorphic function for every $0\le p\le \varsigma-1$ and $0\le p'\le \chi_p-1$, and we define $\Delta_{(p_1,p),(p_2,p')}(\epsilon)=G_{p_2,p'}(\epsilon)-G_{p_1,p}(\epsilon)$ for every $\epsilon\in \mathcal{E}_{p_1,p}\cap\mathcal{E}_{p_2,p'}$, in the case that $\mathcal{E}_{p_1,p}\cap\mathcal{E}_{p_2,p'}\neq\emptyset$. Moreover, we assume:
\begin{enumerate}
\item For every $0\le p\le \varsigma-1$ and $0\le p'\le \chi_p$ the function $G_{p',p}(\epsilon)$ is bounded as $\epsilon\in\mathcal{E}_{p',p}$ tends to the origin.
\item Let $k\ge1$ be an integer. Let $0\le p\le \varsigma-1$ and $0\le p_1,p_2\le \chi_p-1$ with $p_1\neq p_2$ such that $\mathcal{E}_{p_1,p}\cap\mathcal{E}_{p_2,p}\neq\emptyset$. Then, there exist $K_1,M_1>0$ such that
$$\left\|\Delta_{(p_1,p),(p_2,p)}(\epsilon)\right\|_{\mathbb{E}}\le K_1\exp\left(-\frac{M_1}{|\epsilon|^{k}}\right),$$
for every $\epsilon\in\mathcal{E}_{p_1,p}\cap\mathcal{E}_{p_2,p}$.
\item Let $\kappa>0$. Let $0\le p,p'\le \varsigma-1$ with $p\neq p'$, $0\le p_1\le \chi_p-1$ and $0\le p_2\le\chi_{p'}-1$ such that $\mathcal{E}_{p_1,p}\cap\mathcal{E}_{p_2,p'}\neq\emptyset$. Then, there exist $K_2>0$ and $M_2\in\mathbb{R}$ such that
$$ \left\|\Delta_{(p_1,p),(p_2,p')}(\epsilon)\right\|_{\mathbb{E}}\le K_2|○\epsilon|^{M_2}\exp\left(-\frac{\kappa}{2\log(q)}\log^2|\epsilon|\right),$$
for every $\epsilon\in \mathcal{E}_{p_1,p}\cap\mathcal{E}_{p_2,p'}$.
\end{enumerate}
Then, there exists a convergent power series $a(\epsilon)\in\mathbb{E}\{\epsilon\}$ defined on some neighborhood of the origin and $\hat{G}_1(\epsilon),\hat{G}_2(\epsilon)\in\mathbb{E}[[\epsilon]]$ such that for every $0\le p\le \varsigma-1$ and all $0\le p_1\le\chi_p-1$, the function $G_{p_1,p}$ can be written in the form
$$G_{p_1,p}(\epsilon)=a(\epsilon)+G_{p_1,p}^1(\epsilon)+G_{p_1,p}^2(\epsilon),$$
where $G^{1}_{p_1,p}(\epsilon)$ is holomorphic on $\mathcal{E}_{p_1,p}$ and admits $\hat{G}_{1}(\epsilon)$ as its $1/k-$Gevrey asymptotic expansion in $\mathcal{E}_{p_1,p}$; $G^{2}_{p_1,p}(\epsilon)$ is holomorphic on $\mathcal{E}_{p_1,p}$ and admits $\hat{G}_{2}(\epsilon)$ as its $q-$Gevrey asymptotic expansion of order $1/\kappa$ on $\mathcal{E}_{p_1,p}$. In both cases, the formal power series are common for every $0\le p\le \varsigma-1$ and all $0\le p_1\le\chi_p-1$. 
\end{theo}
\begin{proof}
For every $0\le p,p'\le \varsigma-1$, $0\le p_1\le\chi_p$ and $0\le p_2\le\chi_{p'}$ we define the holomorphic cocycles $\Delta^1_{(p_1,p),(p_2,p')}(\epsilon)=\Delta_{(p_1,p),(p_2,p')}$ if $p=p'$ (and $p_1\neq p_2$) and $\mathcal{E}_{p_1,p}\cap\mathcal{E}_{p_2,p'}\neq\emptyset$ and $\Delta^1_{(p_1,p),(p_2,p')}(\epsilon)\equiv 0$ otherwise. A direct consequence of Lemma XI-2-6 in~\cite{hssi} states that for every $0\le p,p'\le \varsigma-1$, $0\le p_1\le\chi_p-1$ and $0\le p_2\le\chi_{p'}-1$ there exist holomorphic functions $\Psi^1_{p_1,p},\Psi^1_{p_2,p'}:\mathcal{E}_{p_1,p}\to\mathbb{C}$ such that 
$$\Delta^1_{(p_1,p),(p_2,p')}=\Psi^1_{p_2,p'}(\epsilon)-\Psi^1_{p_1,p}(\epsilon),$$
for every $\epsilon\in\mathcal{E}_{p_1,p}\cap\mathcal{E}_{p_2,p'}$ in the case that $\mathcal{E}_{p_1,p}\cap\mathcal{E}_{p_2,p'}\neq\emptyset$. Moreover, there exists a formal power series $\sum_{m\ge0}\phi_{m,1}\epsilon^m\in\mathbb{E}[[\epsilon]]$ such that for every $0\le p\le \varsigma-1$, $0\le p'\le\chi_p-1$ and every closed proper subsector $\mathcal{W}\subseteq\mathcal{E}_{p',p}$ with vertex at the origin, there exist $\breve{K}_1,\breve{M}_1>0$ with
$$\left\|\Psi^1_{p',p}(\epsilon)-\sum_{m=0}^{M-1}\phi_{m,1}\epsilon^m\right\|\le \breve{K}_1 (\breve{M}_{1})^{M} M!^{1/k}|\epsilon|^{M},$$
for every $\epsilon\in\mathcal{W}$, and all positive integer $M$.

We now turn our attention to the $q-$Gevrey settings and define, for every $0\le p,p'\le \varsigma-1$, $0\le p_1\le\chi_p$ and $0\le p_2\le\chi_{p'}$ the holomorphic cocycles $\Delta^2_{(p_1,p),(p_2,p')}(\epsilon)=\Delta_{(p_1,p),(p_2,p')}$ if $p\neq p'$ and $\mathcal{E}_{p_1,p}\cap\mathcal{E}_{p_2,p'}\neq\emptyset$, and $\Delta^2_{(p_1,p),(p_2,p')}(\epsilon)\equiv 0$ otherwise. A direct application of Theorem (q-RS) in Section~\ref{seccion6}, which can be found in detail in Lemma 8 in~\cite{ma} one can also affirm there exist bounded holomorphic functions $\Psi^2_{p_1,p}:\mathcal{E}_{p_1,p}\to\mathbb{C}$ such that
$$\Delta^2_{(p_1,p),(p_2,p')}(\epsilon)=\Psi^2_{p_2,p'}(\epsilon)-\Psi^2_{p_1,p}(\epsilon),$$
for every $\epsilon\in\mathcal{E}_{p_1,p}\cap\mathcal{E}_{p_2,p'}$ in the case that $\mathcal{E}_{p_1,p}\cap\mathcal{E}_{p_2,p'}\neq\emptyset$. Moreover, there exists a formal power series $\sum_{m\ge0}\phi_{m,2}\epsilon^m\in\mathbb{E}[[\epsilon]]$ such that for every $0\le p\le \varsigma-1$, $0\le p'\le\chi_p-1$ and every closed proper subsector $\mathcal{W}\subseteq\mathcal{E}_{p',p}$ with vertex at the origin, there exist $\breve{K}_2,\breve{M}_2>0$ with
$$\left\|\Psi^2_{p',p}(\epsilon)-\sum_{m=0}^{M}\phi_{m,2}\epsilon^m\right\|\le \breve{K}_2 (\breve{M}_{2})^{M+1} q^{\frac{(M+1)M}{2\kappa}}|\epsilon|^{M+1},$$
for every $\epsilon\in\mathcal{W}$, and all integer $M\ge0$.

We consider the functions 
$$a_{p',p}(\epsilon)=G_{p',p}(\epsilon)-\Psi^1_{p',p}(\epsilon)-\Psi^2_{p',p}(\epsilon).$$
Let $0\le p,p'\le \varsigma-1$ and $0\le p_1\le \chi_p-1$, $0\le p_2\le \chi_{p'}-1$ such that $\mathcal{E}_{p_1,p}\cap\mathcal{E}_{p_2,p'}\neq\emptyset$. We have
\begin{align*}
a_{p_1,p}(\epsilon)-a_{p_2,p'}(\epsilon)&=G_{p_1,p}(\epsilon)-G_{p_2,p'}(\epsilon)-\Psi^1_{p_1,p}(\epsilon)-\Psi^2_{p_1,p}(\epsilon)+\Psi^1_{p_2,p'}(\epsilon)+\Psi^2_{p_2,p'}(\epsilon)\\
&=G_{p_1,p}(\epsilon)-G_{p_2,p'}(\epsilon)-\Delta^1_{(p_1,p),(p_2,p')}(\epsilon)-\Delta^2_{(p_1,p),(p_2,p')}(\epsilon)\\
&=G_{p_1,p}(\epsilon)-G_{p_2,p'}(\epsilon)-\Delta_{(p_1,p),(p_2,p')}(\epsilon)=0,
\end{align*}
for every $\epsilon\in\mathcal{E}_{p_1,p}\cap\mathcal{E}_{p_2,p'}$. Then, $a_{p',p}$ is the restriction on $\mathcal{E}_{p_1,p}$ of a holomorphic function $a(\epsilon)$, defined on a neighborhood of the origin but zero, and bounded at the origin. This yields $a(\epsilon)$ turns out to be a holomorphic function defined on some neighborhood of the origin. We write 
$$G_{p',p}(\epsilon)=a(\epsilon)+\Psi^1_{p',p}(\epsilon)+\Psi^2_{p',p}(\epsilon),$$
for every $0\le p\le \varsigma-1$ and $0\le p'\le \chi_p-1$. The result follows from here.
\end{proof}

\subsection{Solution and multiscale Gevrey and $q-$Gevrey asymptotics of the main problem}\label{seccion112}

In this last section, we state the main result of the present work, giving rise to the existence of an analytic and a formal solution of the main problem. We also provide an asymptotic behavior relating both in two different levels: $q-$Gevrey and Gevrey, in which the formal and the analytic solutions can be splitted. 

We now enunciate the main problem under study. We have decided to include every detail in the hypotheses made for the sake of completeness, although they can be detected from the preceding sections. 

Let $k \geq 1$ and $D \geq 2$ be integers. Let $\delta,k_{1}>0$, $q>1$ and $\alpha$ be real numbers and for $1 \leq l \leq D$, let $d_{l},\Delta_{l} \geq 0$ be nonnegative integers. We make the assumption that
\begin{equation}
kd_{D} - 1 \geq k_{1}\delta + kd_{l} \ \ , \ \ \Delta_{l} \geq kd_{l} \label{assum_dD_dl_k1_Delta_l_f}
\end{equation}
for all $1 \leq l \leq D-1$. Let $Q(X),R_{l}(X) \in \mathbb{C}[X]$, $1 \leq l \leq D$, be polynomials such that
\begin{equation}
\mathrm{deg}(Q) \geq \mathrm{deg}(R_D) \geq \mathrm{deg}(R_{l}) \ \ , \ \ Q(im) \neq 0 \ \ , \ \
R_{D}(im) \neq 0, \label{assum_deg_Q_RD_Rl_f}
\end{equation}
for all $m \in \mathbb{R}$, all $1 \leq l \leq D-1$. We make the additional assumption that there exists a bounded sectorial
annulus
$$ A_{Q,R_{D}} = \{ z \in \mathbb{C} / r_{Q,R_{D}} \leq |z| \leq r_{Q,R_{D}}^{1} \ \ , \ \
|\mathrm{arg}(z) - d_{Q,R_{D}}| \leq \eta_{Q,R_{D}} \} $$
for some $d_{Q,R_{D}} \in \mathbb{R}$, $\eta_{Q,R_{D}}>0$ and $0<r_{Q,R_{D}}<r_{Q,R_{D}}^{1}$ such that
\begin{equation}
\frac{Q(im)}{R_{D}(im)} \in A_{Q,R_{D}} \label{quotient_Q_RD_in_Ann_f}
\end{equation} 
for all $m \in \mathbb{R}$. Let $P_{m}(\tau) = Q(im) - R_{D}(im)(k\tau^{k})^{d_{D}}\in\mathbb{C}[\tau,m]$ with roots given by $\{q_{l}(m):0\le l\le kd_{D}-1\}$. Choose a set of real numbers $\{\mathfrak{d}_{p}\}_{0\le p\le \varsigma-1}$, for some integer $\varsigma\ge2$; we make this choice according to the next hypotheses:
\noindent 1) There exists $M_{1}>0$ such that $|\tau - q_{l}(m)| \geq M_{1}(1 + |\tau|)$ for all $0 \leq l \leq kd_{D}-1$, all $m \in \mathbb{R}$, all $\tau \in \bar{U}_{\mathfrak{d}_{p}} \cup
\bar{D}(0,\mu_{0}) \cup \bar{A}_{\mu_1}$, where $U_{\mathfrak{d}_p}$ is an unbounded sector with bisecting direction $\mathfrak{d}_{p}\in\mathbb{R}$, $0<\mu_0<\mu_1$ and $\bar{A}_{\mu_1}$stands for the annulus $\{\tau\in\mathbb{C}: |\tau|\ge \mu_1\}$.  

\noindent 2) There exists $M_{2}>0$ such that $|\tau - q_{l_0}(m)| \geq M_{2}|q_{l_0}(m)| $ for some $l_{0} \in \{0,\ldots,kd_{D}-1 \}$, all $m \in \mathbb{R}$,
$\tau \in \bar{U}_{\mathfrak{d}_p} \cup \bar{D}(0,\mu_{0}) \cup \bar{A}_{\mu_1}$. 

\noindent 3) The sector 
$$\bar{U}_{\mathfrak{d}_{p},\mathfrak{d}_{p+1}} =
\{ \tau \in \mathbb{C}^{\ast} / \mathrm{arg}(\tau) \in [\mathfrak{d}_{p},\mathfrak{d}_{p+1}] \}$$
contains at least one root $q_{l}(m)$, for some $0 \leq l \leq kd_{D}-1$, all $m \in \mathbb{R}$. 

We assume there exists $\hat{Q}>1$ and $0<\check{Q}<1$ such that $\hat{Q}\mu_1=q^{\delta}\check{Q}\mu_0$.

Let $\{\mathcal{E}_{p}\}_{0\le p\le \varsigma-1}$ be a good covering in $\mathbb{C}^{\star}$ such that the family $\{(U_{\mathfrak{d}_p,\theta,\epsilon_0 r_{\mathcal{T}}})_{0\le p\le \varsigma-1},\mathcal{T}\}$ is associated to the good covering $\{\mathcal{E}_{p}\}_{0\le p\le \varsigma-1}$ (see Definition~\ref{defi7}).

For every $0\le p\le \varsigma-1$ we first consider the linear initial value problem
\begin{equation}\label{e1p}
P_1(t,z,\epsilon,\sigma_{q^{\delta}},\partial_t,\partial_z)u_p(t,z,\epsilon)=f(t,z,\epsilon),
\end{equation}
where 
\begin{equation}\label{e1_1}P_1(t,z,\epsilon,\sigma_{q^{\delta}},\partial_t,\partial_z):=Q(\partial_{z})-R_{D}(\partial_{z})\epsilon^{kd_{D}}(t^{k+1}\partial_{t})^{d_{D}}- \sum_{l=1}^{D-1} \epsilon^{\Delta_{l}}(t^{k+1}\partial_{t})^{d_{l}}c_{l}(z,\epsilon)R_{l}(\partial_{z})\sigma_{q^{\delta}},
\end{equation}
with initial data $u_p(0,z,\epsilon) \equiv 0$. Here, $\sigma_{q^{\delta}}$ stands for the dilation operator on $t$ variable acting on a function $H(t)$ by $\sigma_{q^{\delta}}(H)(t)=H(q^{\delta}t)$, and for every $0\le l\le D-1$, $c_{l}(z,\epsilon):=\mathcal{F}^{-1}(m\mapsto C_{l}(m,\epsilon))(z)$, for some map $m \mapsto C_{l}(m,\epsilon)$ belonging to $E_{(\beta,\mu)}$ for some $\beta>0$ and $\mu > \mathrm{deg}(R_{D}) + 1$ and depends holomorphically on $\epsilon$ on $D(0,\epsilon_{0})$, for some $\epsilon_0>0$. $c_{l}(z,\epsilon)$ is a holomorphic function on $H_{\beta'}\times D(0,\epsilon_0)$, for every $0<\beta'<\beta$, where $H_{\beta'} = \{ z \in \mathbb{C} / |\mathrm{Im}(z)|< \beta' \}$. The forcing term $f(t,z,\epsilon)$ is given by
$$f(t,z,\epsilon):=\mathcal{F}^{-1}\left(m\mapsto F(t,m,\epsilon)\right)(z),$$
where $F(t,m,\epsilon):=\mathcal{L}_{m_k}^{\gamma}\left(\tau\mapsto \Psi_{k}(\tau,m,\epsilon)\right)(t\epsilon),$ for some $\gamma\in\mathbb{R}$, and $\psi_{k}(\tau,m,\epsilon)\in\mathrm{Exp}_{(k_{1},\beta,\mu,\alpha);\mathbb{C}}^{q}$ (see Definition~\ref{defi2}) for all $\epsilon \in D(0,\epsilon_{0})$, holomorphic with respect to $\epsilon$ in $D(0,\epsilon_0)$. $f(t,z,\epsilon)$ turns out to be a holomorphic function with respect to $(t,\epsilon)$ in a product of discs at the origin, and in $H_{\beta'}$ with respect to $z$, for every $0<\beta'<\beta$. In the case that $\sup_{\epsilon\in D(0,\epsilon_0)}\left\|C_l(m,\epsilon)\right\|_{\beta,\mu}\le \gamma_{l}$ for some small enough $\gamma_l$ for every $0\le l\le D-1$ and $r_{Q,R_D}>0$ is large enough, then for all $0\le p\le \varsigma-1$, the problem (\ref{e1_1}) admits a unique solution $u_p(t,z,\epsilon)\in\mathcal{O}(\mathcal{T}\times H_{\beta'}\times \mathcal{E}_p)$, for every $0<\beta'<\beta$ and small enough $r_{\mathcal{T}}>0$, with $u_p(0,z,\epsilon) \equiv 0$ (see Theorem~\ref{teo1}).

Let $0<\beta'<\beta$. We denote $\mathbb{F}$ the Banach space of bounded holomorphic functions on $\mathcal{T}\times H_{\beta'}$. There exists $\hat{u}(t,z,\epsilon)\in\mathbb{F}[[\epsilon]]$ which solves (\ref{e1_1}) and is the $q-$Gevrey asymptotic expansion of order $1/\kappa$ of the solution $u_p(t,z,\epsilon)$, seen as a holomorphic function from $\mathcal{E}_{p}$ into $\mathbb{F}$, for all $0\le p\le \varsigma-1$, and any fixed $0<\kappa<\frac{k}{\delta}\min_{l=1}^{D-1}d_l$ (see Theorem~\ref{teo2}).

Let $\mathbf{D} \geq 2$ be an integer. For $1 \leq l \leq \mathbf{D}$, let
$\mathbf{d}_{l}$,$\boldsymbol{\delta}_{l}$,$\mathbf{\Delta}_{l} \geq 0$ be nonnegative integers. We assume that $1 = \boldsymbol{\delta}_{1}$, $\boldsymbol{\delta}_{l} < \boldsymbol{\delta}_{l+1}$ and $\boldsymbol{\delta}_{\mathbf{D}}\ge\boldsymbol{\delta}_{l}+\frac{1}{k}$ for all $1 \leq l \leq \mathbf{D}-1$. We also make the assumption that
\begin{equation}
\mathbf{d}_{\mathbf{D}} = (\boldsymbol{\delta}_{\mathbf{D}}-1)(k+1) \ \ , \ \
\mathbf{d}_{l} > (\boldsymbol{\delta}_{l}-1)(k+1) \ \ , \ \
\mathbf{\Delta}_{l} - \mathbf{d}_{l} + \boldsymbol{\delta}_{l} - 1 \geq 0 \ \ , \ \
\mathbf{\Delta}_{\mathbf{D}} = \mathbf{d}_{\mathbf{D}} - \boldsymbol{\delta}_{\mathbf{D}} + 1
\label{assum_dl_delta_l_Delta_l_bf_Borel_f}
\end{equation}
for all $1 \leq l \leq \mathbf{D}-1$. Let $\mathbf{Q}(X),\mathbf{R}_{l}(X) \in \mathbb{C}[X]$, $0 \leq l \leq \mathbf{D}$, be polynomials such that
\begin{equation}
\mathrm{deg}(\mathbf{Q}) \geq \mathrm{deg}(\mathbf{R}_{\mathbf{D}}) \geq \mathrm{deg}(\mathbf{R}_{l}) \ \ , \ \
\mathbf{Q}(im) \neq 0 \ \ , \ \ \mathbf{R}_{\mathbf{D}}(im) \neq 0 \label{assum_Q_Rl_bf_Borel_f}
\end{equation}
for all $m \in \mathbb{R}$, all $0 \leq l \leq \mathbf{D}-1$. We consider $\mu>\deg(\mathbf{R}_{\mathbf{D}})+1$ and assume there exists
$$ S_{\mathbf{Q},\mathbf{R}_{\mathbf{D}}} = \{ z \in \mathbb{C} /
|z| \geq r_{\mathbf{Q},\mathbf{R}_{\mathbf{D}}} \ \ , \ \ |\mathrm{arg}(z) - d_{\mathbf{Q},\mathbf{R}_{\mathbf{D}}}|
\leq \eta_{\mathbf{Q},\mathbf{R}_{\mathbf{D}}} \} $$
for some $d_{\mathbf{Q},\mathbf{R}_{\mathbf{D}}} \in \mathbb{R}$, $\eta_{\mathbf{Q},\mathbf{R}_{\mathbf{D}}},r_{\mathbf{Q},\mathbf{R}_{\mathbf{D}}}>0$ such that
\begin{equation}
\frac{\mathbf{Q}(im)}{\mathbf{R}_{\mathbf{D}}(im)} \in S_{\mathbf{Q},\mathbf{R}_{\mathbf{D}}} \label{quotient_Q_RD_in_S_bf_f}
\end{equation} 
for all $m \in \mathbb{R}$. Let $\{\mathbf{q}_{l}(m):0\le l\le (\boldsymbol{\delta}_{\mathbf{D}}-1)k-1\}$ be the roots of the polynomial $\mathbf{P}_m(\tau)=\mathbf{Q}(im)k-\mathbf{R}_{\mathbf{D}}(im)k^{\boldsymbol{\delta}_{\mathbf{D}}}\tau^{(\boldsymbol{\delta}_{\mathbf{D}}-1)k}$. 

We fix $0\le p\le \varsigma-1$ and consider a finite set of real numbers $\mathfrak{e}_{p',p}$, $0 \leq p' \leq \chi_p-1$ related to an unbounded sector $S_{\mathfrak{e}_{p',p}}$ with vertex at 0 and bisecting direction $\mathfrak{e}_{p',p}$, and $\rho>0$ in such a way that:

\noindent 1) There exists $\mathbf{M}_{1}>0$ such that $|\tau - \mathbf{q}_{l}(m)| \geq \mathbf{M}_{1}(1 + |\tau|)$ for all $0 \leq l \leq (\boldsymbol{\delta}_{\mathbf{D}}-1)k-1$, $m \in \mathbb{R}$, $\tau \in S_{\mathfrak{e}_{p',p}} \cup \bar{D}(0,\rho)$, for all $0 \leq p' \leq \chi_p -1$.\medskip

\noindent 2) There exists $\mathbf{M}_{2}>0$ such that $|\tau - \mathbf{q}_{l_0}(m)| \geq \mathbf{M}_{2}|\mathbf{q}_{l_0}(m)|$ for some $l_{0} \in \{0,\ldots,(\boldsymbol{\delta}_{\mathbf{D}}-1)k-1 \}$, all $m \in \mathbb{R}$, all
$\tau \in S_{\mathfrak{e}_{p',p}} \cup \bar{D}(0,\rho)$, and for all $0 \leq p' \leq \chi_p -1$.\medskip

\noindent 3) For all $0 \leq p' \leq \chi_p-1$, one has $S_{\mathfrak{e}_{p',p}} \subset U_{\mathfrak{d}_{p}}$.

Let $\{ \mathcal{E}_{p',p}\}_{\substack{0\le p\le \varsigma-1\\0\le p'\le \chi_p-1}}$ be a good covering in $\mathbb{C}^{\star}$ such that $\mathcal{E}_{p',p}\subseteq \mathcal{E}_{p}$ for every $0\le p\le \varsigma-1$ and all $0\le p'\le \chi_p-1$. We consider a sector $\tilde{\mathcal{T}}\subseteq\mathcal{T}$ and
$$S_{\mathfrak{e}_{p',p},\theta,\epsilon_0r_{\mathcal{T}}}=\{T\in\mathbb{C}^{\star}:|T|\le \epsilon_0r_{\mathcal{T}},|\mathfrak{e}_{p',p}-\arg(T)|<\frac{\theta}{2}\},$$
where $\frac{\pi}{k}<\theta<\frac{\pi}{k}+Ap(S_{\mathfrak{e}_{p',p}})$, with $Ap(S_{\mathfrak{e}_{p',p}})$ such that $\epsilon t\in S_{\mathfrak{e}_{p',p},\theta,\epsilon_0r_\mathcal{T}}$, for every $t\in\mathcal{\tilde{T}}$ and $\epsilon\in\mathcal{E}_{p',p}$, by reducing the aperture of $\tilde{\mathcal{T}}$, if necessary, so that the set $\{(S_{\mathfrak{e}_{p',p},\theta,\epsilon_0r_\mathcal{T}})_{\substack{0\le p\le \varsigma -1\\0\le p'\le \chi_p-1}},\tilde{\mathcal{T}}\}$ is associated to the good covering $\{ \mathcal{E}_{p',p}\}_{\substack{0\le p\le \varsigma-1\\0\le p'\le \chi_p-1}}$. 

Let 
$$\mathbf{c}_0(z,\epsilon):=\mathcal{F}^{-1}\left(m\mapsto\mathbf{C}_{0,0}(m,\epsilon)\right)(z),$$
where $\mathbf{C}_{0,0}(m,\epsilon)\in E_{\beta,\mu)}$, and depends holomorphically on $\epsilon\in D(0,\epsilon_0)$. Then, $\mathbf{c}_{0}(z,\epsilon)$ turns out to be a  holomorphic function on $H_{\beta'}\times D(0,\epsilon_0)$, for any $0<\beta'<\beta$. We also consider bounded and holomorphic functions $\mathbf{c}_{0,0}(\epsilon),\mathbf{c}_{\mathbf{F}}(\epsilon)$ defined on $D(0,\epsilon_0)$, which vanish at $\epsilon=0$.

For large enough $r_{\mathbf{Q},\mathbf{R}_{\mathbf{D}}}>0$ and small enough $\boldsymbol{\varsigma}_{0,0},\boldsymbol{\varsigma}_{0},\boldsymbol{\varsigma}_{F}>0$ such that
\begin{equation}
\sup_{\epsilon \in D(0,\epsilon_{0})} \left| \frac{\mathbf{c}_{0,0}(\epsilon)}{\epsilon} \right|
\leq \boldsymbol{\varsigma}_{0,0} \ \ , \ \
||\mathbf{C}_{0,0}(m,\epsilon)||_{(\beta,\mu)} \leq \boldsymbol{\varsigma}_{0} \ \ , \ \
\sup_{\epsilon \in D(0,\epsilon_{0})} \left| \frac{\mathbf{c}_{\mathbf{F}}(\epsilon)}{\epsilon} \right|
\leq \boldsymbol{\varsigma}_{\mathbf{F}} \label{varsigma00_C00_cF_small_bf_2_f}
\end{equation}
for all $\epsilon \in D(0,\epsilon_{0})$, then, for every $0\le p\le \varsigma-1$, and $0\le p'\le \chi_p-1$, there exists a solution $y_{p',p}(t,z,\epsilon)$ of 
\begin{equation}\label{e2_1}P_2\left(t,z,\epsilon,\partial_t,\partial_z\right)y_{p',p}(t,z,\epsilon)=\mathbf{c}_{\mathbf{F}}(\epsilon)u_{p}(t,z,\epsilon),
\end{equation}
with $y_{p',p}(0,z,\epsilon)\equiv0$, where
$$P_2\left(t,z,\epsilon,\partial_t,\partial_z\right):=\mathbf{Q}(\partial_z)\partial_{t}-\sum_{l=1}^{\mathbf{D}}\epsilon^{\boldsymbol{\Delta}_l}t^{\mathbf{d}_l}\partial_t^{\boldsymbol{\delta}_l}\mathbf{R}_l(\partial_z)-\mathbf{c}_{0,0}(\epsilon)\mathbf{c}_{0}(z,\epsilon)
\mathbf{R}_{0}(\partial_z).$$
In view of Theorem~\ref{teo3}, we get the existence of a unique solution $y_{p',p}(t,z,\epsilon)$ of the previous problem which belongs to $\mathcal{O}(\tilde{\mathcal{T}}\times H_{\beta'}\times \mathcal{E}_{p',p})$, for every $0<\beta'<\beta$.

\begin{theo}
There exists a unique formal power series $\hat{y}(t,z,\epsilon)\in\mathbb{F}[[\epsilon]]$, formal solution of the problem
\begin{equation}\label{e2_1f}
P_2\left(t,z,\epsilon,\partial_t,\partial_z\right)\hat{y}(t,z,\epsilon)=\mathbf{c}_{\mathbf{F}}(\epsilon)\hat{u}(t,z,\epsilon).
\end{equation}
Moreover, one can write $\hat{y}(t,z,\epsilon)=a(t,z,\epsilon)+\hat{y}_{1}(t,z,\epsilon)+\hat{y}_{2}(t,z,\epsilon),$ in such a way that for every $0\le p\le \varsigma-1$ and $0\le p'\le \chi_p-1$ one has that 
$$y_{p',p}(t,z,\epsilon)=a(t,z,\epsilon)+y^{1}_{p',p}(t,z,\epsilon)+y^{2}_{p',p}(t,z,\epsilon),$$
where $\epsilon\mapsto y^{1}_{p',p}(t,z,\epsilon)$ is a $\mathbb{F}$-valued function that admits $\hat{y}_{1}(t,z,\epsilon)$ as its Gevrey asymptotic expansion of order $1/k$ in $\mathcal{E}_{p',p}$, and $\epsilon\mapsto y^{2}_{p',p}(t,z,\epsilon)$ is a $\mathbb{F}$-valued function that admits $\hat{y}_{2}(t,z,\epsilon)$ as its $q-$Gevrey asymptotic expansion of order $1/\kappa$, for every fixed $0<\kappa<\frac{k}{\delta}\min_{l=1}^{D-1}d_{l}$ (see Theorem~\ref{teo2}) in $\mathcal{E}_{p',p}$.
\end{theo}
\begin{proof}
For every $0\le p\le \varsigma-1$ and all $0\le p'\le \chi_p-1$ we define $G_{p',p}(\epsilon):=(t,z)\mapsto y_{p',p}(t,z,\epsilon)$, which is a holomorphic and bounded function from $\mathcal{E}_{p',p}$ into $\mathbb{F}$. In view of Theorem~\ref{teo3}, and Theorem~\ref{teo4} one yields the existence of $\hat{y}(t,z,\epsilon)$ which can be splitted in the form of the enunciate and the expected asymptotic properties. It only rests to check that $\hat{y}(t,z,\epsilon)$ is a formal solution of (\ref{e2_1f}).

We write $\hat{y}(t,z,\epsilon):=\sum_{m\ge0}\frac{H_m(t,z)\epsilon^m}{m!}$. From the properties of the asymptotic expansions Gevrey and $q-$Gevrey, and Taylor expansion at 0, we have
\begin{equation}\label{e2354}
\lim_{\epsilon\to0,\epsilon\in\mathcal{E}_{p',p}}\sup_{t\in\tilde{\mathcal{T}},z\in H_{\beta'}}\left|\partial_{\epsilon}^{m}y_{p',p}(t,z,\epsilon)-H_{m}(t,z)\right|=0,
\end{equation}
for every $0\le p\le \varsigma-1$, $0\le p'\le \chi_p-1$and $m\ge0$. The funtion $y_{p',p}(t,z,\epsilon)$ solves (\ref{e2_1}). We take derivatives of order $m$ with respect to $\epsilon$ at both sides of the equation (\ref{e2_1}), for some integer $m\ge0$ and get
\begin{align*}
\mathbf{Q}(\partial_z)\partial_{t}\partial_\epsilon^my_{p',p}(t,z,\epsilon)&=\sum_{l=1}^{\mathbf{D}}\sum_{m_1+m_2=m}\frac{m!}{m_1!m_2!}(\partial_\epsilon^{m_1}\epsilon^{\boldsymbol{\Delta}_l})t^{\mathbf{d}_l}\partial_t^{\boldsymbol{\delta}_l}\mathbf{R}_l(\partial_z)\partial_\epsilon^{m_2}y_{p',p}(t,z,\epsilon)\\
&+\sum_{m_1+m_2+m_3=m}\frac{m!}{m_1!m_2!m_3!}(\partial_\epsilon^{m_1}\mathbf{c}_{0,0}(\epsilon))(\partial_\epsilon^{m_2}\mathbf{c}_{0}(z,\epsilon))
\mathbf{R}_{0}(\partial_z)\partial_{\epsilon}^{m_3}y_{p',p}(t,z,\epsilon)\\
&+\sum_{m_1+m_2=m}\frac{m!}{m_1!m_2!}(\partial_{\epsilon}^{m_1}\mathbf{c}_{\mathbf{F}}(\epsilon))\partial_{\epsilon}^{m_2}u_p(t,z,\epsilon),
\end{align*}
for all $m\ge 0$, $(t,z,\epsilon)\in\tilde{\mathcal{T}}\times H_{\beta'}\times\mathcal{E}_{p',p}$. We write 
$$\mathbf{c}_{\mathbf{F}}(\epsilon)=\sum_{m\ge0}\mathbf{c}_{\mathbf{F},m}\frac{\epsilon^m}{m!}, \quad \mathbf{c}_{0}(z,\epsilon)=\sum_{m\ge0}\mathbf{c}_{0,m}(z)\frac{\epsilon^m}{m!},\quad \mathbf{c}_{0,0}(\epsilon)=\sum_{m\ge0}\mathbf{c}_{0,0,m}\frac{\epsilon^m}{m!}.$$ Letting $\epsilon\to 0$ in the last equality and making use of (\ref{e2354}) and (\ref{limit_deriv_order_m_of_up_epsilon}) we get a recursion formula for the coefficients $H_m(t,z)$:
\begin{align*}
\mathbf{Q}(\partial_z)\partial_{t}H_m(t,z)&=\sum_{l=1}^{\mathbf{D}}\sum_{m_1+m_2=m-\boldsymbol{\Delta}_l}\frac{m!}{m_1!m_2!}t^{\mathbf{d}_l}\partial_t^{\boldsymbol{\delta}_l}\mathbf{R}_l(\partial_z)H_{m_2}(t,z)\\
&+\sum_{m_1+m_2+m_3=m}\frac{m!}{m_1!m_2!m_3!}\mathbf{c}_{0,0,m_1}\mathbf{c}_{0,m_2}(z)\mathbf{R}_{0}(\partial_z)H_{m_3}(t,z)\\
&+\sum_{m_1+m_2=m}\frac{m!}{m_1!m_2!}\mathbf{c}_{\mathbf{F},m_1}h_{m_2}(t,z).
\end{align*}
An analogous argument as in the last part of the proof of Theorem~\ref{teo2} yields the conclusion.
\end{proof}

We consider the main problem under study in the work. Namely, we consider the equation
\begin{equation}\label{epral}
P_1(t,z,\epsilon,\sigma_{q^{\delta}},\partial_t,\partial_z)c_{\mathbf{F}}(\epsilon)^{-1}P_2\left(t,z,\epsilon,\partial_t,\partial_z\right)y_{p',p}(t,z,\epsilon)=f(t,z,\epsilon),
\end{equation}
for every $0\le p\le \varsigma-1$ and all $0\le p'\le \chi_p$.

\begin{theo}\label{teopral}
For every $0\le p\le \varsigma-1$ and $0\le p'\le \chi_p-1$ the problem (\ref{epral}) admits an analytic solution $y_{p',p}(t,z,\epsilon)\in\mathcal{O}(\tilde{\mathcal{T}}\times H_{\beta'}\times\mathcal{E}_{p',p})$, for every $0<\beta'<\beta$, with $y_{p',p}(0,z,\epsilon)\equiv 0$. In addition to this, the problem (\ref{epral}) admits a unique formal power series solution $\hat{y}(t,z,\epsilon)\in\mathbb{F}[[\epsilon]]$ which satisfies that
$$\hat{y}(t,z,\epsilon)=a(t,z,\epsilon)+y_{1}(t,z,\epsilon)+y_{2}(t,z,\epsilon),$$
and for every $0\le p\le \varsigma-1$ and $0\le p'\le \chi_p-1$ one has that 
$$y_{p',p}(t,z,\epsilon)=a(t,z,\epsilon)+y^{1}_{p',p}(t,z,\epsilon)+y^{2}_{p',p}(t,z,\epsilon),$$
where $\epsilon\mapsto y^{1}_{p',p}(t,z,\epsilon)$ is a $\mathbb{F}$-valued function that admits $\hat{y}_{1}(t,z,\epsilon)$ as its Gevrey asymptotic expansion of order $1/k$ in $\mathcal{E}_{p',p}$, and $\epsilon\mapsto y^{2}_{p',p}(t,z,\epsilon)$ is a $\mathbb{F}$-valued function that admits $\hat{y}_{2}(t,z,\epsilon)$ as its $q-$Gevrey asymptotic expansion of order $1/\kappa$, for every fixed $0<\kappa<\frac{k}{\delta}\min_{l=1}^{D-1}d_{l}$ in $\mathcal{E}_{p',p}$.
\end{theo}

\end{document}